\documentclass[12pt, a4paper, reqno]{amsart}
\usepackage[utf8]{inputenc}
\title{Generalization of formal monad theory to lax functors}
\author{Kengo Hirata}
\thanks{I am deeply grateful to my supervisor, Prof. M. Hasegawa, for his guidance and support throughout my MSc research. My sincere thanks also go to Dr. S. Fujii, Dr. Y. Maehara, and my colleagues for the fruitful discussions.}

\address{Research Institute for Mathematical Sciences, Kyoto University \\ Kyoto 606-8502, Japan}
\email{khirata@kurims.kyoto-u.ac.jp, k.hirata@sms.ed.ac.uk}

\usepackage{comment}
\usepackage[
  backend=biber,
  style=alphabetic,
  backref=true,
  maxbibnames = 99,
  firstinits=true,
]{biblatex}
\usepackage{geometry}
\usepackage{amsthm,amsfonts}
\usepackage{amssymb}
\usepackage{mathtools}
\usepackage{mathrsfs}
\usepackage{mathbbol} 
\usepackage{extarrows}
\usepackage{stmaryrd}
\usepackage{esint}
\usepackage{stackengine,scalerel}
\usepackage{thmtools}
\usepackage{bbm}
\usepackage{subfiles}
\usepackage{ifdraft}
\usepackage{musicography}
\usepackage{multirow}
\usepackage{booktabs, makecell}

\usepackage{ebproof}
\usepackage{float}

\usepackage[table]{xcolor}
\definecolor{darkblue}{rgb}{0.2,0,0.6}
\usepackage[%
  colorlinks=true,%
  allcolors=darkblue,%
]{hyperref}
\usepackage{cleveref}
\usepackage[textsize=tiny,backgroundcolor=white]{todonotes}

\numberwithin{equation}{subsection}
\crefname{equation}{}{}
\crefname{section}{\S\!}{\S\!}
\Crefname{section}{\S\!}{\S\!}

\declaretheoremstyle[
  spaceabove=6pt, spacebelow=6pt,
  headfont=\normalfont\itshape,
  notefont=\mdseries, notebraces={(}{)},
  bodyfont=\normalfont,
  postheadspace=0.5em,
  qed=$\square$
]{myproofstyle}

\declaretheorem[style=plain,numberwithin=subsection,name=Theorem]{theorem}
\declaretheorem[style=plain,sibling=theorem,name=Lemma]{lemma}
\declaretheorem[style=plain,sibling=theorem,name=Proposition]{proposition}
\declaretheorem[style=plain,sibling=theorem,name=Corollary]{corollary}

\newtheorem*{schema}{Theorem-Schema}

\declaretheorem[style=definition,qed=$\blacksquare$,sibling=theorem,name=Definition]{definition}
\declaretheorem[style=definition,qed=$\blacksquare$,sibling=theorem,name=Example]{example}

\declaretheorem[style=definition,qed=$\blacksquare$,sibling=theorem,name=Remark]{remark}

\usepackage{tikz}
\usetikzlibrary{arrows.meta,decorations.pathreplacing,decorations.markings,shapes,calc}
\usetikzlibrary{matrix,arrows,decorations.pathmorphing,backgrounds,decorations.markings,positioning}
\usetikzlibrary{cd}

\tikzset{1cell/.style={->, labelsize, auto}}
\tikzset{1cell equality/.style={double,double equal sign distance, labelsize, auto}}
\tikzset{1cell squig/.style={%
  1cell,%
  decorate,%
  decoration={%
    zigzag,%
    segment length=9.25\pgflinewidth,%
    amplitude=1.9\pgflinewidth,%
    post=lineto, post length=6\pgflinewidth,%
    pre=lineto, pre length=6\pgflinewidth,%
  },%
}}

\tikzset{2cell/.style={-implies,double,double equal sign distance,shorten 
>=2pt, shorten <=3pt}}
\tikzset{2cellshort/.style={-implies,double,double equal sign distance,shorten 
>=4pt, shorten <=5pt}}
\tikzset{2cellr/.style={implies-,double,double equal sign distance,shorten 
>=3pt, shorten <=2pt}}

\tikzset{3cell/.style={-implies,double,double distance=2.5pt,shorten >=2pt, 
shorten <=3pt}}
\tikzset{labelsize/.style={font=\scriptsize}}
\tikzset{string/.style={very thick}}
\tikzset{
  pto/.style={->,postaction={decorate},
    decoration={
        markings,
        mark=at position 0.5 with {\arrow{|}}}
  },
}

\newcommand{\twocell}[5][]{\draw[2cell] (#2)++(#4) to node[auto,labelsize,#1]{#3} ++(#5);}

\newcommand{\twocelll}[3][]{\twocell[#1]{#2}{#3}{0.25,0}{-0.5,0}}
\newcommand{\twocellr}[3][]{\twocell[#1]{#2}{#3}{-0.25,0}{0.5,0}}
\newcommand{\twocelld}[3][]{\twocell[#1]{#2}{#3}{0,0.25}{0,-0.5}}
\newcommand{\twocellu}[3][]{\twocell[#1]{#2}{#3}{0,-0.25}{0,0.5}}

\newcommand{\twocellur}[3][]{\twocell[#1]{#2}{#3}{-0.175,-0.175}{0.35,0.35}}

\newcommand{\twocelldr}[3][]{\twocell[#1]{#2}{#3}{-0.175,0.175}{0.35,-0.35}}
\newcommand{\twocelldl}[3][]{\twocell[#1]{#2}{#3}{0.175,0.175}{-0.35,-0.35}}

\newcommand{\pushout}[1]{%
  \draw (#1) to ++(1ex,0ex);%
  \draw (#1) to ++(0ex,-1ex);%
}

\newcommand{\id}{{\mathrm{id}}}
\newcommand{\Id}{{\mathrm{Id}}}
\newcommand{\op}{{\mathrm{op}}}
\newcommand{\co}{{\mathrm{co}}}
\newcommand{\ob}{{\mathrm{ob}}}

\newcommand{\dom}{{\mathrm{dom}}}
\newcommand{\cod}{{\mathrm{cod}}}

\newcommand{\N}{\mathbb{N}}

\newcommand{\cat}[1]{\mathcal{#1}}

\newcommand{\iso}{\cong}
\newcommand{\eqv}{\simeq}

\newcommand{\Cat}{\mathbf{Cat}}
\newcommand{\CAT}{\mathbf{CAT}}
\newcommand{\Ab}{\mathbf{Ab}}
\newcommand{\Set}{\mathbf{Set}}
\newcommand{\TwoCat}{\mathbf{2\text{-}Cat}}
\newcommand{\TwoCAT}{\mathbf{2\text{-}CAT}}
\newcommand{\MonCat}{\mathbf{MonCat}_l}
\newcommand{\Adj}{\mathbf{Adj}}
\newcommand{\Mnd}{\mathbf{Mnd}}
\newcommand{\mnd}{\mathrm{mnd}}
\newcommand{\Dist}{\mathbf{Dist}}
\newcommand{\A}{\cat{A}}
\newcommand{\barA}{\overline{\cat{A}}}
\newcommand{\B}{\cat{B}}
\newcommand{\barB}{\overline{\cat{B}}}
\newcommand{\C}{\cat{C}}
\newcommand{\D}{\cat{D}}
\newcommand{\E}{\cat{E}}
\newcommand{\barC}{\overline{\cat{C}}}
\newcommand{\barD}{\overline{\cat{D}}}
\newcommand{\K}{\cat{K}}
\newcommand{\M}{\cat{M}}
\newcommand{\X}{\cat{X}}
\newcommand{\Y}{\cat{Y}}
\newcommand{\One}{\mathbf{1}}
\newcommand{\Two}{\mathbf{2}}
\newcommand{\Eye}{{\mathbf{2}_\mathbf{2}}}
\newcommand{\one}{\mathbbm{1}}

\newcommand{\cG}{\mathscr{G}}
\newcommand{\cmp}{\mathbf{Cmp}}
\newcommand{\quiv}{\mathbf{Quiv}}

\newcommand{\Q}{\mathcal{Q}}
\newcommand{\QQ}{\mathscr{Q}}
\newcommand{\R}{\mathcal{R}}
\newcommand{\I}{\mathcal{I}}

\newcommand{\TT}[1]{{\left[#1\right]}}
\newcommand{\TP}[1]{{\prescript{2}{P}{}{\left[#1\right]}}}
\newcommand{\TL}[1]{{\prescript{2}{L}{}{\left[#1\right]}}}
\newcommand{\TO}[1]{{\prescript{2}{Op}{}{\left[#1\right]}}}
\newcommand{\Tw}[2]{{\prescript{2}{#1}{}{\left[#2\right]}}}

\newcommand{\PP}[1]{{\prescript{P}{P}{}{\left[#1\right]}}}
\newcommand{\PL}[1]{{\prescript{P}{L}{}{\left[#1\right]}}}
\newcommand{\PO}[1]{{\prescript{P}{Op}{}{\left[#1\right]}}}

\newcommand{\LL}[1]{{\prescript{L}{L}{}{\left[#1\right]}}}
\newcommand{\LO}[1]{{\prescript{L}{Op}{}{\left[#1\right]}}}
\newcommand{\Lw}[2]{{\prescript{L}{#1}{}{\left[#2\right]}}}

\newcommand{\OL}[1]{{\prescript{Op}{L}{}{\left[#1\right]}}}
\newcommand{\OO}[1]{{\prescript{Op}{Op}{}{\left[#1\right]}}}
\newcommand{\Ow}[2]{{\prescript{Op}{#1}{}{\left[#2\right]}}}

\newcommand{\ww}[3]{{\prescript{#1}{#2}{}{\left[#3\right]}}}

\newcommand{\TwoCatTI}{{\mathbf{2\text{-}Cat}_\mathrm{Icon}}}

\newcommand{\TwoCatLI}{{\mathbf{2\text{-}Cat}^\mathrm{Lax}_\mathrm{Icon}}}
\newcommand{\BiCatLI}{{\mathbf{BiCat}^\mathrm{Lax}_\mathrm{Icon}}}
\newcommand{\BiCatLZ}{{\mathbf{BiCat}^\mathrm{Lax}_0}}

\newcommand{\walg}[3]{{{#1}\text{-}{#3}\text{-}\mathrm{Alg}_{#2}}}
\newcommand{\alg}[2]{{{#2}\text{-}\mathrm{Alg}_{#1}}}
\newcommand{\palg}[2]{{\mathrm{Ps}\text{-}{#2}\text{-}\mathrm{Alg}_{#1}}}
\newcommand{\lalg}[2]{{\mathrm{Lax}\text{-}{#2}\text{-}\mathrm{Alg}_{#1}}}
\newcommand{\calg}[2]{{\mathrm{OpLax}\text{-}{#2}\text{-}\mathrm{Alg}_{#1}}}

\newcommand{\lladj}[1]{{{#1}\text{-}\mathrm{LDAdj}}}

\newcommand{\alA}{\mathbb{A}}
\newcommand{\alB}{\mathbb{B}}
\newcommand{\alX}{\mathbb{X}}
\newcommand{\alY}{\mathbb{Y}}
\newcommand{\g}{\mathbb{g}}
\newcommand{\morl}{\mathbb{l}}
\newcommand{\f}{\mathbb{f}}

\newcommand{\p}{\mathbb{p}}
\newcommand{\q}{\mathbb{q}}

\newsavebox{\ia}
\sbox{\ia}{%
  \begin{tikzpicture}[y=0.5cm,baseline={(0,0)}]
    \coordinate(l) at (0,0);
    \coordinate(r) at (0.5,0);
    \node at (0.25,0.3) [font=\tiny]{$\bot$};
    \draw[1cell, transform canvas={yshift=8pt}] (l) to (r);
    \draw[1cell, transform canvas={yshift=0.5pt}] (r) to (l);
  \end{tikzpicture}\!
}
\newcommand{\inlineadj}{\mathrel{\usebox{\ia}}}

\NewDocumentCommand{\lend}{e{_^}}{
  \newintop{\hspace{-0.09em}\leftarrow}[#1][#2]%
}%
\NewDocumentCommand{\psend}{e{_^}}{
  \newintop{\sim}[#1][#2]%
}%
\NewDocumentCommand{\newintop}{moo}{%
  \ThisStyle{\ensurestackMath{%
  \stackengine{0pt}{%
    \SavedStyle\int%
      \IfValueT{#2}{_{#2}}%
      \IfValueT{#3}{^{#3}}%
  }{%
    \SavedStyle#1%
  }{O}{l}{F}{F}{L}}}}%

\begin{document}

\begin{abstract}
  We study lax functors between bicategories as a generalized concept of monads
  and describe generalized notions and theorems of formal monad theory for lax functors.
  Our first approach is to use the 2-monad whose lax algebras are lax functors.
  We define lax doctrinal adjunctions for a 2-monad $T$ on a 2-category $\mathcal{K}$,
  and we show that if $\mathcal{K}$ admits and $T$ preserves certain codescent objects,
  the 2-category $\mathrm{Lax}\text{-}{T}\text{-}\mathrm{Alg}_{c}$ of lax algebras and colax morphisms can coreflectively be embedded
  in the 2-category of lax doctrinal adjunctions.
  This coreflective embedding generalizes the relation between monads and adjunctions.
  Our second approach is to see a distributive law for monads as a 2-functor from a lax Gray tensor product,
  and we show a generalized form of Beck's characterization of distributive laws.
\end{abstract}

\maketitle
\tableofcontents
\section{Introduction}\label{sec:intro}
A lax functor $F\colon\C\rightsquigarrow\D$ between bicategories
is a 2-categorically weaker notion of functors,
which does not need to preserve compositions strictly as $Fg.Ff = F(gf)$ and $1_{FA} = F1_A$,
but it only requires the existence of comparison 2-cells
$Fg.Ff \Rightarrow F(gf)$ and $1_{FA} \Rightarrow F1_A$
together with some coherence conditions.

Many interesting structures arise as lax functors, often with very simple domains or codomains.
The simplest and most important case is when the domain $\C$ is the terminal 2-category $\One$.
In this case, a lax functor $\One\rightsquigarrow\D$ is a monad in the bicategory $\D$.
Another case is when the domain is restricted to an indiscrete 2-category $X^\mathrm{ind}$
whose set of objects is $X$.
This time, a lax functor $X^\mathrm{ind} \rightsquigarrow \B$ is a $\B$-enriched category
whose set of the objects is $X$.

When we restrict the domain and codomain to single object bicategories,
a lax functor is merely a lax monoidal functor between monoidal categories.
For example, the forgetful functor $\Ab\rightarrow\Set$ has a canonical lax monoidal structure.
Other examples include graded rings, i.e.,
to give a graded ring $R \cong \bigoplus_{n}R_n$ is to give a lax monoidal functor
$\N^{\mathrm{dis}}\rightarrow\Ab; n \mapsto R_n$ where $\N^{\mathrm{dis}}$
is a discrete category whose objects are natural numbers with the monoidal structure defined by addition.

If $\D$ is the single object full sub-2-category of $\Cat$ with the object $\mathbb{C}$,
which is just a monoidal category of endo-functors $\mathrm{End}(\mathbb{C})$,
a lax functor $\A\rightsquigarrow\mathrm{End}(\mathbb{C})$ from a 2-category $\A$
is a 2-category-graded monad,
which includes all the notions ordinary monads, graded monads~\cite{Fujii_2019_grad_index_monads,FujiiKM16}, and category graded monads
\cite{OrchWadlEade_2020_unif_grad_param_mnd}.

In this paper, we view a lax functor as a generalized monad,
and we generalize notions and theorems of monads to lax functors.
Our results will be presented in \Cref{sec:adj_lax_classifier,sec:Gray_dist}.
In \Cref{sec:adj_lax_classifier}, we study the generalization of the relation of monads and adjunctions,
including Eilenberg-Moore/Kleisli objects.
This result includes a generalization of work by Street~\cite{Street_1972_two_const_lax_func}: his work was about a lax functor $\A \to \Cat$ where $\A$ is 1-category, which we generalize to a lax functor $\A \rightsquigarrow \K$ between 2-categories.
In \Cref{sec:Gray_dist}, we study the generalization of distributive laws of monads.

In detail, we study the lax algebras of a 2-monad $T$ on a 2-category $\K$ in \Cref{sec:adj_lax_classifier}.
{\color{red}%
    The content of this section turned out to be almost identical to Section 8 to Lack's paper ``Morita Contexts as Lax Functors'' in 2014
}\cite{Lack_2014_morita}.
A reader can also find similar contents with clearer proofs in \cite{Miloslav} as well.
The key player in the section is an adjunction $X\inlineadj A$ in the base 2-category $\K$
whose codomain $A$ of the left adjoint has a structure of strict algebra $\mathbb{A}$,
which will be called a \emph{lax doctrinal adjunction}.
The name lax doctrinal adjunction is taken after the \emph{doctrinal adjunction} introduced by Kelly in his paper~\cite{Kelly_1974_doctrinal_adjunction}.
Recall that a doctrinal adjunction consists of the following data: an adjunction $A\inlineadj B$ in the base 2-category $\K$
in which both objects $A$ and $B$ are equipped with structures of strict algebras $\mathbb{A}$ and $\mathbb{B}$
together with the structure of lax morphism on the right adjoint,
which is equivalent to give a colax morphism structure on the left adjoint.
Although a structure of strict algebra is given only to one side $A$
in our lax doctrinal adjunction $X\inlineadj A$,
a lax algebra structure $\mathbb{X}$ on $X$ can be induced along the left adjoint,
and we can show that the right and left adjoint has a canonical structure of lax and colax morphisms.

The main result in \Cref{sec:adj_lax_classifier} is the coreflective embedding theorem (\Cref{thm:coreflective_embedding_thm}),
which proves that if $\K$ admits and $T$ preserves certain codescent objects,
$T$ satisfies the coherence condition for lax algebras.
That is, $\lalg{c}{T}$ is a coreflective sub-2-category of lax doctrinal adjunctions.
Since a lax algebra and a lax doctrinal adjunction of the identity 2-monad is
just a monad and an adjunction in $\K$,
this coreflective embedding is the generalization of the correspondence between monads and Kleisli adjunctions.
When we consider the 2-monad $T^\mathbf{Fun}$ on $\TT{\ob(\A), \K}$ whose strict algebra is a 2-functor
$\A\rightarrow\K$ and whose lax algebra is a lax functor $\A\rightsquigarrow\K$,
we obtain the correspondence between lax functors and pointwise adjunction between its Kleisli 2-functor in \cite{Street_1972_two_const_lax_func}.

In \Cref{sec:Gray_dist}, we study a generalization of Beck's distributive laws \cite{Beck_1969_dist_law}.
Let $\K$ be a 2-category and $\C$ be a bicategory.
We denote the 2-category of lax functors, lax natural transformations, and modification by $\LL{\C,\K}$.
If $\A$ is a 2-category, let $\TL{\A,\K}$ be the full sub-2-category of $\LL{\A,\K}$ whose objects are 2-functors.
In \cite{Nikolic_2019_strict_tensor_prod}, Nikoli\'{c} observed that a tensor product 2-category
$\C\boxtimes\D$ of bicategories $\C$ and $\D$ can be defined to have the universal property of
$\LL{\B,\LL{\A,\K}} \cong \TL{\A\boxtimes\B,\K}$,
and he gave two direct descriptions of the construction of $\C\boxtimes\D$.
Since a distributive law between monads is known to be equivalent to
a monad in the 2-category $\mnd_l(\K)$ of monads in $\K$,
Nikoli\'{c}'s tensor product includes the construction of
the free 2-category $\One\boxtimes\One$ of a distributive law
satisfying $\TL{\One\boxtimes\One, \K} \cong \mnd_l(\mnd_l(\K))$ as a specific case.

Although Nikoli\'{c} gave two direct constructions of $\C\boxtimes\D$,
we use the 2-step description $\barC\otimes_l\barD$ using the lax functor classifier $\overline{\A}$
and the lax Gray tensor product $\otimes_l$ as remarked in Section 4 in Nikoli\'{c}'s paper.
Here, a lax functor classifier 2-category $\barC$ of a bicategory $\C$ is universal for
of $\LL{\C,\K} \iso \TL{\barC, \K}$,
and a lax Gray tensor product $\A\otimes_l\B$ of 2-categories $\A$ and $\B$ is universal for
of $\TL{\B,\TL{\A,\K}} \iso \TL{\A\otimes_l \B, \K}$.

We prove Beck's characterization of distributive laws in a generalized form
with Nikoli\'{c}'s generalized distributive law,
which is our main contribution in \Cref{sec:Gray_dist}.
In particular, we show that every lax functor classifier $\barC$ has a canonical
comonoid structure in the monoidal category $(\TwoCat_0,\otimes_l,\One)$,
which induces a composition of generalized monads.

\section{2-categories preliminaries}\label{sec:2_caty_pre}
\subsection{Terminology and preliminaries}

When we say 2-category, 2-functor, or 2-natural transformation, we mean a strict one:
$\Cat$-enriched category, $\Cat$-functor, or $\Cat$-natural transformation.
On the other hand, when we say bicategory, pseudo functor, or pseudo natural transformation,
we shall mean the weaker up-to-isomorphism notion.
Furthermore, when we say lax as a prefix, we shall assume that there is only a comparison map
with coherence conditions.

Besides the 2-/pseudo/(op) lax natural transformations,
we also consider another type of transformation called \emph{icon}s \cite{Lack_2007_icons},
named for Identity Component Oplax Natural transformations.
Here is the table of the names of the 2-categories to which we will refer.

\begin{table}[H]
  \centering
  \rowcolors{2}{gray!25}{white}
  \begin{tabular}{c!{\color{white}\vline width 1pt}lll}
    \rowcolor{gray!50}
    2-category          & object      & 1-cell           & 2-cell        \\
    $\TwoCat$           & 2-category  & 2-functor        & 2-nat. trans. \\
    $\TwoCatTI$         & 2-category  & 2-functor        & icon          \\
    $\TwoCatLI$         & 2-category  & lax functor      & icon          \\
    $\BiCatLI$          & bicategory  & lax functor      & icon          \\
    $\ww{w}{w'}{\A,\K}$ & $w$-functor & $w'$-nat. trans. & modification
  \end{tabular}
\end{table}

At the bottom of the table above, we assume that $w$ and $w'$ are elements in the set $\{s,p,l,c\}$,
where each letter in $\{s,p,l,c\}$ represents strict, pseudo, lax, and colax/oplax.
We use these subscripts $w\in\{s,p,l,c\}$ throughout this paper.
The dual subscript $\overline{w}$ is defined as $\overline{l} = c, \overline{c} = l,$ and $\overline{s} = s,\overline{p} = p$.
When we draw a diagram, weak morphisms are sometimes described with zigzag arrows ``$\rightsquigarrow$''
to distinguish them from strict ones.

More concretely, the functor 2-categories are denoted as in the table below.

\begin{table}[H]
  \centering
  \rowcolors{2}{gray!25}{white}
  \begin{tabular}{c!{\color{white}\vline width 1pt}cccc}
    \rowcolor{gray!50}
           &
    \multicolumn{4}{c}{functor}                                        \\
    \arrayrulecolor{gray!50} \specialrule{5pt}{0pt}{-8pt}
    \arrayrulecolor{white} \cmidrule[1pt]{2-5}
    \arrayrulecolor{block}
    \rowcolor{gray!50}
    \multirow{-2}{*}{\shortstack{\vspace{1em}
    \\ natural
    \\ transformation }}
           & 2-           & pseudo       & lax          & oplax        \\
    2-     & $\TT{\A,\K}$ &              &              &              \\
    pseudo & $\TP{\A,\K}$ & $\PP{\A,\K}$ &              &              \\
    lax    & $\TL{\A,\K}$ & $\PL{\A,\K}$ & $\LL{\A,\K}$ & $\OL{\A,\K}$ \\
    oplax  & $\TO{\A,\K}$ & $\PO{\A,\K}$ & $\LO{\A,\K}$ & $\OO{\A,\K}$
  \end{tabular}
\end{table}

We omitted the upper right part since we usually do not consider stricter transformations between weaker functors.
For a bicategory $\C$, we use the same notation for seven functor categories at the right-hand side of the table.
The duality of those functor categories can be stated in the following way.

\begin{proposition}
  Let $\C$ and $\D$ be bicategories.
  Then, there are the following isomorphism of bicategories for each $w,w'\in\{p,l,c\}$,
  and if $\C$ and $\D$ are both 2-categories, the same for $w,w'\in\{s, p,l,c\}$.
  \[
    \ww{w}{w'}{\C, \D}
    \ \iso\ {\ww{w}{\overline{w}'}{\C^\op,\D^\op}}^\op
    \ \iso\ {\ww{\overline{w}}{\overline{w}'}{\C^\co,\D^\co}}^\co
  \]
\end{proposition}

The following asserts the functoriality of $\ww{w}{w'}{-,-}$.

\begin{proposition}
  Let $\C$ be a bicategory and $\K$ be a 2-category.
  \begin{enumerate}
    \item $\ww{w}{w'}{\C,-}\colon \TwoCat \rightarrow \TwoCAT$ is a 2-functor.
    \item $\ww{w}{w'}{-,\K}\colon \TwoCat_0^\op \rightarrow \TwoCAT_0$ is an ordinary functor.
  \end{enumerate}
\end{proposition}
\begin{proof}
  Both were proved 
  in \cite{Gray_1974_formal_caty_theor} for the case $\C$ is a 2-category and $w = s$.
  A similar proof works for general $w$.
  One can also reduce those general cases to the case Gray proved 
  with the $w$-functor classifier in \Cref{sec:adj_lax_classifier}.
\end{proof}

\begin{remark}\label{rem:ww_is_not_2-functorial_in_first_variable}
  The functor $\ww{w}{w'}{-,\K}$ does not have a canonical extension to a 2-functor from $\TwoCat$ except for the case $w=w'=s$.
  Let us first examine the worst case when $w\in\{l,c\}$.
  If we wish to post-compose a lax functor $F$ to a 2-natural transformation $\alpha\colon K\Rightarrow H$,
  it ends up with uncomposable pair of 2-cell $FHf.F\alpha_A \Rightarrow F(Hf.\alpha_A)$
  and $F\alpha_B.FKf \Rightarrow F(\alpha_B.Kf) = F(Hf.\alpha_A)$ for each $f\colon A\rightarrow B$.
  Therefore, we do not know where to map $\alpha$ by $\ww{l}{w'}{-, \K}$.
  \[
    \begin{tikzpicture}
      \node(GKA) at (0,2) {$FKA$};
      \node(GKA') at (0,0) {$FKB$};
      \node(GHA) at (4,2) {$FHB$};
      \node(GHA') at (4,0) {$FHB$};
      \draw[1cell,swap] (GKA) to node {$FKf$} (GKA');
      \draw[1cell] (GHA) to node {$GHf$} (GHA');
      \draw[1cell] (GKA) to node {$F\alpha_A$} (GHA);
      \draw[1cell,swap] (GKA') to node {$F\alpha_{B}$} (GHA');
      \draw[1cell,bend right=15,swap,sloped] (GKA) to node [auto=false,fill=white] {$F(\alpha_{B}.K\!f)$} (GHA');
      \draw[1cell,bend left=15, sloped] (GKA) to node[auto=false,fill=white] {$F(H\!f.\alpha_A)$} (GHA');
      \draw[1cell equality] ($(2,1)+(0.1,0.1)$) to ($(2,1)-(0.1,0.1)$);
      \twocelldl[]{3.1,1.5}{$F^2$}
      \twocellur[swap]{0.7,0.6}{$F^2$}
    \end{tikzpicture}
  \]

  The better but unsuccessful case is when $w=p$.
  We can compose the 2-cells above, so we obtain a pseudo natural transformation.
  But since there is no canonical bicategory with pseudo natural transformations as 2-cells,
  and it is unavoidable to think of tricategories.

  The most subtle case is $w=s$ and $w'\in\{p,l,c\}$.
  Let $\alpha\colon F\Rightarrow G$ be a 2-natural transformation,
  and $\beta\colon K\Rightarrow H$ be a $w'$-natural transformation.
  Then, to define a 2-natural transformation $\Tw{w'}{\alpha,\K}\colon\Tw{w'}{F,\K}\Rightarrow\Tw{w'}{G,\K}$,
  We need to show the following commutes.
  \[
    \begin{tikzpicture}[
        nat trans/.style={1cell,-implies,double,double equal sign distance}
      ]
      \node(ul) at (0,0) {$KF$};
      \node(ur) at (2,0) {$KG$};
      \node(dl) at (0,-1.2) {$HF$};
      \node(dr) at (2,-1.2) {$HG$};
      \draw[nat trans] (ul) to node{$K\alpha$} (ur);
      \draw[nat trans] (dl) to node {$H\alpha$} (dr);
      \draw[nat trans,swap] (ul) to node {$\beta F$} (dl);
      \draw[nat trans] (ur) to node{$\beta G$} (dr);
    \end{tikzpicture}
  \]
  There is a comparison modification $\beta_{\alpha_{(-)}} \colon \beta G.K\alpha \Rrightarrow H\alpha.\beta F$,
  but there is no need for $\beta GA.K\alpha_A$ and $H\alpha_A.\beta FA$ to coincide.

  Also, $\ww{w}{w'}{-,\K}$ does not have any canonical extension to an endo-2-functor on $\TwoCatTI$ even for the case $w=w'=s$.
  To see this, let $F,G\colon \A\rightarrow\B$ be 2-functors.
  Then, if we assume an icon $F^*\Rightarrow G^* \colon \ww{w}{w'}{\B,\K}\rightarrow \ww{w}{w'}{\A,\K}$ exists,
  every $w$-functor $H\colon \B\rightsquigarrow \K$ satisfies $HF = HG$,
  which is unnecessary for an icon $F\Rightarrow G$ to exist.
\end{remark}

Let us denote 
the arrow 2-category as $\Two$ and the free 2-category on a single 2-cell as $\Eye$.
We define 2-functors $\dom,\cod\colon \Tw{w}{\Two,\K}\rightarrow \K$
to as evaluation 2-functor at the domain and codomain of the unique non-identity 1-cell for each.

\begin{lemma}\label{lem:w_nat_is_functor}
  Let $\A$ and $\K$ be 2-categories.
  Then, a $w$-functor $H\colon\A\rightsquigarrow \Tw{\overline{w}'}{\Two,\K}$
  is precisely a $w'$-natural transformation
  between $w$-functors $F\rightsquigarrow G$ where
  $F = \dom H$ and $G = \cod H$.

  Similarly, a $w$-functor $\A\rightsquigarrow\Tw{\overline{w}'}{\Eye,\K}$ is precisely a modification between $w'$-natural transformations.
\end{lemma}
\begin{proof}
  We show this for $w=l$ and $w'=l$.
  A lax functor $H\colon\A\rightsquigarrow \TO{\Two,\K}$ consists of the following data.
  \begin{itemize}
    \item Each object $A\in\A$ is mapped to a 1-cell $FA\xrightarrow{\alpha_A}GA$.
    \item Each 1-cell $f\colon A\rightarrow B$ is mapped to a triple $(Ff,Gf,\alpha_f)$.
          \[
            \begin{tikzpicture}
              \node(ul) at (0,0) {$FA$};
              \node(ur) at (2,0) {$FB$};
              \node(dl) at (0,-1.2) {$GA$};
              \node(dr) at (2,-1.2) {$GB$};
              \draw[1cell] (ul) to node{$Ff$} (ur);
              \draw[1cell,swap] (dl) to node {$Gf$} (dr);
              \draw[1cell,swap] (ul) to node {$\alpha_A$} (dl);
              \draw[1cell] (ur) to node{$\alpha_B$} (dr);
              \twocellur{1.1,-0.7}{$\alpha_f$}
            \end{tikzpicture}
          \]
    \item Each 2-cell $\alpha\colon f\Rightarrow g$ is mapped to a pair of 2-cells $F\alpha\colon Ff\Rightarrow Fg$ and $G\alpha\colon Gf\Rightarrow Gg$
          such that are compatible with $\alpha_f$.
    \item For each object $A$, a pair of 2-cells $F^0\colon 1_{FA} \Rightarrow F1_A$ and
          $G^0\colon 1_{GA} \Rightarrow G1_A$.
    \item For each composable pair $f$ and $g$, a pair of 2-cells $F^2\colon Fg.Ff \Rightarrow F(gf)$ and
          $G^2\colon Gg.Gf \Rightarrow G(gf)$.
    \item The unit and associativity law for $F^0$ and $F^2$ holds. The same for $G^0$ and $G^2$.
    \item $\alpha_f$ is compatible with the horizontal composition of 1-cells in $\A$.
    \item $F\alpha$ and $G\alpha$ are compatible with the horizontal and vertical compositions of 2-cells in $\A$.
  \end{itemize}
  These are the same as the data that make $F$ and $G$ lax functors, and $\alpha$ a lax natural transformation.

  Similarly, the correspondence for modifications can be proved by writing down the data.
\end{proof}

For each 2-category $\A$, we denote the underlying category as $\A_0$.
For example, $\TwoCat_0$ is an ordinary category of 2-categories and 2-functors.
When we regard $\M$ as a single-object bicategory,
a monoidal category $\M$ will be denoted by $\Sigma\M$.

\subsection{Computad}
Throughout this paper, we frequently say as ``generate a 2-category from some data, and then take the quotient by some equalities''.
It might be worth being precise with these statements,
so we would like to recall \emph{computads} introduced in \cite{Street_1976_computad}
to justify these ``data,'' ``generate,'' and ``take the quotient''.

\begin{definition}
  A \emph{computad} $\mathscr{G}$ consists of a pair $\left(|\cG|, {\{\cG(A,B)\}}_{A,B\in{|\cG|}_0}\right)$,
  where $|\cG|$ is a quiver, and for each of two objects $A,B$ of $|\cG|$,
  there is another quiver $\cG(A,B)$ whose objects are of paths in $|\cG|$
  which starts with $A$ and ends with $B$.
  We call an edge of $\cG(A,B)$ a 2-cell, which will be described as follows.
  
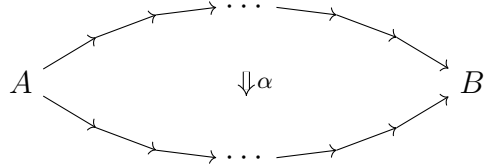
\[
    \begin{tikzpicture}
      \node(a) at (-3,0) {$A$};
      \node(b) at (3,0) {$B$};
      \node at (0,1) {$\cdots$};
      \node at (0,-1) {$\cdots$};
      \draw[1cell] (a) to (-2,0.6);
      \draw[1cell] (-2,0.6) to (-1.2,0.9);
      \draw[1cell] (-1.2,0.9) to (-0.4,1);
      \draw[1cell] (0.4,1) to (1.2,0.9) ;
      \draw[1cell] (1.2,0.9) to (2,0.6);
      \draw[1cell] (2,0.6) to (b);
      \draw[1cell] (a) to (-2,-0.6);
      \draw[1cell] (-2,-0.6) to (-1.2,-0.9);
      \draw[1cell] (-1.2,-0.9) to (-0.4,-1);
      \draw[1cell] (0.4,-1) to (1.2,-0.9) ;
      \draw[1cell] (1.2,-0.9) to (2,-0.6);
      \draw[1cell] (2,-0.6) to (b);
      \twocelld{0,0}{$\alpha$}
    \end{tikzpicture}
  \]

  A homomorphism $T\colon \cG\rightarrow \mathscr{H}$ of computads is a homomorphism $|T|\colon|\cG|\rightarrow|\mathscr{H}|$ of quivers
  together with a family of homomorphisms $T_{AB}\colon\cG(A,B)\rightarrow\mathscr{H}(TA,TB)$
  such that an object $p=(f_0,f_1,\dots,f_n)$ in $\cG(A,B)$ is mapped
  to the object $Tp = (Tf_0,Tf_1,\dots,Tf_n)$ in $\mathscr{H}(TA,TB)$.

  We denote the category of computads by $\cmp$.
\end{definition}

There is an obvious forgetful functor $\mathcal{U}\colon\TwoCat_0 \rightarrow \cmp$.
Conversely, given a computad $\cG$, we can generate a free 2-category $\mathcal{F}\cG$,
whose objects are those in $|\cG|$ and whose 1-cells are the paths in $|\cG|$.
For the detailed construction of 2-cells in $\mathcal{F}\cG$, look for Kelly's original paper \cite{Street_1976_computad}.
This construction $\mathcal{F}$ of free 2-categories gives a left adjoint to the forgetful functor $\mathcal{U}$,
and one can show that $\mathcal{U}\colon\TwoCat_0 \rightarrow \cmp$ is monadic.

From the monadicity, every 2-category can be presented as a coequalizer of free 2-categories generated from computads.
It is also justified to define a 2-category by taking the quotient of a free 2-category
generated from a computad by some equalities of 1-cells and 2-cells.

\begin{example}[$\Mnd$, $\Adj$]
  A \emph{monad} in a 2-category consists of an endo-1-cell $t\colon A\rightarrow A$ and two 2-cells $\eta\colon 1\Rightarrow t$, $\mu\colon tt\Rightarrow t$ satisfying
  $\mu*\mu t = \mu*t\mu$, $\mu*t\eta = 1$, and $\mu*\eta t = 1$.
  Since a monad is defined only with equations,
  there is a 2-category $\Mnd$ such that a monad in $\K$ corresponds to the
  image of a 2-functor $\Mnd\rightarrow\K$,
  where $\Mnd$ is the quotient of the following computad
  
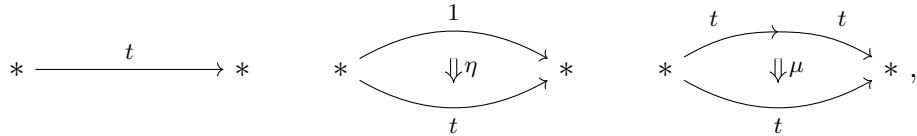
\[
    \begin{tikzpicture}[baseline=(s.base)]
      \node(s) at (0,0) {$*$};
      \node(t) at (3,0) {$*$};
      \draw[1cell] (s) to node {$t$} (t);
    \end{tikzpicture}
    \qquad
    \begin{tikzpicture}[baseline=(s.base)]
      \node(s) at (0,0) {$*$};
      \node(t) at (3,0) {$*$};
      \draw[1cell,bend left] (s) to node {$1$} (t);
      \draw[1cell,bend right,swap] (s) to node{$t$} (t);
      \twocelld{1.5,0}{$\eta$}
    \end{tikzpicture}
    \qquad
    \begin{tikzpicture}[baseline=(s.base)]
      \node(s) at (0,0) {$*$};
      \node(t) at (3,0) {$*$};
      \draw[1cell,bend left=15] (s) to node {$t$} (1.5,0.5);
      \draw[1cell,bend left=15] (1.5,0.5) to node {$t$} (t);
      \draw[1cell,bend right,swap] (s) to node{$t$} (t);
      \twocelld{1.5,0}{$\mu$}
    \end{tikzpicture},
  \]
  by the three equations.

  An \emph{adjunction} $f \dashv g$ in a 2-category $\K$ consists of
  1-cells $f\colon A\rightarrow B$, $g\colon B\rightarrow A$,
  and 2-cells $\eta\colon 1_A \Rightarrow gf$, $\varepsilon\colon fg\Rightarrow 1_B$
  satisfying the triangular identity.
  In the same way as monads, there is a 2-category
  $\Adj$ of two objects $*_+$ and $*_-$,
  which is the quotient of the computad
  \[
    \begin{tikzpicture}[baseline=(s.base)]
      \node(s) at (0,0) {$*_+$};
      \node(t) at (3,0) {$*_-$};
      \draw[1cell,transform canvas={yshift =-5pt}] (t) to node {$g$} (s);
      \draw[1cell,transform canvas={yshift = 5pt}] (s) to node {$f$} (t);
    \end{tikzpicture}
    \qquad
    \begin{tikzpicture}[baseline=(s.base)]
      \node(s) at (0,0) {$*_+$};
      \node(t) at (3,0) {$*_+$};
      \draw[1cell,bend right=13,swap] (s) to node {$f$} (1.5,-0.5);
      \draw[1cell,bend right=13,swap] (1.5,-0.5) to node {$g$} (t);
      \draw[1cell,bend left] (s) to node{$1$} (t);
      \twocelld{1.5,0}{$\eta$}
    \end{tikzpicture}
    \qquad
    \begin{tikzpicture}[baseline=(s.base)]
      \node(s) at (0,0) {$*_-$};
      \node(t) at (3,0) {$*_-$};
      \draw[1cell,bend left=13] (s) to node {$g$} (1.5,0.5);
      \draw[1cell,bend left=13] (1.5,0.5) to node {$f$} (t);
      \draw[1cell,bend right,swap] (s) to node{$1$} (t);
      \twocelld{1.5,0}{$\varepsilon$}
    \end{tikzpicture},
  \]
  by the equations precisely making $f$ left adjoint to $g$.
  An adjunction in a 2-category $\K$ corresponds to a 2-functor $\Adj\rightarrow\K$.

  In $\Adj$, all endo-morphisms of $*_{+}$ have a form of $gfgf\cdots gf$,
  and $gf$ constitutes a monad together with $g\varepsilon f$ and $\eta$.
  It is easy to check that $\Mnd$ is a full sub-2-category of $\Adj$ with a unique object $*_+$.
  These 2-categories and the inclusion $\iota\colon\Mnd\hookrightarrow\Adj$ are studied in \cite{Schanuel_1986_the_free_adj}.
\end{example}

\section{Formal monad theory}\label{sec:formal_mnd_theor}
In this section, we review formal monad theory introduced in  \cite{Street_1972_formal_theor_monad},
which formalizes the theory of monads in terms of general 2-category theory.

In \Cref{subsec:adj_mnd_EM}, we review monads and adjunctions in a 2-category.
We also recall Eilenberg-Moore objects and some characterizations of them.
Most statements in \Cref{subsec:adj_mnd_EM} will be proved in a more generalized form with lax algebras of a 2-monad $T$
in \Cref{sec:adj_lax_classifier}.

In \Cref{subsec:dist_law}, we recall Beck's distributive laws of monads \cite{Beck_1969_dist_law}.
We review that distributive laws allow us to compose
two monads, and also that distributive laws can be viewed as liftings of monads to Eilenberg-Moore objects.
Those distributive laws can be generalized to lax functors,
which will be shown in \Cref{sec:Gray_dist}.

\subsection{Adjunction, monad, Eilenberg-Moore object}\label{subsec:adj_mnd_EM}

Let $(t\colon A\rightarrow A,\mu^t,\eta^t)$ and $(s\colon B\rightarrow B,\mu^s,\eta^s)$ be monads in $\K$.
A \emph{lax morphism} of monads $(f,\bar{f})\colon t\rightarrow s$ consists of a 1-cell $f\colon A\rightarrow B$
and a 2-cell $\bar{f}\colon sf\Rightarrow ft$, satisfying the following.
\[
  \begin{tikzpicture}
    \node(ul) at (-2,1) {$ssf$};
    \node(um) at ( 0,1) {$sft$};
    \node(ur) at ( 2,1) {$ftt$};
    \node(dl) at (-2,0) {$sf$};
    \node(dr) at ( 2,0) {$ft$};
    \draw[1cell,swap] (ul) to node {$\mu^s f$} (dl);
    \draw[1cell] (ur) to node {$f\mu^t$} (dr);
    \draw[1cell] (ul) to node {$s\bar{f}$} (um);
    \draw[1cell] (um) to node {$\bar{f}t$} (ur);
    \draw[1cell] (dl) to node {$\bar{f}$} (dr);
  \end{tikzpicture}
  \qquad
  \begin{tikzpicture}
    \node(ul) at (-2,0) {$f$};
    \node(ur) at (-2,-1) {$sf$};
    \node(dr) at (0,-1) {$ft$};
    \node(w) at (0,0) {$f$};
    \draw[1cell,swap] (ul) to node{$\eta^s f$} (ur);
    \draw[1cell] (ur) to node {$\bar{f}$} (dr);
    \draw[1cell] (w) to node {$f\eta^t$} (dr);
    \draw[1cell equality] (ul) to (w);
  \end{tikzpicture}
\]
A \emph{2-cell} between two lax monad morphisms $(f,\bar{f})$ and $(g,\bar{g})$ is a 2-cell $\alpha\colon f\Rightarrow g$ in $\K$
satisfying the following.
\[
  \begin{tikzpicture}
    \node(ul) at (-2,0) {$sf$};
    \node(ur) at (0,0) {$sg$};
    \node(dl) at (-2,-1) {$ft$};
    \node(dr) at (0,-1) {$gt$};
    \draw[1cell] (ul) to node{$s\alpha$} (ur);
    \draw[1cell] (dl) to node{$\alpha t$} (dr);
    \draw[1cell] (ur) to node {$\bar{g}$} (dr);
    \draw[1cell,swap] (ul) to node {$\bar{f}$} (dl);
  \end{tikzpicture}
\]
We write the 2-category of monads and lax monad morphisms as $\mathrm{mnd}_l(K)$.

There is also the dual notion of morphism of monads,
called \emph{colax morphism} $(f,\tilde{f})$,
which has an opposite direction of the 2-cell
$\tilde{f}\colon ft\Rightarrow sf$.
\[
  \begin{tikzpicture}
    \node(ur) at ( 2,1) {$ssf$};
    \node(um) at ( 0,1) {$sft$};
    \node(ul) at (-2,1) {$ftt$};
    \node(dr) at ( 2,0) {$sf$};
    \node(dl) at (-2,0) {$ft$};
    \draw[1cell] (ur) to node {$\mu^s f$} (dr);
    \draw[1cell,swap] (ul) to node {$f\mu^t$} (dl);
    \draw[1cell] (um) to node {$s\bar{f}$} (ur);
    \draw[1cell] (ul) to node {$\bar{f}t$} (um);
    \draw[1cell] (dl) to node {$\bar{f}$} (dr);
  \end{tikzpicture}
  \qquad
  \begin{tikzpicture}
    \node(ul) at (0,0) {$f$};
    \node(ur) at (0,-1) {$sf$};
    \node(dr) at (-2,-1) {$ft$};
    \node(w) at (-2,0) {$f$};
    \draw[1cell] (ul) to node{$\eta^s f$} (ur);
    \draw[1cell] (dr) to node {$\bar{f}$} (ur);
    \draw[1cell,swap] (w) to node {$f\eta^t$} (dr);
    \draw[1cell equality] (ul) to (w);
  \end{tikzpicture}
\]
A 2-cell between colax morphisms is defined similarly to a 2-cell between lax morphisms.
The 2-category of monads with colax monad morphisms is denoted by $\mathrm{mnd}_c(K)$.

We define a category $\Delta_a$ called the \emph{augmented simplex category} as following:
\begin{itemize}
  \item $\ob(\Delta_a) = \N$. We regard $n$ as a set $\{0,1,\dots,n-1\}$.
  \item A morphism $f\colon n\rightarrow m$ is an order preserving map.
\end{itemize}
This $\Delta_a$ has a structure of strict monoidal category defined by $n\otimes m = n + m$ with the monoidal unit $0$.
In $\Delta_a$, all objects are of form $n = 1^{\otimes n}$,
and all morphisms are generated from $0\rightarrow 1$, $2\rightarrow 1$, and $1\rightarrow 1$
by composition $\circ$ and tensor product $\otimes$.
So, a strict monoidal functor $\Delta_a\rightarrow \M$ is only determined by the diagram $I\xrightarrow{e} A \xleftarrow{m} A\otimes A$,
which is the image of $0\rightarrow 1 \leftarrow 2$.
Furthermore, one can show that a diagram
$I\xrightarrow{e} A \xleftarrow{m} A\otimes A$
is an image of a strict monoidal functor $\Delta_a\rightarrow \M$ precisely when $(A,e,m)$ is a monoid.
Therefore, a 2-functor $\Sigma\Delta_a\rightarrow\K$ is a monad in $\K$ and $\Mnd \iso \Sigma\Delta_a$.

We can easily prove the following.

\begin{proposition}\label{prop:caty_of_mnd}
  The 2-category $\mnd_l(\K)$ of monad and lax morphism in $\K$ can be presented in the following two other ways.
  \[
    \mnd_l(\K) \iso \LL{\One, \K} \iso \TL{\Mnd, \K}.
  \]
  And for colax morphisms of monads, we also have
  \[
    \mnd_c(\K) \iso \LO{\One, \K} \iso \TO{\Mnd, \K}.
  \]
\end{proposition}

We identify all three; a monad $t$ in $\K$, a lax functor $\One\rightsquigarrow\K$, and a 2-functor $\Mnd\rightarrow\K$.
The following are three examples of lax monad morphisms.

\begin{example}\label{eg:monad-morphism-1_Cat-smashing}
  We define an ordinary monad $T$ on $\Cat$ by sending a category $\mathbb{C}$ to the following pushout in $\Cat$.
  \[
    \begin{tikzpicture}
      \node(ul) at (0,1.5) {$\ob(\mathbb{C})$};
      \node(ur) at (2,1.5) {$\mathbb{C}$};
      \node(dl) at (0,0) {$\one$};
      \node(dr) at (2,0) {$T\mathbb{C}$};
      \draw[1cell] (ul) to node{$\iota$}(ur);
      \draw[1cell] (dl) to (dr);
      \draw[1cell,swap] (ul) to node{$!$}(dl);
      \draw[1cell] (ur) to (dr);
      \pushout{1.5,0.5}
    \end{tikzpicture}
  \]
  Here, $\iota$ is the identity-on-objects inclusion functor,
  and $T\mathbb{C}$ is a single-object category
  whose morphisms are equivalence classes of sequences of morphisms in $\mathbb{C}$.

  Let $S$ be the free monoid monad on $\Set$
  and $\mathrm{Mor}\colon \Cat\rightarrow \Set$ be the functor sending a category $\mathbb{C}$ to its set of morphisms.
  Then, there is a natural transformation $\alpha\colon S\circ\mathrm{Mor}\Rightarrow \mathrm{Mor}\circ T$ whose component
  $\alpha_\mathbb{C}\colon S(\mathrm{Mor}(\mathbb{C})) \rightarrow \mathrm{Mor}(T\mathbb{C})$
  is a function that sends a sequence of morphisms to its equivalence class,
  and $(\mathrm{Mor}, \alpha)$ defines a lax monad morphism in $\CAT$.
\end{example}
\begin{example}\label{eg:monad-morphism-2_ring}
  Let $T$ be the free Abelian group monad
  and $S$ be the free monoid monad on $\Set$.
  Then, there is a natural transformation $\alpha\colon ST\Rightarrow TS$
  whose component $\alpha_X\colon S(T(X)) \Rightarrow T(S(X))$ is a function that maps
  a sequence $(\sum_{i = 1}^{n_1}a_{1i}x_{1i},\dots,\sum_{i=1}^{n_m}a_{mi}x_{mi})$ of formal sums in $X$ to its multiplication
  $\sum_{(i_1,\dots,i_m)=(1,\dots,1)}^{(n_1,\dots,n_m)} a_{1i_i}\cdots a_{mi_m} (x_{1i_1},\dots,x_{mi_m})$.
  This $(T,\alpha)$ defines a lax monad morphism from $S$ to $S$ in $\Cat$.
\end{example}
\begin{example}\label{eg:monad-morphism-3_Ab-CMon-in-MonCat}
  Let $\Id_\Ab$ be the identity monad on $\Ab$ and $S$ be the free commutative monoid monad on $\Set$.
  These monads have canonical lax monoidal structures, so these are monads in $\MonCat$.
  Then, there is a monoidal natural transformation $\alpha\colon SU\Rightarrow U\Id_\Ab$,
  where $U$ is the forgetful functor $\Ab\Rightarrow\Set$,
  and the function $\alpha_G\colon S(UG) \rightarrow UG$ is defined by the multiplication of finite elements in $G$.
  This $(U,\alpha)$ is a lax monad morphism in $\MonCat$.
\end{example}

In ordinary category theory, if there is an adjunction $F\dashv G$,
we always have a monad $GF$.
Conversely, if there is a monad $T$ on a category $\mathbb{A}$,
there are both \emph{Eilenberg-Moore category} $\mathbb{A}^T$
and \emph{Kleisli category} $\mathbb{A}_T$ to form adjunctions between $\mathbb{A}$.

Just as in ordinary category theory, we can also have a monad from an adjunction
in arbitral 2-category $\K$.

\begin{proposition}
  For each adjunction $(f,g,\eta,\varepsilon)$, there is an induced monad $(gf,g\varepsilon f, \eta)$.
\end{proposition}

The converse construction does not always work in a general 2-category.
That is because we do not know how to construct ``the Eilenberg-Moore category'' or
``the Kleisli category,'' which should be an object in $\K$.
However, we can still define Eilenberg-Moore and Kleisli objects
from their universality, which might not always exist \cite{Street_1972_formal_theor_monad}.

\begin{definition}
  An \emph{Eilenberg-Moore object} $A^t$ for a formal monad $t$ on $A$ in a 2-category $\K$ is
  the representing object of the functor sending $X$ to the Eilenberg-Moore category of
  the monad $\K(X,t)$ on $\K(X,A)$.
  That is, an object $A^t$ in $\K$ with a natural isomorphism $\K\left(X,A^t\right) \iso {\K(X,A)}^{\K(X,t)}$.

  A \emph{Kleisli} object is an Eilenberg-Moore object in $\K^\op$.
\end{definition}

An Eilenberg-Moore object in $\Cat$ is an Eilenberg-Moore category.
Hence, in $\Cat$, the category $\Ab$ is the Eilenberg-Moore object of the free Abelian group monad on $\Set$.
This $\Ab$ is also the Eilenberg-Moore object in the 2-category $\MonCat$
of monoidal categories, lax monoidal functors, and monoidal natural transformations,
where $\MonCat$ is known not always to have Eilenberg-Moore objects.

We call an object in ${\K(X,A)}^{\K(X,t)}$ a \emph{module} that is a pair $(a\colon X\rightarrow A, \alpha\colon ta\Rightarrow a)$,
satisfying the following equality.
The Eilenberg-Moore object can also be presented as a terminal module $(u^t \colon A^t \rightarrow A, \upsilon\colon tu^t\rightarrow t)$ with one- and two-dimensional universality.

\[
  \begin{tikzpicture}[baseline={(0,1)}]
    \node (x) at (-2,2) {$X$};
    \node (u) at (0,2) {$A$};
    \node (d) at (0,0) {$A$};
    \draw[1cell] (x) to node{$a$} (u);
    \draw[1cell,swap] (x) to node{$a$} (d);
    \draw[1cell,bend right=20,swap] (u) to node{$t$} (d);
    \draw[1cell,bend left=20] (u) to node{$1_A$} (d);
    \twocelld{-0.8,1.4}{$\alpha$};
    \twocelll{0,1}{$\eta$};
  \end{tikzpicture}
  =
  \begin{tikzpicture}[baseline={(0,1)}]
    \node (x) at (-2,2) {$X$};
    \node (u) at (0,2) {$A$};
    \node (d) at (0,0) {$A$};
    \draw[1cell] (x) to node{$a$} (u);
    \draw[1cell,swap] (x) to node{$a$} (d);
    \draw[1cell,bend left=20] (u) to node{$1_A$} (d);
    \draw[1cell equality] ($(-0.6,1.3)+(0,0.1)$) to ($(-0.6,1.3)+(0,-0.1)$);
  \end{tikzpicture}
  ,
  \begin{tikzpicture}[baseline={(0,0)}]
    \node (x) at (-2,1) {$X$};
    \node (u) at (0,1) {$A$};
    \node (m) at (0.5,0) {$A$};
    \node (d) at (0,-1) {$A$};
    \draw[1cell] (x) to node{$a$} (u);
    \draw[1cell,swap] (x) to node{$a$} (m);
    \draw[1cell,swap] (x) to node{$a$} (d);
    \draw[1cell,bend left=15] (u) to node{$t$} (m);
    \draw[1cell,bend left=15] (m) to node{$t$} (d);
    \twocelld{-0.5,0.65}{$\alpha$};
    \twocelld{-0.5,-0.1}{$\alpha$};
  \end{tikzpicture}
  =
  \begin{tikzpicture}[baseline={(0,0)}]
    \node (x) at (-2,1) {$X$};
    \node (u) at (0,1) {$A$};
    \node (m) at (0.5,0) {$A$};
    \node (d) at (0,-1) {$A$};
    \draw[1cell] (x) to node{$a$} (u);
    \draw[1cell,swap] (x) to node{$a$} (d);
    \draw[1cell,bend left=15] (u) to node{$t$} (m);
    \draw[1cell,bend left=15] (m) to node{$t$} (d);
    \draw[1cell,bend right=27,swap] (u) to node{$t$} (d);
    \twocelld{-1,0.6}{$\alpha$};
    \twocelll{0,0}{$\mu$};
  \end{tikzpicture}.
\]

If $(t,\mu,\eta)$ is a monad, $(t\colon A\rightarrow A, \mu\colon tt\Rightarrow t)$ defines a module.
So, if an EM-object $A^t$ exists in $\K$, we obtain a 1-cell $f^t\colon A\rightarrow A^t$ from the universality.
And it also turns out that $f^t$ and $u^t$ are adjunct,
where the unit $\eta$ and counit $\varepsilon$ satisfy
$t = u^t.f^t$, $\upsilon = u^t.\varepsilon$, and $\mu = \upsilon.f^t$.
As in the ordinary category theory,
this adjunction $f^t\dashv u^t$ is also the terminal one
whose associated monad is $(t,\mu,\eta)$.
That is because if there is an adjunction $l\dashv r\colon B\rightarrow A$ in $\K$ that associates $t$,
there is a unique \emph{Eilenberg-Moore comparison map} $k\colon B\rightarrow A^t$ satisfying $r = u^tk$ and $kl = f^t$
since $r\colon B\rightarrow A$ and $r\varepsilon \colon tr = rlr\Rightarrow r$ define a module of $t$.

The following theorem also gives other characterizations of Eilenberg-Moore objects.

\begin{theorem}\label{thm:EM_characterization}
  Let $t$ be a monad in a 2-category $\K$.
  All the following characterize the Eilenberg-Moore object of $t$.
  \begin{enumerate}
    \item The conical lax limit of the lax functor $t\colon \One \rightsquigarrow \K$.
    \item The conical lax limit of the 2-functor $t\colon \Mnd\rightarrow \K$.
    \item \cite{Schanuel_1986_the_free_adj} The $\Cat$-enriched pointwise right Kan extension of $t\colon\Mnd\rightarrow \K$ along $\iota\colon\Mnd\rightarrow\Adj$.
          \[
            \begin{tikzpicture}
              \node(m) at (0,0) {$\Mnd$};
              \node(a) at (0,-1.3) {$\Adj$};
              \node(k) at (2,0) {$\K$};
              \draw[1cell,right hook->] (m) to node[swap]{$\iota$} (a);
              \draw[1cell] (m) to node{$t$} (k);
              \draw[1cell] (a) to node[swap]{$f^t\dashv u^t\colon A^t\rightarrow A$} (k);
              \twocellu{0.65,-0.45}{}
            \end{tikzpicture}
          \]
  \end{enumerate}
\end{theorem}

By identifying an object $A$ with the identity monad $1_A$,
there is an inclusion 2-functor $\Id^\K_{(-)}\colon\K\rightarrow\mnd_l(\K)$.
The left adjoint $U$ always exists for this $\Id^\K_{(-)}$,
which is the 2-functor taking out the underlying object of a monad.
Even more interesting is the right adjoint of $\Id^\K_{(-)}$,
which exists if and only if $\K$ admits all Eilenberg-Moore objects.

\begin{theorem}\label{thm:fact_EM-is-right-adjunction}
  If $\K$ has all Eilenberg-Moore objects,
  there are adjunctions

  \[
    \begin{tikzpicture}[baseline={(0,-1)}]
      \node(k) at (-2,0) {$\K$};
      \node(m) at (4,0) {$\mnd_l(\K)\iso\TL{\Mnd,\K}.$};
      \draw[1cell,right hook->] (k) to node[auto=false,fill=white,near end]{$\Id^\K_{(-)}$} (m);
      \draw[1cell,transform canvas={yshift=12pt},swap] (m) to node{$U$} (k);
      \draw[1cell,transform canvas={yshift=-12pt}] (m) to node{$\mathbf{EM}$} (k);
      \node[labelsize] at (0,0.2) {$\bot$};
      \node[labelsize] at (0,-0.2) {$\bot$};
    \end{tikzpicture}
  \]
\end{theorem}

Using Eilenberg-Moore objects, we can express a lax monad morphism in another way: a lift $\widetilde{f}\colon A^t\rightarrow B^s$ of a 1-cell $f\colon A\rightarrow B$ satisfying $u^s\widetilde{f} = fu^t$.
With this correspondence, $\mnd_l(\K)$ is a full sub-2-category of the arrow category.

\begin{theorem}\label{thm:fact_mnd_l_is_refl_subcaty}
  We denote the full sub-2-category of $\TT{\Two,\K}$
  consisting of right adjoints as $\mathrm{radj}(\K)$.

  Let $\mathrm{radj}(\K)$ be the full sub-2-category of $\TT{\Two,\K}$ whose objects are right adjoints.
  If $\K$ has all Eilenberg-Moore objects,
  sending a monad $t$ to the forgetful 1-cell $A^t \rightarrow A$ presents $\mnd_l(\K)$ as a reflective sub-2-category of $\mathrm{radj}(\K)$.
  \[
    \begin{tikzpicture}
      \node(m) at (-2,0) {$\mnd_l(\K)$};
      \node(r) at (2,0) {$\mathrm{radj}(\K)$};
      \node at (0,0.2) {$\bot$};
      \draw[1cell,right hook->,transform canvas={yshift=-4pt},swap] (m) to node{$EM$} (r);
      \draw[1cell,bend right=18] (r) to node{$ $} (m);
    \end{tikzpicture}
  \]
\end{theorem}

\begin{example}
  The lax monad morphisms in \Cref{eg:monad-morphism-1_Cat-smashing,eg:monad-morphism-2_ring,eg:monad-morphism-3_Ab-CMon-in-MonCat}
  correspond to the following liftings of 1-cells (in $\CAT$, $\CAT$, and $\MonCat$ for each).
  \[
    \begin{tikzpicture}
      \node(ul) at (0,1.5) {$\mathbf{Mon}$};
      \node(ur) at (2,1.5) {$\mathbf{Mon}$};
      \node(dl) at (0,0)   {$\Cat$};
      \node(dr) at (2,0)   {$\Set$};
      \draw[1cell equality] (ul) to (ur);
      \draw[1cell,swap] (dl) to node{$\mathrm{Mor}$} (dr);
      \draw[1cell,right hook->] (ul) to (dl);
      \draw[1cell] (ur) to node{$U$} (dr);
    \end{tikzpicture}
    \qquad
    \begin{tikzpicture}
      \node(ul) at (0,1.5) {$\mathbf{Mon}$};
      \node(ur) at (2,1.5) {$\mathbf{Mon}$};
      \node(dl) at (0,0)   {$\Set$};
      \node(dr) at (2,0)   {$\Set$};
      \draw[1cell]      (ul) to node{$\bar{T}$} (ur);
      \draw[1cell,swap] (dl) to node{$T$} (dr);
      \draw[1cell,swap] (ul) to node{$U$}(dl);
      \draw[1cell]      (ur) to node{$U$} (dr);
    \end{tikzpicture}
    \qquad
    \begin{tikzpicture}
      \node(ul) at (0,1.5) {$\Ab$};
      \node(ur) at (2,1.5) {$\mathbf{CMon}$};
      \node(dl) at (0,0)   {$\Ab$};
      \node(dr) at (2,0)   {$\Set$};
      \draw[1cell,right hook->] (ul) to (ur);
      \draw[1cell,swap] (dl) to node{$U$} (dr);
      \draw[1cell equality] (ul) to (dl);
      \draw[1cell] (ur) to node{$U$} (dr);
    \end{tikzpicture}
  \]
  In the second example, $T$ and $\bar{T}$ are also monads, and so defines a distributive law which we will recall in the next section.
  In the third example, the monad on the left side is an identity monad.
  Therefore, the lax monad morphism arose from the pseudo monoidal functor $\Ab\hookrightarrow\mathbf{CMon}$
  by the adjunction in \Cref{thm:fact_EM-is-right-adjunction}.
\end{example}

\subsection{Distributive law}\label{subsec:dist_law}

Given two monads $T$ and $S$ on the same category $\mathbb{C}$,
we cannot always compose them to obtain another monad $TS$
since we do not know how to define their multiplication $TSTS \Rightarrow TS$.
For example, it is known that
the powerset monad $\mathcal{P}$ and the distribution monad $\mathcal{D}$ on $\Set$ has no canonical composition
\cite{VaraWhins_2006_dist_p_d}.

A similar argument is that, given two monadic functors
$U_1\colon\mathbb{A}\rightarrow\mathbb{B}$ and $U_2\colon\mathbb{B}\rightarrow\mathbb{C}$,
the composed functor $U_2U_1$ is not always monadic.
For example, the composition $\Cat\rightarrow\quiv\rightarrow\Set$ of forgetful functors
from $\Cat$ through the category of quivers to $\Set$ is not monadic,
although each step is monadic.

In 1969, Beck defined a necessary and sufficient data called \emph{a distributive law} for a pair of monads $T$ and $S$ to compose
and give a new monad $TS$ compatible with $T$ and $S$ in a certain sense \cite{Beck_1969_dist_law}.
Beck also showed that those data are equivalent to a lift $\overline{T}$ of a monad $T$
to the Eilenberg-Moore category $\mathbb{A}^S$.
In this situation, the composition of the monadic functors
${\left(\mathbb{A}^S\right)}^{\overline{T}} \rightarrow \mathbb{A}^S \rightarrow \mathbb{A}$
is again monadic.
We review these results.

\begin{definition}
  Let $(s,\mu^s,\eta^s)$ and $(t,\mu^t,\eta^t)$ be two monads on the same object $A$.
  Then, a \emph{distributive law} is a 2-cell $\gamma \colon st \Rightarrow ts$ satisfying the axioms below.
  \[
    \begin{tikzpicture}[baseline={(0,0.25)}]
      \node(ul) at (-1, 1) {$A$};
      \node(ur) at ( 1, 1) {$A$};
      \node(dl) at (-1,-0.5) {$A$};
      \node(dr) at ( 1,-0.5) {$A$};
      \draw[1cell,bend left] (ul) to node{$1$} (ur);
      \draw[1cell,swap,bend right] (ul) to node{$t$} (ur);
      \draw[1cell,swap,bend right] (dl) to node{$t$} (dr);
      \draw[1cell,swap] (ul) to node{$s$} (dl);
      \draw[1cell] (ur) to node{$s$} (dr);
      \twocelldl{0,-0.1}{$\gamma$}
      \twocelld{0,1}{$\eta^t$}
    \end{tikzpicture}
    \ =\
    \begin{tikzpicture}[baseline={(0,0.25)}]
      \node(ul) at (-1, 1) {$A$};
      \node(ur) at ( 1, 1) {$A$};
      \node(dl) at (-1,-0.5) {$A$};
      \node(dr) at ( 1,-0.5) {$A$};
      \draw[1cell,bend left] (ul) to node{$1$} (ur);
      \draw[1cell,bend left] (dl) to node{$1$} (dr);
      \draw[1cell,swap,bend right] (dl) to node{$t$} (dr);
      \draw[1cell,swap] (ul) to node{$s$} (dl);
      \draw[1cell] (ur) to node{$s$} (dr);
      \twocelld{0,-0.5}{$\eta^t$}
      \draw[double,double equal sign distance, shorten >=2pt, shorten <=3pt]
      (0,0.6)++(0.15,0.15) to ++(-0.3,-0.3);
    \end{tikzpicture}
    \qquad
    \begin{tikzpicture}[baseline={(0,0)}]
      \node(ul) at (-1, 1) {$A$};
      \node(ur) at ( 1, 1) {$A$};
      \node(dl) at (-1,-1) {$A$};
      \node(dr) at ( 1,-1) {$A$};
      \draw[1cell] (ul) to node{$t$} (ur);
      \draw[1cell,swap] (dl) to node{$t$} (dr);
      \draw[1cell,swap,bend right] (ul) to node{$s$} (dl);
      \draw[1cell,bend left] (ur) to node{$1$} (dr);
      \draw[1cell,swap,bend right] (ur) to node{$s$} (dr);
      \twocelldl{-0.3,0}{$\gamma$}
      \twocelll[swap]{1,0}{$\eta^s$}
    \end{tikzpicture}
    \ =\
    \begin{tikzpicture}[baseline={(0,0)}]
      \node(ul) at (-1, 1) {$A$};
      \node(ur) at ( 1, 1) {$A$};
      \node(dl) at (-1,-1) {$A$};
      \node(dr) at ( 1,-1) {$A$};
      \draw[1cell] (ul) to node{$t$} (ur);
      \draw[1cell,swap] (dl) to node{$t$} (dr);
      \draw[1cell,swap,bend right] (ul) to node{$s$} (dl);
      \draw[1cell,bend left] (ul) to node{$1$} (dl);
      \draw[1cell, bend left] (ur) to node{$1$} (dr);
      \twocelll[swap]{-1,0}{$\eta^s$}
      \draw[double,double equal sign distance, shorten >=2pt, shorten <=3pt]
      (0.4,0)++(0.15,0.15) to ++(-0.3,-0.3);
    \end{tikzpicture}
  \]
  \[
    \begin{tikzpicture}[baseline={(0,0.25)},x=9mm]
      \node(ul) at (-1.5, 1) {$A$};
      \node(ur) at ( 1.5, 1) {$A$};
      \node(um) at ( 0, 1.5) {$A$};
      \node(dl) at (-1.5,-0.5) {$A$};
      \node(dr) at ( 1.5,-0.5) {$A$};
      \draw[1cell,bend left=12] (ul) to node{$t$} (um);
      \draw[1cell,bend left=12] (um) to node{$t$} (ur);
      \draw[1cell,swap,bend right] (ul) to node{$t$} (ur);
      \draw[1cell,swap,bend right] (dl) to node{$t$} (dr);
      \draw[1cell,swap] (ul) to node{$s$} (dl);
      \draw[1cell] (ur) to node{$s$} (dr);
      \twocelld[swap]{0,1}{$\mu^t$}
      \twocelldl[swap]{0,-0.4}{$\gamma$}
    \end{tikzpicture}
    \ =\
    \begin{tikzpicture}[baseline={(0,0.25)},x=9mm]
      \node(ul) at (-1.5, 1) {$A$};
      \node(ur) at ( 1.5, 1) {$A$};
      \node(um) at ( 0, 1.5) {$A$};
      \node(dm) at ( 0,0) {$A$};
      \node(dl) at (-1.5,-0.5) {$A$};
      \node(dr) at ( 1.5,-0.5) {$A$};
      \draw[1cell,bend left=12] (ul) to node{$t$} (um);
      \draw[1cell,bend left=12] (um) to node{$t$} (ur);
      \draw[1cell,bend left=12] (dl) to node{$t$} (dm);
      \draw[1cell,bend left=12] (dm) to node{$t$} (dr);
      \draw[1cell,swap,bend right] (dl) to node{$t$} (dr);
      \draw[1cell,swap] (ul) to node{$s$} (dl);
      \draw[1cell] (ur) to node{$s$} (dr);
      \draw[1cell] (um) to node[auto=false,fill=white]{$s$} (dm);
      \twocelld[swap]{0,-0.5}{$\mu^t$}
      \twocelldl[swap]{-0.7,0.6}{$\gamma$}
      \twocelldl[swap]{ 0.8,0.6}{$\gamma$}
    \end{tikzpicture}
    \qquad
    \begin{tikzpicture}[baseline={(0,0)},x=9mm]
      \node(ul) at (-1, 1.25) {$A$};
      \node(ur) at ( 1, 1.25) {$A$};
      \node(mr) at ( 1.5, 0) {$A$};
      \node(dl) at (-1,-1.25) {$A$};
      \node(dr) at ( 1,-1.25) {$A$};
      \draw[1cell] (ul) to node{$t$} (ur);
      \draw[1cell,swap] (dl) to node{$t$} (dr);
      \draw[1cell,swap,bend right] (ul) to node{$s$} (dl);
      \draw[1cell,bend left=12] (ur) to node{$s$} (mr);
      \draw[1cell,bend left=12] (mr) to node{$s$} (dr);
      \draw[1cell,swap,bend right] (ur) to node{$s$} (dr);
      \twocelldl[swap]{-0.6,0}{$\gamma$}
      \twocelll[swap]{1,0}{$\mu^s$}
    \end{tikzpicture}
    \ =\
    \begin{tikzpicture}[baseline={(0,0)},x=9mm]
      \node(ul) at (-1, 1.25) {$A$};
      \node(ur) at ( 1, 1.25) {$A$};
      \node(ml) at (-0.5, 0) {$A$};
      \node(mr) at ( 1.5, 0) {$A$};
      \node(dl) at (-1,-1.25) {$A$};
      \node(dr) at ( 1,-1.25) {$A$};
      \draw[1cell] (ul) to node{$t$} (ur);
      \draw[1cell,swap] (dl) to node{$t$} (dr);
      \draw[1cell,swap,bend right] (ul) to node{$s$} (dl);
      \draw[1cell,bend left=12] (ul) to node{$s$} (ml);
      \draw[1cell,bend left=12] (ml) to node{$s$} (dl);
      \draw[1cell,bend left=12] (ur) to node{$s$} (mr);
      \draw[1cell,bend left=12] (mr) to node{$s$} (dr);
      \draw[1cell] (ml) to node[auto=false,fill=white]{$t$} (mr);
      \twocelll[swap]{-1,0}{$\mu^s$}
      \twocelldl[swap]{0.5,0.65}{$\gamma$}
      \twocelldl[swap]{0.5,-0.65}{$\gamma$}
    \end{tikzpicture}
  \]
\end{definition}

\begin{theorem}\label{thm:fact_dist_law_characterization}\cite{Beck_1969_dist_law}
  Let $(t,\mu^t, \eta^t)$ and $(s,\mu^s,\eta^s)$ be monads on $A\in\K$,
  and $\gamma\colon st\Rightarrow ts$ be a 2-cell. Then, the followings are equivalent.
  \begin{enumerate}
    \item $\gamma$ is a distributive law $st\Rightarrow ts$.
    \item $(t,\gamma)$ is a lax monad morphism $s\rightarrow s$, and
          $((t,\gamma),\mu^t,\eta^t)$ is a monad on $(s,\mu^s,\eta^s)$ in $\mnd_l(\K)$.
    \item $(s,\gamma)$ is a colax monad morphism $t\rightarrow t$, and
          $((s,\gamma),\mu^s,\eta^s)$ is a monad on $(t,\mu^t,\eta^t)$ in $\mnd_c(\K)$.
    \item The composition $ts$ has a monad structure with a unit $\eta^t.\eta^s$ and multiplication $(\mu^t.\mu^s)*\gamma$.
  \end{enumerate}
\end{theorem}

\begin{corollary}\label{cor:fact_dist_is_lift_of_monad}
  To give a distributive law $st\Rightarrow ts$ is to give a lift of the monad $t$ to $A^s$,
  that is, a monad $(\bar{t}, \bar{\mu}, \bar{\eta})$ on $A^s$ such that
  $u^s\bar{t} = tu^s$,
  $u^s\bar{\mu} = \mu u^s$, and
  $u^s\bar{\eta} = \eta u^s$.
\end{corollary}
\begin{proof}
  Immediate from \Cref{thm:fact_dist_law_characterization,thm:fact_mnd_l_is_refl_subcaty}.
\end{proof}

\begin{proposition}
  Let $\gamma\colon st\Rightarrow ts$ be a distributive law.
  Then, from \Cref{cor:fact_dist_is_lift_of_monad}, we have a monad $(\bar{s}, \bar{\mu}, \bar{\eta})$ on $A^s$.
  The composition of the Eilenberg-Moore adjunctions
  \[
    \begin{tikzpicture}[baseline={(0,-0.4)}]
      \node(k) at (-2.2,0) {${(A^s)}^{\bar{t}}$};
      \node(m) at (0,0) {$A^s$};
      \node(l) at (2,0) {$A$};
      \draw[1cell,transform canvas={yshift=6pt},swap] (m) to node{$ $} (k);
      \draw[1cell,transform canvas={yshift=-6pt}] (k) to node{$ $} (m);
      \draw[1cell,transform canvas={yshift=6pt},swap] (l) to node{$ $} (m);
      \draw[1cell,transform canvas={yshift=-6pt}] (m) to node{$ $} (l);
      \node[labelsize] at (-1,0) {$\bot$};
      \node[labelsize] at (1,0) {$\bot$};
    \end{tikzpicture}
  \]
  exhibits ${(A^s)}^{\bar{t}}$ as the Eilenberg-Moore object of the composed monad $ts$ via $\gamma$.
\end{proposition}

Since distributive laws were defined with some 1-cells and 2-cells satisfying some equality of 2-cells,
we can define the free 2-category of the distributive law.

\begin{definition}
  The 2-category $\Dist$ is a single object 2-category generated by the computad
  \[
    \begin{tikzpicture}[baseline={(0,0)}]
      \node(s) at (0,0) {$*$};
      \node(t) at (2,0) {$*$};
      \draw[1cell, transform canvas={yshift = 5pt}] (s) to node{$t$} (t);
      \draw[1cell, transform canvas={yshift =-5pt}] (s) to node[swap]{$s$} (t);
    \end{tikzpicture}
    \qquad
    \begin{tikzpicture}[baseline={(0,-0.85)}]
      \node(s) at (0,0) {$*$};
      \node(t) at (2,0) {$*$};
      \draw[1cell,bend left] (s) to node {$1$} (t);
      \draw[1cell,bend right,swap] (s) to node{$t$} (t);
      \twocelld{1,0}{$\eta^t$}
      \node(s1) at (0,-1.7) {$*$};
      \node(t1) at (2,-1.7) {$*$};
      \draw[1cell,bend left=13] (s1) to node {$t$} (1,-1.3);
      \draw[1cell,bend left=13] (1,-1.3) to node {$t$} (t1);
      \draw[1cell,bend right,swap] (s1) to node{$t$} (t1);
      \twocelld{1,-1.7}{$\mu^t$}
    \end{tikzpicture}
    \qquad
    \begin{tikzpicture}[baseline={(0,-0.85)}]
      \node(s) at (0,0) {$*$};
      \node(t) at (2,0) {$*$};
      \draw[1cell,bend left] (s) to node {$1$} (t);
      \draw[1cell,bend right,swap] (s) to node{$s$} (t);
      \twocelld{1,0}{$\eta^s$}
      \node(s1) at (0,-1.7) {$*$};
      \node(t1) at (2,-1.7) {$*$};
      \draw[1cell,bend left=13] (s1) to node {$s$} (1,-1.3);
      \draw[1cell,bend left=13] (1,-1.3) to node {$s$} (t1);
      \draw[1cell,bend right,swap] (s1) to node{$s$} (t1);
      \twocelld{1,-1.7}{$\mu^s$}
    \end{tikzpicture}
    \qquad
    \begin{tikzpicture}[baseline={(0,0)}]
      \node(ul) at (-1, 1) {$*$};
      \node(ur) at ( 1, 1) {$*$};
      \node(dl) at (-1,-1) {$*$};
      \node(dr) at ( 1,-1) {$*$};
      \node at ( 1.2,-1.05) {$,$};
      \draw[1cell] (ul) to node{$t$} (ur);
      \draw[1cell,swap] (dl) to node{$t$} (dr);
      \draw[1cell,swap] (ul) to node{$s$} (dl);
      \draw[1cell] (ur) to node{$s$} (dr);
      \twocelldl{0,0}{$\gamma$}
    \end{tikzpicture}
  \]
  and subject to the equations that make both $(s,\mu^s,\eta^s)$ and $(t,\mu^t,\eta^t)$ monads
  and make $\gamma$ a distributive law $st \Rightarrow ts$.
  A 2-functor $\Dist\rightarrow\K$ is a distributive law in $\K$.
\end{definition}

Now we know that a distributive law is an object of $\mnd_l(\mnd_l(\K))$,
which is also equivalent to the data of a 2-functor $\Dist\rightarrow\K$.
This statement can be strengthened as follows.

\begin{proposition}
  Let $\K$ be a 2-category.
  \begin{enumerate}
    \item $\TL{\Dist, \K} \iso \TL{\Mnd, \TL{\Mnd, \K}}$.
    \item $\TO{\Dist, \K} \iso \TO{\Mnd, \TO{\Mnd, \K}}$.
    \item The composition monad via a distributive law of law defines a 2-functor
          \begin{align*}
            \mathbf{comp_\K}\colon\TL{\Dist,\K} & \quad\rightarrow\quad \TL{\Mnd, \K} \\
            \gamma\colon st \Rightarrow ts      & \quad \mapsto\quad \ ts
          \end{align*}
    \item Let $T$ be the 2-functor $\TL{\Mnd, -}\colon \TwoCat\rightarrow \TwoCat$.
          Then, $\mathbf{comp}\colon T^2\Rightarrow T$ and $\Id^{(-)}\colon \Id_\TwoCat \Rightarrow T$
          make $T$ a 2-monad on $\TwoCat$.
  \end{enumerate}
\end{proposition}

Distributive law is defined not only between two monads but also for any combinations: a monad and a comonad, a comonad and a monad, two comonads \cite{PowerWatanabe_2002_combining_mnd_comnd}.
For example, a distributive law of a monad $(t,\mu,\eta)$ over a comonad $(w,\delta,\varepsilon)$ is
a 2-cell $\gamma\colon tw \Rightarrow wt$ compatible with $\mu$, $\eta$, $\delta$, and $\varepsilon$.
A distributive law of this type does not give rise to a ``composition'' of a monad and a comonad,
but we can still state the following.

\begin{proposition}
  Let $\K$ be a 2-category.
  A distributive law of a monad over a comonad in $\K$ is equivalent to a comonad in $\mnd_l(\K) \iso \TL{\Mnd, \K}$.
  And it is also equivalent to a monad in the 2-category of comonads $\TO{\Mnd^\co, \K}$.
\end{proposition}

\section{Lax doctrinal adjunction and lax morphism classifier}\label{sec:adj_lax_classifier}
In this section, we view a monad as a \emph{lax algebra} of an identity 2-monad on a 2-category
and extend formal monad theory, including relations with adjunctions and Eilenberg-Moore objects,
to the theory of lax algebras of a general 2-monad.

In particular, we introduce \emph{lax doctrinal adjunctions} and show our first main theorem:
The 2-category $\lalg{c}{T}$ of lax algebras and colax morphisms is a coreflective sub-2-category of
the 2-category $\lladj{T}$ of lax doctrinal adjunctions.
As a corollary, we deduce that the 2-category $\mnd_c(\K)$ is a coreflective sub-2-category
of the 2-category $\mathrm{ladj}(\K)$ of left adjoints,
where the coreflection is defined by mapping a monad to its Kleisli adjunction,
which is the dual of \Cref{thm:fact_mnd_l_is_refl_subcaty}.

\subsection{2-monad}\label{subsec:2-monad}
We review some results of 2-monads in this section.
A 2-monad is a monad in $\TwoCat$.
We denote the Eilenberg-Moore category of a 2-monad $(T,e,m)$ by $\alg{s}{T}$.
In addition, there are other 2-categories of weak algebras and weak morphisms.
\begin{table}[H]
  \centering
  \rowcolors{2}{gray!25}{white}
  \begin{tabular}{c!{\color{white}\vline width 1pt}cccc}
    \rowcolor{gray!50}
             &
    \multicolumn{4}{c}{algebra}                                                 \\
    \arrayrulecolor{gray!50} \specialrule{5pt}{0pt}{-6pt}
    \arrayrulecolor{white} \cmidrule[1pt]{2-5}
    \arrayrulecolor{black}
    \rowcolor{gray!50}
    morphism & strict       & pseudo        & lax           & oplax             \\
    strict   & $\alg{s}{T}$ & $\palg{s}{T}$ & $\lalg{s}{T}$ & $\!\calg{s}{T}\!$ \\
    pseudo   & $\alg{p}{T}$ & $\palg{p}{T}$ & $\lalg{p}{T}$ & $\!\calg{p}{T}\!$ \\
    lax      & $\alg{l}{T}$ & $\palg{l}{T}$ & $\lalg{l}{T}$ & $\!\calg{l}{T}\!$ \\
    colax    & $\alg{c}{T}$ & $\palg{c}{T}$ & $\lalg{c}{T}$ & $\!\calg{c}{T}\!$
  \end{tabular}
\end{table}

The following are main examples we are concerned about.
See \cite{BlaKelPow_1989_two_dim_mon_theor} for more examples of 2-monads and their algebras.

\begin{example}\label{eg:T_2-Cat_is_2-monad}
  There is a 2-monad $T_{\TwoCat}$ on the 2-category of $\Cat$-graphs as shown in \cite{Shulman_2012_not_every_ps_alg}.
  The strict and pseudo $T_{\TwoCat}$-algebras are small 2-categories and unbiased bicategories,
  $w$-morphisms are $w$-functors,
  and the 2-cells are icons \cite{Lack_2007_icons}.
  \begin{table}[H]
    \centering
    \begin{tabular}{ll}
      2-monad $\mathbf{T_{\TwoCat}}$ &                     \\\toprule
      strict algebra                 & 2-category          \\
      pseudo algebra                 & unbiased bicategory \\\midrule
      strict morphism                & 2-functor           \\
      pseudo morphism                & pseudo functor      \\
      (co)lax morphism               & (op)lax functor     \\\midrule
      2-cell                         & icon                \\\bottomrule
    \end{tabular}
  \end{table}
  In this example, $\alg{s}{T_{\TwoCat}}$ is $\TwoCatTI$, and $\palg{l}{\TwoCat}$ is 2-equivalent to $\BiCatLI$.

  The free 2-category $\mathcal{F}\mathcal{G}$ generated from a $\Cat$-graph
  $\mathcal{G} = (X, {\{\mathcal{G}(x,y)\in\Cat\}}_{x,y\in X})$ has a set of objects $X$
  and hom categories
  \[
    \mathcal{F}\mathcal{G}(x,y) =
    \sum_{p\colon x\rightarrow y}\ \prod_{(p_i,p_{i+1})\in p}
    \mathcal{G}(p_i,p_{i+1})
  \]
  where the sum is taken over a range of sequences of objects, starting at $x$ and ending at $y$.
\end{example}

\begin{example}\label{eg:T_Fun_is_2-monad}
  Let $\A$ be a small 2-category and $\K$ be a cocomplete 2-category.
  There is a 2-monad $T_{\mathbf{Fun}}$ on $[\ob(\A),\K]$,
  whose 2-category of algebras $\walg{w}{w'}{T_\mathbf{Fun}}$ is the 2-category $\ww{w}{w'}{\A,\K}$ of $w$-functors $\A\rightarrow\K$, $w'$-transformations, and modifications.
  \begin{table}[H]
    \centering
    \begin{tabular}{ll}
      2-monad $\mathbf{T_{\mathbf{Fun}}}$ &                              \\\toprule
      strict algebra                      & 2-functor $\A\rightarrow \K$ \\
      pseudo algebra                      & pseudo functor               \\
      (op)lax algebra                     & (op)lax functor              \\\midrule
      strict morphism                     & 2-/strict nat.\@ trans.      \\
      pseudo morphism                     & pseudo nat.\@ trans.         \\
      (co)lax morphism                    & (op)lax nat.\@ trans.        \\\midrule
      2-cell                              & modification                 \\\bottomrule
    \end{tabular}
  \end{table}

  Let $H\colon \ob(\A)\rightarrow \A$ be the inclusion 2-functor.
  Then, this 2-monad $T_\mathbf{Fun}$ arises from the adjunction $\mathrm{Lan}_H\dashv (-)\circ H$.
  This 2-monad is cocontinuous since colimits in $\TT{\A,\K}$ are calculated pointwise.
\end{example}

\begin{example}
  Let $w,w' \in \{s,p,l,c\}$.
  Since $\Tw{w'}{\A,-}$ is a 2-functor $\TwoCAT\rightarrow \TwoCAT$,
  for each 2-monad $T$ on $\K$, $\Tw{w}{\A,T}$ defines a 2-monad on $\Tw{w}{\A,\K}$.
  That 2-monad satisfies $\walg{w}{\bar{w}'}{\Tw{w}{\A,T}} \iso \Tw{w}{\A,\walg{w}{\bar{w}'}{T}}$.
\end{example}

\subsection{Weak morphism classifier}\label{subsec:classifier}
For each $w,w'\in\{s,p,l,c\}$, if the inclusion $\alg{s}{T}\rightarrow \walg{w}{w'}{T}$ has a left 2-adjunction, we call it a \emph{$w'$-morphism classifier} and denote it by $\Q^w_{w'}\colon \walg{w}{w'}{T}\rightarrow \alg{s}{T}$.
It is known that if $\alg{s}{T}$ has enough colimits, the weak morphism classifier exists.
This fact was first studied in the paper \cite{BlaKelPow_1989_two_dim_mon_theor} for the case of $w=s$,
and other cases, for example, $w = w' = l$, were shown in \cite{Lack_2002_cod_obj_coh}.
Later, Nunes studied biadjoint triangle theorems
in~\cite{nunes2017biadjointtriangles,nunes2018liftingbiadjointslaxalgebras},
and one can see the existence of weak morphism classifiers as consequence of them.
We review Lack's results, in particular, when $w = l$ and ${w'} = \bar{w} = c$.

\begin{definition}
  \emph{Lax coherence data} consist of three objects $X_0$, $X_1$, $X_2$, six 1-cells:
  \[
    \begin{tikzpicture}[baseline={(0,-0.8)}]
      \node(dummy) at (0,0.5) {};
      \node(X1) at (3,0) {$X_0,$};
      \node(X2) at (0,0) {$X_1$};
      \node(X3) at (-3,0) {$X_2$};
      \draw[1cell, transform canvas={yshift=12pt}] (X3) -- node {$r$} (X2);
      \draw[1cell] (X3) -- node {$s$} (X2);
      \draw[1cell, transform canvas={yshift=-12pt}] (X3) -- node {$t$} (X2);
      \draw[1cell, transform canvas={yshift=12pt}] (X2) -- node(l) {$v$} (X1);
      \draw[1cell,swap] (X1) -- node {$i$} (X2);
      \draw[1cell, transform canvas={yshift=-12pt}] (X2) -- node {$w$} (X1);
    \end{tikzpicture}
  \]
  and three equality and two 2-cells:
  \begin{align*}
    vi = 1, \qquad vs = vr, \qquad wr = vt \\
    \chi^2\colon wt \Rightarrow ws, \qquad \chi^0\colon 1\Rightarrow wi .
  \end{align*}
  \emph{Oplax coherence data} has the opposite direction of 2-cells $\chi^2\colon ws\Rightarrow wt$, $\chi^0\colon 1\Rightarrow wi$,
  \emph{pseudo coherence data} has those with isomorphic 2-cells, and strict coherence data has five equalities.

  A \emph{colax codescent object of lax coherence data} is a triple $(C,c\colon X_0\rightarrow C,\gamma\colon cw\Rightarrow cv)$
  satisfying the following equalities with one- and two-dimensional universality.
  \[
    \begin{tikzpicture}[baseline=(l.base)]
      \node(Top) at (0,1.6) {$X_1$};
      \node(LU) at (-1.7,0.8) {$X_2$};
      \node(RU) at (1.7,0.8) {$X_0$};
      \node(MU) at (0,0) {$X_1$};
      \node(Bot) at (0,-1.6) {$X_0$};
      \node(LD) at (-1.7,-0.8) {$X_1$};
      \node(RD) at (1.7,-0.8) {$C$};
      \draw[1cell] (LU) to node {$t$} (Top);
      \draw[1cell,swap] (LU) to node {$s$} (MU);
      \draw[1cell,swap] (LU) to node {$r$} (LD);
      \draw[1cell] (Top) to node {$w$} (RU);
      \draw[1cell,swap] (MU) to node {$w$} (RU);
      \draw[1cell] (MU) to node {$v$} (Bot);
      \draw[1cell,swap] (LD) to node {$v$} (Bot);
      \draw[1cell] (RU) to node(l) {$c$} (RD);
      \draw[1cell,swap] (Bot) to node {$c$} (RD);
      \twocelld{0,0.8}{$\chi^2$}
      \twocelld{0.8,-0.4}{$\gamma$}
    \end{tikzpicture}
    \quad=\quad
    \begin{tikzpicture}[baseline=(l.base)]
      \node(Top) at (0,1.6) {$X_1$};
      \node(LU) at (-1.7,0.8) {$X_2$};
      \node(RU) at (1.7,0.8) {$X_0$};
      \node(Bot) at (0,-1.6) {$X_0$};
      \node(LD) at (-1.7,-0.8) {$X_1$};
      \node(RD) at (1.7,-0.8) {$C$};
      \node(MD) at (0,0) {$X_0$};
      \draw[1cell] (LU) to node {$t$} (Top);
      \draw[1cell,swap] (LU) to node {$r$} (LD);
      \draw[1cell] (Top) to node {$w$} (RU);
      \draw[1cell,swap] (LD) to node {$v$} (Bot);
      \draw[1cell] (LD) to node {$w$} (MD);
      \draw[1cell] (RU) to node(l) {$c$} (RD);
      \draw[1cell,swap] (Bot) to node {$c$} (RD);
      \draw[1cell,swap] (MD) to node {$c$} (RD);
      \draw[1cell,swap] (Top) to node {$v$} (MD);
      \twocelld{0.8,0.4}{$\gamma$}
      \twocelld{0,-0.8}{$\gamma$}
    \end{tikzpicture}
  \]
  \[
    \begin{tikzpicture}[baseline={(0,-0.3)}]
      \node(top) at (0,0) {$X_1$};
      \node(ul) at (-1.7,0.8) {$X_0$};
      \node(ur) at (1.7,0.8) {$X_0$};
      \node(dl) at (-1.7,-0.8) {$X_0$};
      \node(dr) at (1.7,-0.8) {$C$};
      \node(m) at (0,-1.6) {$X_0$};
      \draw[1cell equality] (ul) to (dl);
      \draw[1cell equality,bend left=17] (ul) to (ur);
      \draw[1cell, swap] (ul) to node{$i$} (top);
      \draw[1cell, swap] (top) to node{$w$} (ur);
      \draw[1cell] (ur) to node {$c$} (dr);
      \draw[1cell,swap] (top) to node{$v$} (m);
      \draw[1cell equality] (dl) to (m);
      \draw[1cell,swap] (m) to node {$c$} (dr);
      \twocelld{0,0.65}{$\chi^0$}
      \twocelld{0.8,-0.4}{$\gamma$}
    \end{tikzpicture}
    \quad=\quad
    \begin{tikzpicture}[baseline={(0,-0.3)}]
      \node(ul) at (-1.7,0.8) {$X_0$};
      \node(ur) at (1.7,0.8) {$X_0$};
      \node(dl) at (-1.7,-0.8) {$X_0$};
      \node(dr) at (1.7,-0.8) {$C$};
      \node(m) at (0,-1.6) {$X_0$};
      \draw[1cell equality] (ul) to (dl);
      \draw[1cell equality,bend left=17] (ul) to (ur);
      \draw[1cell] (ur) to node {$c$} (dr);
      \draw[1cell equality] (dl) to (m);
      \draw[1cell,swap] (m) to node {$c$} (dr);
      \draw[1cell,swap,bend left=17] (dl) to node {$c$} (dr);
    \end{tikzpicture}
  \]
  We can also consider any $w'$-codescent object of $w$-coherence data for each $w$ and $w'$, which has the corresponding direction of $\gamma$.
  We call them $(w,w')$-codescent objects.
\end{definition}

In much of the literature, the term codescent object refers to $(s,w)$-codescent objects,
but as the main subject in this section is lax algebras, we need to consider general coherence data.

Now, with a certain weight, $(w,w')$-codescent objects are weighted colimits, and in particular,
they can be generated from coinserter and coequifiers.

\begin{theorem}
  Let $(T,m,e)$ be a 2-monad on $\K$.
  If $\alg{s}{T}$ admits $(w,w')$-codescent objects,
  the weak morphism classifier $\Q^w_{w'}\colon\walg{w}{w'}{T} \rightarrow \alg{s}{T}$ exists.
\end{theorem}
\begin{proof}
  We give a sketch of the proof for the case of $w = l$, $w'=c$.
  A detailed proof can be found in \cite{BlaKelPow_1989_two_dim_mon_theor,Lack_2002_cod_obj_coh}.

  Let $\alX = (X,x,\xi^2,\xi^0)$ be a lax algebra.
  And let us also write the free strict algebra on $T^nX$ $(n\geq1)$ as $T^n\alX = (T^{n}X, T^{n+1}X\xrightarrow{\mu}T^nX)$.
  Then, by the 2-adjunction between $\alg{s}{T}$ and $\K$,
  one can check that to give a colax morphism $\f\colon\alX\rightsquigarrow \alB$ to a strict algebra $\alB$
  is to give a pair $(\f'\colon T\alX\rightarrow\alB, \bar{f}'\colon g.Tx\Rightarrow g.m_X)$
  satisfying the following equality of 2-cells in $\alg{s}{T}$.
  \[
    \begin{tikzpicture}[baseline=(l.base)]
      \node(Top) at (0,1.6) {$T^2\alX$};
      \node(LU) at (-1.7,0.8) {$T^3\alX$};
      \node(RU) at (1.7,0.8) {$T\alX$};
      \node(MU) at (0,0) {$T^2\alX$};
      \node(Bot) at (0,-1.6) {$T\alX$};
      \node(LD) at (-1.7,-0.8) {$T^2\alX$};
      \node(RD) at (1.7,-0.8) {$\alB$};
      \draw[1cell] (LU) to node {$T^2x$} (Top);
      \draw[1cell,swap] (LU) to node {$Tm$} (MU);
      \draw[1cell,swap] (LU) to node {$mT$} (LD);
      \draw[1cell] (Top) to node {$Tx$} (RU);
      \draw[1cell,swap] (MU) to node {$Tx$} (RU);
      \draw[1cell] (MU) to node {$m$} (Bot);
      \draw[1cell,swap] (LD) to node {$m$} (Bot);
      \draw[1cell] (RU) to node(l) {$\f'$} (RD);
      \draw[1cell,swap] (Bot) to node {$\f'$} (RD);
      \twocelld{0,0.8}{$T\xi^2$}
      \twocelld{0.8,-0.4}{$\bar{f}'$}
    \end{tikzpicture}
    \quad=\quad
    \begin{tikzpicture}[baseline=(l.base)]
      \node(Top) at (0,1.6) {$T^2\alX$};
      \node(LU) at (-1.7,0.8) {$T^3\alX$};
      \node(RU) at (1.7,0.8) {$T\alX$};
      \node(Bot) at (0,-1.6) {$T\alX$};
      \node(LD) at (-1.7,-0.8) {$T^2\alX$};
      \node(RD) at (1.7,-0.8) {$\alB$};
      \node(MD) at (0,0) {$T\alX$};
      \draw[1cell] (LU) to node {$T^2x$} (Top);
      \draw[1cell,swap] (LU) to node {$mT$} (LD);
      \draw[1cell] (Top) to node {$Tx$} (RU);
      \draw[1cell,swap] (LD) to node {$m$} (Bot);
      \draw[1cell] (LD) to node {$Tx$} (MD);
      \draw[1cell] (RU) to node(l) {$\f'$} (RD);
      \draw[1cell,swap] (Bot) to node {$\f'$} (RD);
      \draw[1cell,swap] (MD) to node {$\f'$} (RD);
      \draw[1cell,swap] (Top) to node {$m$} (MD);
      \twocelld{0.8,0.4}{$\bar{f}'$}
      \twocelld{0,-0.8}{$\bar{f}'$}
    \end{tikzpicture}
  \]
  \[
    \begin{tikzpicture}[baseline={(0,-0.3)}]
      \node(top) at (0,0) {$T^2\alX$};
      \node(ul) at (-1.7,0.8) {$T\alX$};
      \node(ur) at (1.7,0.8) {$T\alX$};
      \node(dl) at (-1.7,-0.8) {$T\alX$};
      \node(dr) at (1.7,-0.8) {$\alB$};
      \node(m) at (0,-1.6) {$T\alX$};
      \draw[1cell equality] (ul) to (dl);
      \draw[1cell equality,bend left=17] (ul) to (ur);
      \draw[1cell, swap] (ul) to node{$Te$} (top);
      \draw[1cell, swap] (top) to node{$Tx$} (ur);
      \draw[1cell] (ur) to node {$\f'$} (dr);
      \draw[1cell,swap] (top) to node{$m$} (m);
      \draw[1cell equality] (dl) to (m);
      \draw[1cell,swap] (m) to node {$\f'$} (dr);
      \twocelld{0,0.65}{$T\xi^0$}
      \twocelld{0.8,-0.4}{$\bar{f}'$}
    \end{tikzpicture}
    \quad=\quad
    \begin{tikzpicture}[baseline={(0,-0.3)}]
      \node(ul) at (-1.7,0.8) {$T\alX$};
      \node(ur) at (1.7,0.8) {$T\alX$};
      \node(dl) at (-1.7,-0.8) {$T\alX$};
      \node(dr) at (1.7,-0.8) {$\alB$};
      \node(m) at (0,-1.6) {$T\alX$};
      \draw[1cell equality] (ul) to (dl);
      \draw[1cell equality,bend left=17] (ul) to (ur);
      \draw[1cell] (ur) to node {$\f'$} (dr);
      \draw[1cell equality] (dl) to (m);
      \draw[1cell,swap] (m) to node {$\f'$} (dr);
      \draw[1cell,swap,bend left=17] (dl) to node {$\f'$} (dr);
    \end{tikzpicture}
  \]
  Therefore, the $(l,c)$-codescent object of the following lax coherence data in $\alg{s}{T}$ classifies the colax morphism out of $\alX$.
  \[
    \begin{tikzpicture}[baseline={(0,-0.8)}]
      \node(dummy) at (0,0.5) {};
      \node(X1) at (3,0) {$T\alX$};
      \node(X2) at (0,0) {$T^2\alX$};
      \node(X3) at (-3,0) {$T^3\alX$};
      \draw[1cell, transform canvas={yshift=14pt}] (X3) -- node {$mT$} (X2);
      \draw[1cell] (X3) -- node {$Tm$} (X2);
      \draw[1cell, transform canvas={yshift=-14pt}] (X3) -- node {$T^2x$} (X2);
      \draw[1cell, transform canvas={yshift=14pt}] (X2) -- node(l) {$m$} (X1);
      \draw[1cell,swap] (X1) -- node {$Te$} (X2);
      \draw[1cell, transform canvas={yshift=-14pt}] (X2) -- node {$Tx$} (X1);
    \end{tikzpicture}
  \]
  \[
    T\xi^2\colon Tx.T^2x \Rightarrow Tx.Tm, \qquad T\xi^0\colon Tx.1\Rightarrow Te.
  \]
\end{proof}

Suppose that $(\g \colon T\alX\rightarrow \Q\alX, \bar{g}\colon \g.Tx\Rightarrow \g.m)$ is the colimiting cocone of the coherence data.
Then the 1-cell $p$ of the unit $\p=(p,\bar{p})\colon\alX \rightsquigarrow \Q\alX$ is given by $X\xrightarrow{e}TX\xrightarrow{g}\Q X$.

\begin{example}\label{eg:T_2-cat_admits_classifier}
  The 2-monad $T_\TwoCat$ admits the lax morphism classifier $\Q\colon \BiCatLI \rightarrow \TwoCatTI$.
  We write $\Q\C$ as $\barC$ and call it a \emph{lax functor classifier}.
  From this adjunction, for each 2-category $\K$, we have an isomorphism of categories natural in $\K$.
  \[
    \TwoCatTI(\barC, \K) \iso \BiCatLI(\C,\K)
  \]
\end{example}

This isomorphism says the existence of a canonical bijective correspondence between 2-functor $\barC\rightarrow\K$ and lax functor $\C\rightsquigarrow\K$, which can extend to the correspondence of icons.
More interestingly, this isomorphism of hom-categories extends further.

\begin{proposition}\label{prop:lax_fun_classifier}
  The pre-composition of the lax functor $\p_\C\colon \C\rightsquigarrow \barC$ induces an isomorphism of 2-categories
  $\Tw{w}{\barC, \K} \iso \Lw{w}{\C, \K}$ for each $w \in \{s,p,l,c\}$,
  which is natural in $\C\in\BiCatLZ$ and 2-natural in $\K\in\TwoCAT$.
\end{proposition}

\begin{proof}
  Let $F,G$ be 2-functors $\barC\rightarrow\K$.
  \Cref{lem:w_nat_is_functor} shows the bijective correspondence:
  \begin{center}
    \begin{prooftree}[rule style=double]
      \hypo{\alpha\colon F\rightsquigarrow G\colon w'\text{-natural transformation}}
      \infer1{H\colon \barC\rightarrow \Lw{\bar{w}'}{\Two, \K} \text{\ s.t.\ } \dom H =F, \cod H = G}
      \infer1{H'\colon \C\rightsquigarrow \Lw{\bar{w}'}{\Two, \K} \text{\ s.t.\ } \dom H' =F\p, \cod H' = G\p}
      \infer1{\alpha'\colon F\p\rightsquigarrow G\p\colon w'\text{-natural transformation}}
    \end{prooftree}
  \end{center}
  For modifications, it suffices to consider $\Eye$ in place of $\Two$.
\end{proof}

We can also see $\barC$ in the other way:
Since the data of a lax functor $\C\rightsquigarrow\K$ consists of
the image of 1-cells $\{Ff\}$ and 2-cells $\{F\alpha\}$,
comparison 2-cells $\{F^2_{gf}\colon Fg.Ff\Rightarrow F(gf)\}$ and $\{F^0_A \colon 1_{FA}\Rightarrow Fa_A\}$,
and some equality of 1-cells and 2-cells,
this is just specifying a diagram in $\K$.
Therefore, there exists a ``free lax functor 2-category'' $\barC$ such that data of a diagram $\barC\rightarrow\K$
is the same as that of a lax functor $\C\rightsquigarrow\K$.

More explicitly, $\barC$ is the 2-category obtained by generating from the following computad $\cG_\C$,
\begin{itemize}
  \item Objects $|\cG_\C|$ are those in $\C$.
  \item 1-cells are those in $\C$.
  \item 2-cells are those in $\C$ together with comparison maps $F_{(f_0,\dots,f_{n-1})}$ for each composable sequence $(f_0,\dots,f_{n-1})$ in $\C$.
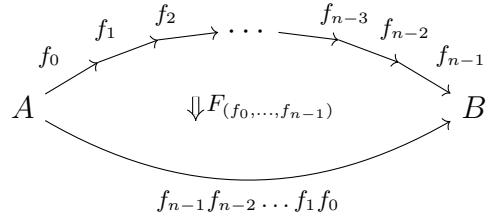
        \[
          \begin{tikzpicture}
            \node(a) at (-3,0) {$A$};
            \node(b) at (3,0) {$B$};
            \node at (0,1) {$\cdots$};
            \draw[1cell] (a) to node{$f_0$} (-2,0.6);
            \draw[1cell] (-2,0.6) to node{$f_1$} (-1.2,0.9);
            \draw[1cell] (-1.2,0.9) to node{$f_2$} (-0.4,1);
            \draw[1cell] (0.4,1) to node{$f_{n-3}$} (1.2,0.9) ;
            \draw[1cell] (1.2,0.9) to node{$f_{n-2}$}(2,0.6);
            \draw[1cell] (2,0.6) to node{$f_{n-1}$} (b);
            \draw[1cell,bend right,swap] (a) to node{$f_{n-1}f_{n-2}\dots f_1f_0$} (b);
            \twocelld{-0.7,0}{$F_{(f_0,\dots,f_{n-1})}$}
          \end{tikzpicture}
        \]
\end{itemize}
and taking the quotient by some equality of 2-cells.
The set of 1-cells in the resulting 2-category $\barC$ consists of
all composable sequences of 1-cells in $\C$.

Then, the unit $\p_\C\colon\C\rightsquigarrow\barC$ is a bijective-on-objects lax functor that maps $f$ to $(f)$.
On the other hand, another bijective-on-objects lax functor of the opposite direction $\q_\C\colon \barC\rightsquigarrow \C$ exists that maps a composable sequence $(f_0,\dots,f_{n-1})$ to its composition $f_{n-1}\cdots f_0$ and $F_{(f_0,\dots,f_{n-1})}$ to the identity.
When $\C$ is a 2-category $\A$, this $\q_\A$ coincides with the counit, and $\q_\A$ becomes a 2-functor.
Calculating the composition $\q_\C\p_\C$, it turns out to be identity,
and there exists an icon to another composition $\eta^\C\colon 1\Rightarrow \p_\C\q_\C$,
whose component $\eta^\C_{(f_0,\dots,f_{n-1})}$ is $F_{(f_0,\dots,f_{n-1})}$.
Since $\q_\C\eta^\C = 1$ and $\eta^\C\q_\C = 1$, $\p_\C$ is right adjoint to $\q_\C$ with identity counit in $\BiCatLI$.

There is also a colax morphism classifier $\mathcal{Q}^\TwoCat_{c}\C$ classifying oplax functors,
which is, in fact, isomorphic to $\overline{\C^\co}^\co$.
This is because $\Ow{w}{\C, \K} \iso \Lw{\bar{w}}{\C^\co,\K^\co}^\co \iso \Tw{\bar{w}}{\overline{\C^\co}, \K^\co}^\co \iso \Tw{w}{\overline{\C^\co}^\co, \K}$.

We now restrict to the case of $w' = \bar{w}$ and suppose that
$T$ preserves $(w,\bar{w})$-codescent objects.
We show that any unit $\p\colon \alX\rightarrow\Q\alX$ has an adjoint in the base 2-category $\K$.
That is,
\begin{itemize}
  \item if $w = c$, $p$ has a left adjoint $q\colon \Q X\rightarrow X$ in $\K$.
  \item if $w = l$, $p$ has a right adjoint $q\colon \Q X\rightarrow X$ in $\K$.
  \item if $w = p$, $p$ and $q\colon \Q X\rightarrow X$ make $X$ and $\Q X$ equivalent in $\K$.
\end{itemize}

\begin{theorem}\label{thm:w-barw_classifier_induce_adjunction}
  Suppose that $\K$ admits and $T$ preserves $(w,\bar{w})$-codescent objects.
  Then, $\alg{s}{T}$ has all $(w,\bar{w})$-codescent objects, and therefore there exists a $\bar{w}$-morphism classifier
  $\Q\colon\walg{w}{\bar{w}}{T}\rightarrow \alg{s}{T}$.
  Let $\alX$ be a $w$-algebra.
  From the universality of codescent object $\Q X$ in $\K$,
  there exists a 1-cell $q\colon \Q X \rightarrow X$ satisfying
  the following,
  which is adjoint to $p$ in $\K$.
  \[
    \begin{tikzpicture}[baseline={(0,0)}]
      \node(r) at ( 2,0) {$X$};
      \node(t) at (0, 0.7) {$TX$};
      \node(b) at (0,-0.7) {$TX$};
      \node(l) at (-2,0) {$T^2X$};
      \node at (0,0) {$\xi^2$};
      \draw[1cell] (l) to node {$Tx$} (t);
      \draw[1cell] (t) to node {$x$} (r);
      \draw[1cell,swap] (l) to node {$m$} (b);
      \draw[1cell,swap] (b) to node {$x$} (r);
    \end{tikzpicture}
    \quad=\quad
    \begin{tikzpicture}[baseline={(0,0)}]
      \node(r) at ( 0.5,0) {$\Q X$};
      \node(x) at (2,0) {$X$};
      \node(t) at (-0.7, 0.7) {$TX$};
      \node(b) at (-0.7,-0.7) {$TX$};
      \node(l) at (-2,0) {$T^2X$};
      \node at (-0.7,0) {$\bar{g}$};
      \draw[1cell] (l) to node {$Tx$} (t);
      \draw[1cell] (t) to node {$g$} (r);
      \draw[1cell,swap] (l) to node {$m$} (b);
      \draw[1cell,swap] (b) to node {$g$} (r);
      \draw[1cell] (r) to node {$q$} (x);
    \end{tikzpicture}
  \]
\end{theorem}
\begin{proof}
  We prove this for the case $w = l$.
  From the axioms of lax algebra, it follows that $\xi^2\colon x.Tx\Rightarrow xm$ defines
  an oplax cocone of the lax coherence data.
  Therefore, we can obtain $q$ from the universality of $\Q X$.
  Note that, if $w = w' = l$, this construction fails because the direction of 2-cells does not match.

  We need to find the unit $\eta\colon 1\Rightarrow qp$ and counit $\varepsilon\colon pq\Rightarrow 1$.
  The unit $\eta$ is easy:
  Since $qp = qge = xe$, we define $\eta$ by $\alpha^0\colon 1\Rightarrow xe$.
  For the counit $\varepsilon\colon pq\Rightarrow 1_{\Q X}$, we use the 2-dimensional universality of $\Q X$.
  We have a 2-cell $pqg = gex = g.Tx.eT \xrightarrow{\bar{g}.eT} g.m.eT = g$,
  so if we show that this is a modification between cocones,
  we obtain a 2-cell $\varepsilon$ satisfying $\varepsilon.g = \bar{g}.eT$.
  It suffices to show the diagram below commutes.
  \[
    \begin{tikzpicture}
      \node(uL) at (0,  0.7) {$g.Tx.eT.Tx$};
      \node(ul) at (4,  0.7) {$gm.eT.Tx$};
      \node(ur) at (7,  0.7) {$g.Tx$};
      \node(uR) at (11, 0.7) {$g.m$};
      \node(dL) at (0, -0.7) {$pqg.Tx$};
      \node(dl) at (4, -0.7) {$g.Tx.eT.Tx$};
      \node(dr) at (7, -0.7) {$g.Tx.eT.Tx$};
      \node(dR) at (11,-0.7) {$g.Tx.eT.Tx$};
      \draw[1cell] (uL) to node {$\bar{g}.eT.Tx$} (ul);
      \draw[1cell] (ur) to node {$\bar{g}$} (uR);
      \draw[1cell] (dL) to node {$pq.\bar{g}$} (dl);
      \draw[1cell] (dr) to node {$\bar{g}.eT.m$} (dR);
      \draw[1cell equality] (uL) to (dL);
      \draw[1cell equality] (uR) to (dR);
      \draw[1cell equality] (ul) to (ur);
      \draw[1cell equality] (dl) to (dr);
    \end{tikzpicture}
  \]
  This diagram can be shown by pre-composing $e$ to the following diagram.
  \[
    \begin{tikzpicture}
      \node(L) at (1, 0) {$g.Tx.T^2x$};
      \node(ul) at (4, 0.7) {$g.m.T^2x$};
      \node(ur) at (7, 0.7) {$g.Tx.mT$};
      \node(R) at (10, 0) {$g.m.mT$};
      \node(dl) at (4,-0.7) {$g.Tx.Tm$};
      \node(dr) at (7,-0.7) {$g.m.Tm$};
      \draw[1cell] (L) to node {$\bar{g}.T^2$} (ul);
      \draw[1cell] (ur) to node {$\bar{g}.mT$} (R);
      \draw[1cell,swap] (L) to node {$g.T\xi^2$} (dl);
      \draw[1cell] (dl) to node {$\bar{g}.Tm$} (dr);
      \draw[1cell equality] (ul) to (ur);
      \draw[1cell equality] (dr) to (R);
    \end{tikzpicture}
  \]
  What remains to show is the triangular identities.
  The first one $1_p = p\xrightarrow{p\eta}pqp\xrightarrow{\varepsilon p}p$ is straightforward.
  To show the other, it suffices to show
  $1_{qg} = 1_x = x\xrightarrow{\eta x}qpx\xrightarrow{q\bar{g}.eT\ =\ q\varepsilon g}qx$
  because of the 2-dimesional universality of $\Q X$, which is not hard to prove.
\end{proof}

\begin{remark}
  Suppose that $\K$ admits and $T$ preserves $(p,w)$-codescent objects.
  As in the former theorem, the inclusion $\alg{s}{T}\rightarrow \palg{w}{T}$ has a right adjoint $\Q_w$.
  This time, the unit $\p$ has an adjunction $q$ not only in $\K$ but also in $\palg{w}{T}$.
\end{remark}

\begin{example}\label{eg:T_Fun_has_Q_and_adjoint}
  We saw in \Cref{eg:T_Fun_is_2-monad}
  that the 2-monad $T_\mathbf{Fun}$ on $\TT{\ob(\A),\K}$ is cocontinuous.
  Therefore, \Cref{thm:w-barw_classifier_induce_adjunction} shows that a colax morphism classifier $\Q_c\colon\LO{\A,\K}\rightarrow \TT{\A,\K}$ exists,
  and between a lax functor $F\colon \A\rightsquigarrow \K$ and the 2-functor $\Q F\colon\A\rightarrow \K$,
  there are pointwise adjunctions $p_A\dashv q_A$.
  \[
    \begin{tikzpicture}
      \node(l) at (-1.5,0) {$FA$};
      \node(r) at (1.5,0) {$\Q FA$};
      \node[labelsize] at (-0.08,-0.17) {$\bot$};
      \draw[1cell] (l) to node {$p_A$} (r);
      \draw[1cell,bend left=20] (r) to node {$q_A$} (l);
    \end{tikzpicture}
  \]
  This colax morphism classifier $\Q_c$ is the \emph{Kleisli construction} of lax functors in \cite{Street_1972_two_const_lax_func}.
  In particular, when $\A = \One$ and $T = \Id_\K$, the colax morphism classifier is the Kleisli object of a monad in $\K$.

  Now, let $\A$ be a small 2-category and $\K$ be a complete 2-category.
  Considering the monad $T$ on $[\ob(\A),\K^\op]$ whose 2-category of strict algebras and strict morphisms is $\TT{\A^\op,\K^\op}$,
  we have a right adjoint $\R_l\colon\LO{\A,\K}\rightarrow \TT{\A,\K}$ of the obvious inclusion 2-functor and
  pointwise adjunctions.
  \[
    \begin{tikzpicture}
      \node(l) at (-1.5,0) {$FA$};
      \node(r) at (1.5,0) {$\R_l FA$};
      \node[labelsize] at (-0.1,0.2) {$\bot$};
      \draw[1cell] (r) to node {$u_A$} (l);
      \draw[1cell,bend left=20,pos=0.5] (l) to node {$f_A$} (r);
    \end{tikzpicture}
  \]
  This time, the right adjunction $\R_l$ is the \emph{Eilenberg-Moore construction} of lax functors in \cite{Street_1972_two_const_lax_func},
  which generalizes the Eilenberg-Moore object.
  We call it a \emph{lax morphism coclassifier}.
\end{example}

Combining the result of \Cref{prop:lax_fun_classifier} and \Cref{eg:T_Fun_has_Q_and_adjoint},
we obtain the following commutative diagrams of adjunctions.
\[
  \begin{tikzpicture}[baseline=(b.base)]
    \node(ul) at (0,2) {$\TT{\,\barA,\K}$};
    \node(ur) at (4,2) {$\TO{\,\barA,\K}$};
    \node(dl) at (0,0) {$\TT{\A,\K}$};
    \node(dr) at (4,0) {$\LO{\A,\K}$};
    \draw[1cell, transform canvas={yshift=-5pt}] (ul) to node {} (ur);
    \draw[1cell, transform canvas={yshift=-5pt}] (dl) to node {} (dr);
    \draw[1cell, swap, transform canvas={xshift=5pt}] (dl) to node {$-\circ p$} (ul);
    \draw[1cell, swap] (dr) to node [label=left:{$\iso$}]{$-\circ p$} (ur);
    \draw[1cell, swap, transform canvas={yshift=5pt}] (ur) to node {$\mathcal{Q}_{c}$} (ul);
    \draw[1cell, swap, transform canvas={yshift=5pt}] (dr) to node {$\mathbf{Kl}\coloneqq\mathcal{Q}_{c}$} (dl);
    \draw[1cell, swap, transform canvas={xshift=-5pt}] (ul) to node(b) {$\mathrm{Lan}_p(-)$} (dl);
    \node at (0,1) {$\scriptstyle\dashv$};
    \node at (2,0) {$\scriptstyle\bot$};
    \node at (2,2) {$\scriptstyle\bot$};
  \end{tikzpicture}
  \qquad
  \begin{tikzpicture}[baseline=(b.base)]
    \node(ul) at (0,2) {$\TT{\,\barA,\K}$};
    \node(ur) at (4,2) {$\TL{\,\barA,\K}$};
    \node(dl) at (0,0) {$\TT{\A,\K}$};
    \node(dr) at (4,0) {$\LL{\A,\K}$};
    \draw[1cell, transform canvas={yshift=5pt}] (ul) to node {} (ur);
    \draw[1cell, transform canvas={yshift=5pt}] (dl) to node {} (dr);
    \draw[1cell, transform canvas={xshift=-5pt}] (dl) to node {$-\circ u$} (ul);
    \draw[1cell, swap] (dr) to node [label=left:{$\iso$}]{$-\circ u$} (ur);
    \draw[1cell, transform canvas={yshift=-5pt}] (ur) to node {$\mathcal{R}_l$} (ul);
    \draw[1cell, transform canvas={yshift=-5pt}] (dr) to node {$\mathbf{EM}\coloneqq\mathcal{R}_l$} (dl);
    \draw[1cell, transform canvas={xshift=5pt}] (ul) to node(b) {$\mathrm{Ran}_u(-)$} (dl);
    \node at (0,1) {$\scriptstyle\dashv$};
    \node at (2,0) {$\scriptstyle\bot$};
    \node at (2,2) {$\scriptstyle\bot$};
  \end{tikzpicture}
\]

The left adjoint $\Q_c$ and the right adjoint $\R_l$ on the top of the diagram can be calculated with lax [co]ends \cite{Hirata_2022_lax_end}.

\subsection{Lax doctrinal adjunction and lax algebra}\label{subsec:T-adj}

We fix a 2-monad $(T,m,e)$ on $\K$.
The word \emph{doctrinal adjunction} was introduced by Kelly in his paper \cite{Kelly_1974_doctrinal_adjunction}.
He showed that if there is an adjunction $f\dashv g$
between $T$-algebras in a 2-category $\K$,
the bijection of mates induces a bijective correspondence between
a 2-cell $\bar{g}$ making $(g,\bar{g})$ a lax morphism
and a 2-cell $\bar{f}$ which makes $(f,\bar{f})$ a colax morphism.

In this section, we define \emph{lax doctrinal adjunction}
as an adjunction $X\inlineadj A$ whose codomain $A$ of the left adjoint is a $T$-algebra.
This time, we require no $T$-algebra structure on $X$,
but there is always a canonical lax $T$-algebra structure on $X$ induced by the left adjoint.
We can also show that the left adjoint has a canonical oplax morphism structure,
and the right adjoint has a canonical lax morphism structure.
This section's goal is to prove the coreflective embedding theorem (\Cref{thm:coreflective_embedding_thm}),
a generalization of the embedding $\mnd_l(\K) \rightarrow \TT{\Two, \K}$.
Let $U\colon \alg{s}{T}\rightarrow \K$ be the forgetful 2-functor.

\begin{lemma}\label{lem:adj_induce_lax_alg}
  Let $\alY = (Y, y, \zeta^2, \zeta^0)$ be a lax algebra.
  Then, any adjunction $l\colon X \inlineadj Y\colon r$ in the base 2-category $\K$
  induces a lax algebra structure on $X$.
\end{lemma}
\begin{proof}
  We define the structure map $x$ by $TX\xrightarrow{Tl}TY\xrightarrow{y}Y\xrightarrow{r}X$,
  and $\xi^2$, $\xi^0$ as follows.
  \[
    \begin{tikzpicture}[baseline={(0,0)}]
      \node(uL) at (-5,3) {$T^2X$};
      \node(ul) at (-3,3) {$T^2Y$};
      \node(ur) at (-1,3) {$TY$};
      \node(uR) at ( 1,3) {$TX$};
      \node(m) at ( 1,1.5) {$TY$};
      \node(dL) at (-5,0) {$TX$};
      \node(dl) at (-3,0) {$TY$};
      \node(dR) at ( 1,0) {$Y$};
      \node(R)  at ( 3,0) {$X$};
      \draw[1cell,swap] (uL) to node {$m$} (dL);
      \draw[1cell,swap] (ul) to node {$m$} (dl);
      \draw[1cell] (uR) to node {$Tl$} (m);
      \draw[1cell] (m) to node {$y$}  (dR);
      \draw[1cell] (uL) to node {$T^2l$} (ul);
      \draw[1cell] (ul) to node {$Ty$} (ur);
      \draw[1cell] (ur) to node {$Tr$} (uR);
      \draw[1cell] (dL) to node {$Tl$} (dl);
      \draw[1cell] (dl) to node {$y$} (dR);
      \draw[1cell equality, bend right] (ur) to (m);
      \draw[1cell] (dR) to node {$r$} (R);
      \twocelldl[near start]{0,2.4}{$T\varepsilon$};
      \twocelldl{-1.2,1.4}{$\zeta^2$};
    \end{tikzpicture}
    \quad
    \begin{tikzpicture}[baseline={(0,0)}]
      \node(ul) at (-4,3) {$X$};
      \node(ur) at (0,3) {$TX$};
      \node(ml) at (-2,2) {$Y$};
      \node(mr) at (0,2) {$TY$};
      \node(d) at (0,1) {$Y$};
      \node(bot) at (0,0) {$X$};
      \draw[1cell] (ul) to node {$e$} (ur);
      \draw[1cell] (ml) to node {$e$} (mr);
      \draw[1cell] (ul) to node [auto=false, fill=white] {$l$} (ml);
      \draw[1cell] (ur) to node {$Tl$} (mr);
      \draw[1cell equality, bend right] (ml) to (d);
      \draw[1cell equality, bend right] (ul) to (bot);
      \draw[1cell] (mr) to node{$y$} (d);
      \draw[1cell] (d) to node {$r$}(bot);
      \twocellur{-0.7,1.35}{$\zeta^0$};
      \twocellur{-1.8,0.9}{$\eta$};
    \end{tikzpicture}
  \]
  Then, one can check that the data above satisfies the axiom of lax algebras.
\end{proof}

Those adjunctions will be called lax doctrinal when $\mathbb{Y}$ is a strict algebra.

\begin{definition}
  A \emph{lax doctrinal adjunction} of a 2-monad $T$ consists of a pair of a strict algebra $\alA$
  and an adjunction $l\dashv r\colon U\alA\rightarrow X$ in the base 2-category $\K$.
  We denote a lax doctrinal adjunction as $l\colon X\rightleftarrows \alA\colon r$, assuming the arrow above is the left adjoint.
  We sometimes omit writing $r$.
  The left/right adjoint 1-cell is called lax doctrinal left/right adjoint.

  We define the 2-category $\lladj{T}$ of lax doctrinal left adjoint
  by the full sub-2-category of the comma 2-category $\Id_\K/U$.
  In other words, $\lladj{T}$ has the following data:
  \begin{itemize}
    \item An object is a lax doctrinal adjunction
          $l\colon X\rightleftarrows \alA$.
    \item A 1-cell is a pair $(f,\f)$ that commutes with the left adjoints in $\K$.
          Abusing the notation, we write this in the right below.
          \[
            \begin{tikzpicture}[baseline={(0,0.3)}]
              \node(l) at (0,-0.7) {$\alA$};
              \node(r) at (2,-0.7) {$\alB$};
              \node(dl) at (0,0) {$U\alA$};
              \node(dr) at (2,0) {$U\alB$};
              \node(ul) at (0,1) {$X$};
              \node(ur) at (2,1) {$Y$};
              \draw[1cell] (l) to node {$\f$} (r);
              \draw[1cell] (dl) to node {$U\f$} (dr);
              \draw[1cell] (ul) to node {$ f$} (ur);
              \draw[1cell,swap] (ul) to node {$l$} (dl);
              \draw[1cell] (ur) to node {$l'$} (dr);
            \end{tikzpicture}
            \qquad \qquad
            \begin{tikzpicture}[baseline={(0,0.3)}]
              \node(dl) at (0,0) {$\alA$};
              \node(dr) at (1.5,0) {$\alB$};
              \node(ul) at (0,1) {$X$};
              \node(ur) at (1.5,1) {$Y$};
              \draw[1cell,swap] (dl) to node {$\f$} (dr);
              \draw[1cell] (ul) to node {$ f$} (ur);
              \draw[1cell,swap] (ul) to node {$l$} (dl);
              \draw[1cell] (ur) to node {$l'$} (dr);
            \end{tikzpicture}
          \]
    \item A 2-cell is a pair of 2-cells that are compatible.
  \end{itemize}
\end{definition}

\begin{lemma}\label{lem:adj_induce_colax_mor_2-cell}
  Let $l\colon X\rightleftarrows\alA$ and $l'\colon Y\rightleftarrows\alB$ be lax doctrinal adjunctions,
  and let $\alX$ and $\alY$ be the induced lax algebras.
  Then, any 1-cell $(f,\f)\colon l \rightarrow l'$ of lax doctrinal adjunctions also induces
  a colax morphism $\alX\rightsquigarrow\alY$, and any 2-cell induces a 2-cell of the corresponding colax morphisms.
\end{lemma}
\begin{proof}
  Observing that $\f \colon \alA\rightarrow \alB$ can be thought of as a strict algebra of the 2-monad
  $\TL{\Two, T}$ on $\TL{\Two, \K}$,
  and the identity 2-cell $l'f = \f l$ and its mate $fr \Rightarrow r'\f$ forms an adjunction between $\f$ and $f$ in $\TL{\Two,\K}$,
  it follows from \Cref{lem:adj_induce_lax_alg} that there is an induced lax $\TL{\Two,T}$-algebra structure,
  that is, colax morphism structure on $f$.

  For 2-cells, it suffices to regard a 2-cell in $\alg{s}{T}$ as a strict algebra of the 2-monad $\TL{\Eye, T}$
  to make a similar proof work.
\end{proof}

For later use, let us describe the data of an induced colax morphism.
Given a 1-cell $(f,\f)\colon l\rightarrow l'$,
the 2-cell $\bar{f}$ of the induced colax morphism $(f,\bar{f})$ is the following.
\[
  \begin{tikzpicture}
    \node(lU) at (0,3) {$TX$};
    \node(lu) at (0,2) {$TA$};
    \node(ld) at (0,1) {$A$};
    \node(lD) at (0,0) {$X$};
    \node(rU) at (4,3) {$TY$};
    \node(ru) at (4,2) {$TB$};
    \node(rd) at (4,1) {$B$};
    \node(rD) at (4,0) {$Y$};
    \node(a) at (1.5,1) {$A$};
    \node(y) at (2.5,0) {$Y$};
    \draw[1cell] (lU) to node{$Tf$} (rU);
    \draw[1cell] (lu) to node{$T\f$} (ru);
    \draw[1cell] (a)  to node{$\f$} (rd);
    \draw[1cell] (lD) to node[swap]{$f$} (y);
    \draw[1cell] (lU) to node[swap]{$Tl$} (lu);
    \draw[1cell] (lu) to node[swap]{$a$} (ld);
    \draw[1cell] (ld) to node[swap]{$r$} (lD);
    \draw[1cell] (rU) to node{$Tl'$} (ru);
    \draw[1cell] (ru) to node{$b$} (rd);
    \draw[1cell] (rd) to node{$r'$} (rD);
    \draw[1cell] (lD) to node[swap]{$l$} (a);
    \draw[1cell] (y)  to node{$l'$} (rd);
    \draw[1cell equality] (ld) to (a);
    \draw[1cell equality] (y) to (rD);
    \twocellu[swap]{0.3,0.6}{$\varepsilon$};
    \twocellu{3.7,0.4}{$\eta'$};
  \end{tikzpicture}
\]

Combining \Cref{lem:adj_induce_lax_alg,lem:adj_induce_colax_mor_2-cell},
we define a 2-functor $\I\colon\lladj{T}\rightarrow\lalg{c}{T}$ sending a lax doctrinal adjunction to the induced lax algebra.

Let $l\colon X\rightleftarrows \alA\colon r$ and $l'\colon Y\rightleftarrows \alB\colon r'$ be lax doctrinal adjunctions.
If $(\g,g)$ is a 1-cell $r\rightarrow r'$ in $U/\Id_\K$,
then the identity 1-cell $r'\g = gr$ and its mate $lg\Rightarrow \g l$
forms an adjunction between $\g$ and $g$ in $\TO{\Two, \K}$.
Therefore, $g$ has a canonical structure of a lax morphism, which is the dual of \Cref{lem:adj_induce_colax_mor_2-cell}.

\begin{corollary}\label{cor:induced_doctrinal_adjunction}
  Let $l\colon X\rightleftarrows \alA\colon r$ be a lax doctrinal adjunction.
  Then, each of the following 2-cells makes $l$ a colax morphism and $r$ a lax morphism $\alA\rightsquigarrow\alX$.
  \[
    \begin{tikzpicture}
      \node(lU) at (0,3) {$TX$};
      \node(lu) at (0,2) {$TA$};
      \node(ld) at (0,1) {$A$};
      \node(lD) at (0,0) {$X$};
      \node(rU) at (3,3) {$TA$};
      \node(rD) at (3,0) {$A$};
      \draw[1cell] (lU) to node{$Tl$} (rU);
      \draw[1cell] (lD) to node[swap]{$l$} (rD);
      \draw[1cell] (lU) to node[swap]{$Tl$} (lu);
      \draw[1cell] (lu) to node[swap]{$a$} (ld);
      \draw[1cell] (ld) to node[swap]{$r$} (lD);
      \draw[1cell] (rU) to node{$a$} (rD);
      \draw[1cell equality] (ld) to (rD);
      \twocellu[swap]{0.6,0.4}{$\varepsilon$};
    \end{tikzpicture}
    \quad
    \begin{tikzpicture}
      \node(lU) at (0,3) {$TA$};
      \node(lD) at (0,0) {$A$};
      \node(rU) at (3,3) {$TX$};
      \node(ru) at (3,2) {$TA$};
      \node(rd) at (3,1) {$A$};
      \node(rD) at (3,0) {$X$};
      \draw[1cell] (lU) to node{$Tf$} (rU);
      \draw[1cell] (lD) to node[swap]{$f$} (rD);
      \draw[1cell] (lU) to node[swap]{$a$} (lD);
      \draw[1cell] (rU) to node{$Tl$} (ru);
      \draw[1cell] (ru) to node{$b$} (rd);
      \draw[1cell] (rd) to node{$r$} (rD);
      \draw[1cell equality] (lD) to (rd);
      \twocellu{2.4,0.4}{$\eta$};
    \end{tikzpicture}
  \]
\end{corollary}
\begin{proof}
  It follows from the commutative diagrams below.
  \[
    \begin{tikzpicture}
      \node(dl) at (0,0) {$\alA$};
      \node(dr) at (1,0) {$\alA$};
      \node(ul) at (0,1) {$X$};
      \node(ur) at (1,1) {$A$};
      \draw[1cell equality] (dl) to (dr);
      \draw[1cell] (ul) to node {$l$} (ur);
      \draw[1cell,swap] (ul) to node {$l$} (dl);
      \draw[1cell equality] (ur) to (dr);
    \end{tikzpicture}
    \qquad
    \begin{tikzpicture}
      \node(dl) at (0,0) {$\alA$};
      \node(dr) at (1,0) {$\alA$};
      \node(ul) at (0,1) {$A$};
      \node(ur) at (1,1) {$X$};
      \draw[1cell equality] (dl) to (dr);
      \draw[1cell] (ul) to node {$r$} (ur);
      \draw[1cell,swap] (dr) to node {$r$} (ur);
      \draw[1cell equality] (dl) to (ul);
    \end{tikzpicture}
  \]
\end{proof}

\begin{example}\label{eg:what_is_ldadj_for_lax_functor}
  If the 2-monad $T$ is $\Id_\K$,
  a strict algebra is an object in $\K$,
  and a lax algebra is a monad in $\K$.
  If we define $\mathrm{ladj}(\K)$ as a full sub-2-category of $\K^\rightarrow$ of left adjoints,
  then the 2-functor $\I\colon\lladj{\Id}\rightarrow\lalg{c}{\Id}$ is
  $\mathrm{ladj}(\K)\rightarrow\mnd_c(\K)$ sending an adjunction to the associated monad.

  More generally, for the case of $T_\mathbf{Fun}$, a lax doctrinal adjunction is a pair of
  a 2-functor $F\colon\A\rightarrow \K$ and a family of adjunctions $\widehat{F}A\inlineadj FA$ in $\K$.
  In this case, the 2-functor $\I\colon\lladj{T_\mathbf{Fun}}\rightarrow \LO{\A,\K}$ sends those pointwise adjunctions
  to the lax functors $\widehat{F}\colon\A\rightsquigarrow\K$,
  where $\widehat{F}f$ is defined by $\widehat{F}A\rightarrow FA \xrightarrow{Ff}FB\rightarrow\widehat{F}B$.
  This construction of a lax functor can be found in \cite{Street_1972_two_const_lax_func}.
\end{example}

As we saw in the example above, the 2-functor $\I$ is a generalization of mapping an adjunction to its associated monad.
In formal monad theory, if $\K$ admits certain limits (resp.\@ colimits), we could say the converse:
every monad arises from its Eilenberg-Moore (resp.\@ Kleisli) adjunction.
Now, recalling \Cref{thm:w-barw_classifier_induce_adjunction},
the statement can be viewed as follows: if $\K$ admits certain colimits and $T$ preserve them,
every lax algebra $\alX$ arises from the lax doctrinal adjunction between $\alX$ and its weak morphism classifier $\Q\X$.

By replacing Kleisli adjunctions with the lax doctrinal adjunctions between lax algebras and those weak morphism classifiers,
the following theorem generalizes the embedding $\mnd_c(\K)\hookrightarrow \mathrm{ladj}(\K)$ defined by Kleisli adjunctions.

\begin{theorem}\label{thm:coreflective_embedding_thm}
  If $\K$ admits and $T$ preserves the $(l,r)$-codescent objects,
  $\I\colon\lladj{T}\rightarrow\lalg{c}{T}$ has the fully-faithful left adjoint $\QQ$,
  which maps a lax algebra $\alX$ to the adjunction between $\alX$ and its colax morphism classifier $\Q \alX$.
\end{theorem}
\begin{proof}
  \Cref{thm:w-barw_classifier_induce_adjunction} shows that $p^\alX\colon X\rightleftarrows\Q X\colon q^\alX$
  is a lax doctrinal adjunction for each lax algebra $\alX$.
  If $(f,\bar{f})$ is an oplax morphism $(f,\bar{f})\colon \alX\rightsquigarrow\alY$,
  the strict morphism $\Q\f\colon\Q \alX\rightarrow\Q \alY$ is defined by the unique map obtained from the universality of $\Q\alX$.
  \[
    \begin{tikzpicture}[baseline=(m.base)]
      \node(L) at (-2, 0) {$T^2\alX$};
      \node(ul) at (-1, 1.7) {$T\alX$};
      \node(ur) at (1, 1.7) {$T\alY$};
      \node(R) at (2, 0) {$\Q\alY$};
      \node(dl) at (-1,-1.7) {$T\alX$};
      \node(dr) at (1,-1.7) {$T\alY$};
      \node(m) at (0,0) {$T^2\alY$};
      \draw[1cell] (L) to node {$Tx$} (ul);
      \draw[1cell] (ul) to node {$Tf$}(ur);
      \draw[1cell] (ur) to node {$\g$} (R);
      \draw[1cell,swap] (L) to node {$m$} (dl);
      \draw[1cell,swap] (dl) to node {$Tf$} (dr);
      \draw[1cell,swap] (dr) to node {$\g$}(R);
      \draw[1cell,swap] (L) to node {$T^2f$} (m);
      \draw[1cell,swap,near end] (m) to node {$Ty$} (ur);
      \draw[1cell,swap] (m) to node {$m$} (dr);
      \twocelld{-0.5,0.8}{$T\bar{f}$}
      \twocelld{1,0}{$\bar{g}$}
    \end{tikzpicture}
    \quad=\quad
    \begin{tikzpicture}[baseline=(m.base)]
      \node(L) at (-2, 0) {$T^2\alX$};
      \node(ul) at (-1, 1.7) {$T\alX$};
      \node(ur) at (1, 1.7) {$T\alY$};
      \node(R) at (2, 0) {$\Q\alY$};
      \node(dl) at (-1,-1.7) {$T\alX$};
      \node(dr) at (1,-1.7) {$T\alY$};
      \node(m) at (0,0) {$\Q\alX$};
      \draw[1cell] (L) to node {$Tx$} (ul);
      \draw[1cell] (ul) to node {$Tf$}(ur);
      \draw[1cell] (ur) to node {$\g$} (R);
      \draw[1cell,swap] (L) to node {$m$} (dl);
      \draw[1cell,swap] (dl) to node {$Tf$} (dr);
      \draw[1cell,swap] (dr) to node {$\g$}(R);
      \draw[1cell] (ul) to node {$\g$} (m);
      \draw[1cell,swap] (dl) to node {$\g$} (m);
      \draw[1cell,dashed] (m) to node {$\Q\f$} (R);
      \twocelld{-1,0}{$\bar{g}$}
    \end{tikzpicture}
  \]
  The diagram below shows that $\QQ(f,\bar{f})\coloneqq (f,\f)$ is a morphism of lax doctrinal adjunctions.
  \[
    \begin{tikzpicture}
      \node(l)  at (0,2) {$X$};
      \node(r)  at (2,2) {$Y$};
      \node(ul) at (0,1) {$TX$};
      \node(ur) at (2,1) {$TY$};
      \node(dl) at (0,0) {$\Q X$};
      \node(dr) at (2,0) {$\Q Y$};
      \draw[1cell] (l) to node {$f$} (r);
      \draw[1cell] (ul) to node {$Tf$} (ur);
      \draw[1cell] (dl) to node {$\Q\f$} (dr);
      \draw[1cell,swap,bend right=50] (l) to node {$p^\alX$} (dl);
      \draw[1cell] (l) to node {$e$} (ul);
      \draw[1cell] (ul) to node {$g$} (dl);
      \draw[1cell,bend left=50] (r) to node {$p^\alY$} (dr);
      \draw[1cell,swap] (r) to node {$e$} (ur);
      \draw[1cell,swap] (ur) to node {$g$} (dr);
    \end{tikzpicture}
  \]
  If $\alpha\colon (f,\bar{f})\rightarrow (f',\bar{f}')$ is a 2-cell in $\lalg{c}{T}$,
  then one can show that $g.T\alpha$ defines a modification between cocones,
  so the 2-dimensional universality of the codescent object shows that there exists a 2-cell $\Q\alpha\colon\Q\f\Rightarrow\Q\f'$
  satisfying $\Q\alpha.g = g.T\alpha$, making $\QQ\alpha\coloneqq (\alpha,\Q\alpha)$ a 2-cell of $\lladj{T}$.

  The above correspondence defines a 2-functor $\QQ\colon\lalg{c}{T}\rightarrow\lladj{T}$.
  We next show that the composite $\I\QQ$ is the identity.
  To this end, let us first calculate the data of $\I\QQ\alX$ for a lax algebra $\alX$.
  The structure 1-cell of $\I\QQ\alX$ is defined by $q.\Q x.Tg.Te$.
  Since $\g$ is a strict morphism $T\alX\rightarrow\Q\alX$, we have $\Q x.Tg = g.m$, therefore $q.\Q x.Tg.Te = q.g.m.Te = x$.
  Let $\eta$ and $\varepsilon$ be the unit and the counit of the lax doctrinal adjunction $X\rightleftarrows \Q\alX$.
  The two coherence 2-cells are defined by
  $\eta$ and $q.\Q x.T\varepsilon.T\Q x.T^2 p$.
  By definition, $\eta$ is equal to $\xi^0$,
  so what remains to prove is that the latter is equal to $\xi^2$.
  Since there is equality of 1-cells $T\Q x.T^2p = T\Q x.T^2g.T^2e = Tg.Tm.T^2e = Tg$
  and equality of 2-cells $\varepsilon.g = \bar{g}.eT$,
  it suffices to show $q.\Q x.T\bar{g}.TeT$ is equal to $\xi^2$.
  Because $\bar{g}$ is a 2-cell in $\alg{s}{T}$,
  we have $\Q x.T\bar{g} = \bar{g}. mT$.
  Therefore, $q.\Q x.T\bar{g}.TeT = q.\bar{g}. mT.TeT = \xi^2$.

  Let $(f,\bar{f})$ be a colax morphism $\alX\rightsquigarrow \alY$
  and $\f \coloneqq \Q (f,\bar{f})$ be the strict morphism $\Q\alX\rightarrow \Q\alY$.
  The 2-cell of the colax morphism $\I\QQ(f,\bar{f})$ is as follows.
  \[
    \begin{tikzpicture}
      \node(lU) at (0,3) {$TX$};
      \node(lu) at (0,2) {$\Q A$};
      \node(lD) at (0,0) {$X$};
      \node(rU) at (6,3) {$TY$};
      \node(ru) at (6,2) {$\Q Y$};
      \node(rD) at (6,0) {$Y$};
      \node(TX) at (1.5,1) {$T X$};
      \node(Q) at (3,2) {$\Q X$};
      \node(TY) at (4.5,1) {$TY$};
      \node(Y) at (3,0) {$Y$};
      \draw[1cell] (lU) to node{$Tf$} (rU);
      \draw[1cell equality] (lu) to (Q);
      \draw[1cell] (Q) to node{$\f$} (ru);
      \draw[1cell] (lD) to node[swap]{$f$} (Y);
      \draw[1cell equality] (Y) to (rD);
      \draw[1cell] (TX) to node{$Tf$} (TY);
      \draw[1cell] (lU) to node[swap]{$g$} (lu);
      \draw[1cell,swap] (lu) to node {$q$} (lD);
      \draw[1cell] (rU) to node{$g$} (ru);
      \draw[1cell] (ru) to node{$q$} (rD);
      \draw[1cell] (lD) to node[swap]{$e$} (TX);
      \draw[1cell] (TX) to node{$g$} (Q);
      \draw[1cell] (Y) to node {$e$} (TY);
      \draw[1cell] (TY) to node[swap]{$g$} (ru);
      \twocellu[swap]{0.6,1.2}{$\varepsilon$};
      \twocellu{5.3,0.6}{$\eta$};
    \end{tikzpicture}
  \]
  We can calculate the left part of this 2-cell as follows.
  \begin{align*}
      &
    \begin{tikzpicture}[baseline={(0,-0.1)}]
      \node(L) at (-2.5,0) {$TX$};
      \node(l) at (-0.9,0) {$\Q X$};
      \node(r) at (0.9,0) {$\Q X$};
      \node(R) at (2.5,0) {$\Q Y$};
      \draw[1cell] (L) to node {$g$} (l);
      \draw[1cell,bend left] (l) to node {$pq$} (r);
      \draw[1cell equality, bend right] (l) to (r);
      \draw[1cell] (r) to node {$\f$} (R);
      \twocelld{-0.05,0}{$\varepsilon$}
    \end{tikzpicture}
    \ =\
    \begin{tikzpicture}[baseline={(0,-0.1)}]
      \node(L) at (-2.5,0) {$TX$};
      \node(l) at (-0.9,0) {$T^2X$};
      \node(r) at (0.9,0) {$\Q X$};
      \node(R) at (2.5,0) {$\Q Y$};
      \draw[1cell] (L) to node {$eT$} (l);
      \draw[1cell,bend left] (l) to node {$g.Tx$} (r);
      \draw[1cell,bend right,swap] (l) to node {$g.m$} (r);
      \draw[1cell] (r) to node {$\f$} (R);
      \twocelld{-0.05,0}{$\bar{g}$}
    \end{tikzpicture}
    \\
    = & \
    \begin{tikzpicture}[baseline={(0,0.6)}]
      \node(ur) at (1.7,1.5) {$\Q Y$};
      \node(um) at (0,1.5) {$TY$};
      \node(ul) at (-1.7,1.5) {$TX$};
      \node(dr) at (1.7,0) {$\Q Y$};
      \node(dm) at (0,0) {$T^2Y$};
      \node(dl) at (-1.7,0) {$T^2X$};
      \node(L) at (-3.4,0) {$TX$};
      \draw[1cell,swap] (dr) to node {$g$} (ur);
      \draw[1cell] (dl) to node {$Tx$} (ul);
      \draw[1cell] (um) to node {$g$} (ur);
      \draw[1cell] (ul) to node {$Tf$} (um);
      \draw[1cell,swap] (dl) to node {$T^2f$} (dm);
      \draw[1cell,swap] (dm) to node {$m$} (dr);
      \draw[1cell] (dm) to node[auto=false,fill=white]{$Ty$} (um);
      \draw[1cell] (L) to node[swap]{$eT$} (dl);
      \twocelld{-1.1,0.75}{$T\bar{f}$}
      \twocelld{ 0.8,0.75}{$\bar{g}$}
    \end{tikzpicture}
    \ =\
    \begin{tikzpicture}[baseline={(0,0.6)}]
      \node(UL) at (-3.4,1.5) {$X$};
      \node(ul) at (-1.7,1.5) {$Y$};
      \node(um) at (0,1.5) {$TY$};
      \node(ur) at (1.7,1.5) {$\Q Y$};
      \node(DL) at (-3.4,0) {$TX$};
      \node(dl) at (-1.7,0) {$TY$};
      \node(dm) at (0,0) {$T^2Y$};
      \node(dr) at (1.7,0) {$TY$};
      \draw[1cell] (DL) to node{$x$} (UL);
      \draw[1cell] (dl) to node [auto=false,fill=white]{$y$} (ul);
      \draw[1cell] (dm) to node[auto=false,fill=white]{$Ty$} (um);
      \draw[1cell,swap] (dr) to node {$g$} (ur);
      \draw[1cell] (UL) to node {$f$} (ul);
      \draw[1cell] (ul) to node {$e$} (um);
      \draw[1cell] (um) to node {$g$} (ur);
      \draw[1cell,swap] (DL) to node {$Tf$} (dl);
      \draw[1cell,swap] (dl) to node {$e$} (dm);
      \draw[1cell,swap] (dm) to node {$m$} (dr);
      \twocelld{-2.7,0.75}{$\bar{f}$}
      \twocelld{ 0.8,0.75}{$\bar{g}$}
    \end{tikzpicture}
  \end{align*}
  Therefore, the entire 2-cell is equal to $\bar{f}$ by the following.
  \[
    \begin{tikzpicture}[baseline={(0,1)}]
      \node(UL) at (-3,1.5) {$X$};
      \node(ul) at (-1.5,1.5) {$Y$};
      \node(um) at (0,1.5) {$TY$};
      \node(ur) at (1.5,1.5) {$\Q Y$};
      \node(DL) at (-3,0) {$TX$};
      \node(dl) at (-1.5,0) {$TY$};
      \node(dm) at (0,0) {$T^2Y$};
      \node(dr) at (1.5,0) {$TY$};
      \node(R) at (2,2.5) {$Y$};
      \draw[1cell] (DL) to node{$x$} (UL);
      \draw[1cell] (dl) to node [auto=false,fill=white]{$y$} (ul);
      \draw[1cell] (dm) to node[auto=false,fill=white]{$Ty$} (um);
      \draw[1cell,swap] (dr) to node {$g$} (ur);
      \draw[1cell] (UL) to node {$f$} (ul);
      \draw[1cell] (ul) to node {$e$} (um);
      \draw[1cell] (um) to node {$g$} (ur);
      \draw[1cell,swap] (DL) to node {$Tf$} (dl);
      \draw[1cell,swap] (dl) to node {$e$} (dm);
      \draw[1cell,swap] (dm) to node {$m$} (dr);
      \draw[1cell,swap] (ur) to node {$q$} (R);
      \draw[1cell equality,bend left] (ul) to (R);
      \twocelld{-2.4,0.75}{$\bar{f}$}
      \twocelld{ 0.7,0.75}{$\bar{g}$}
      \twocelld{0.5,2.2}{$\eta$};
    \end{tikzpicture}
    \ =\
    \begin{tikzpicture}[baseline={(0,1)}]
      \node(UL) at (-3,1.5) {$X$};
      \node(ul) at (-1.5,1.5) {$Y$};
      \node(um) at (0,1.5) {$TY$};
      \node(ur) at (2,2.5) {$Y$};
      \node(DL) at (-3,0) {$TX$};
      \node(dl) at (-1.5,0) {$TY$};
      \node(dm) at (0,0) {$T^2Y$};
      \node(dr) at (1.5,0) {$TY$};
      \draw[1cell] (DL) to node{$x$} (UL);
      \draw[1cell] (dl) to node [auto=false,fill=white]{$y$} (ul);
      \draw[1cell] (dm) to node[auto=false,fill=white]{$Ty$} (um);
      \draw[1cell,swap] (dr) to node {$y$} (ur);
      \draw[1cell] (UL) to node {$f$} (ul);
      \draw[1cell] (ul) to node {$e$} (um);
      \draw[1cell,swap] (um) to node {$y$} (ur);
      \draw[1cell,swap] (DL) to node {$Tf$} (dl);
      \draw[1cell,swap] (dl) to node {$e$} (dm);
      \draw[1cell,swap] (dm) to node {$m$} (dr);
      \draw[1cell equality,bend left] (ul) to (ur);
      \twocelld{-2.4,0.75}{$\bar{f}$}
      \twocelld{ 0.8,0.9}{$\xi^2$}
      \twocelld{0.5,2.2}{$\xi^0$};
    \end{tikzpicture}
    \ =\
    \begin{tikzpicture}[baseline={(0,0.7)}]
      \node(UL) at (-3,1.5) {$X$};
      \node(ul) at (-1.5,1.5) {$Y$};
      \node(DL) at (-3,0) {$TX$};
      \node(dl) at (-1.5,0) {$TY$};
      \draw[1cell] (DL) to node{$x$} (UL);
      \draw[1cell,swap] (dl) to node {$y$} (ul);
      \draw[1cell] (UL) to node {$f$} (ul);
      \draw[1cell,swap] (DL) to node {$Tf$} (dl);
      \twocelld{-2.4,0.75}{$\bar{f}$}
    \end{tikzpicture}
  \]
  For 2-cells, $\I\QQ\alpha = \I(\alpha,\bar{\alpha}) = \alpha$ by definition.

  Finally, we show that $\QQ$ is left adjoint to $\I$ with the identity unit.
  In \Cref{cor:induced_doctrinal_adjunction}, we observed the existence of a canonical colax morphism $(l,\bar{l})\colon \alX\rightarrow\alA$ for each lax doctrinal adjunction $l\colon X\rightleftarrows \alA$.
  From the universality of the colax morphism classifier,
  we obtain a strict morphism $\morl\colon \Q\alX \rightarrow \alA$ satisfying $l = \morl\,p$.
  \[
    \begin{tikzpicture}
      \node(dl) at (0,0) {$\Q\alX$};
      \node(dr) at (1,0) {$\alA$};
      \node(ul) at (0,1) {$X$};
      \node(ur) at (1,1) {$X$};
      \draw[1cell equality] (ul) to (ur);
      \draw[1cell,swap] (dl) to node {$\morl$} (dr);
      \draw[1cell] (ur) to node {$l$} (dr);
      \draw[1cell,swap] (ul) to node{$p$} (dl);
    \end{tikzpicture}
  \]
  To complete the proof, it suffices to show this $(1_X, \morl)$ defines the counit,
  that is, $\I (1_X,\morl) = 1_\alX$ and
  $\morl$ is the identity for $p\colon\alX\rightleftarrows\Q \alX$.
  Since the strict morphism $\morl$ for $p\colon\alX\rightleftarrows\Q \alX$ is defined by the unique map
  satisfying $\morl.\bar{g} = \Q x.T(\varepsilon.\Q x.Tp)$,
  the equality $\Q x.T(\varepsilon.\Q x.Tp) = \Q x.T(\varepsilon.g) = \Q x.T\bar{g}.TeT = \bar{g}.mT.TeT = \bar{g}$
  proves that this $\morl$ is the identity.
  To see $\I(1_X,\morl) = 1$, we need to check that the following is the identity.
  \[
    \begin{tikzpicture}
      \node(lU) at (0,3) {$TX$};
      \node(ld) at (0,1) {$\Q X$};
      \node(lD) at (0,0) {$X$};
      \node(rU) at (4,3) {$TX$};
      \node(ru) at (4,2) {$TA$};
      \node(rd) at (4,1) {$A$};
      \node(rD) at (4,0) {$X$};
      \node(a) at (2,1) {$\Q X$};
      \node(y) at (2,0) {$X$};
      \draw[1cell equality] (lU) to (rU);
      \draw[1cell] (a)  to node{$\morl$} (rd);
      \draw[1cell equality] (lD) to (y);
      \draw[1cell] (lU) to node[swap]{$g$} (ld);
      \draw[1cell] (ld) to node[swap]{$q$} (lD);
      \draw[1cell] (rU) to node{$Tl$} (ru);
      \draw[1cell] (ru) to node{$a$} (rd);
      \draw[1cell] (rd) to node{$r$} (rD);
      \draw[1cell] (lD) to node[auto=false,fill=white,swap]{$p$} (a);
      \draw[1cell] (y)  to node[auto=false,fill=white]{$l$} (rd);
      \draw[1cell equality] (ld) to (a);
      \draw[1cell equality] (y) to (rD);
      \twocellu[swap]{0.3,0.6}{$\varepsilon'$};
      \twocellu{3.7,0.4}{$\eta$};
    \end{tikzpicture}
  \]
  By the construction of $\morl$, it satisfies $\morl.\bar{g} = a.T\bar{l}$.
  So we can calculate as $\morl.\varepsilon'.g = \morl.\bar{g}.eT = a.T\bar{l}.eT = a.e.\bar{l} = \bar{l}$.
  Therefore, the 2-cell of the colax morphism $\I(1_X,\morl)$ is
  the vertical composition of $q.g.\eta = \eta.r.a.Tl$ and $r.\bar{l} = r.\varepsilon.a.Tl$, which is the identity.
\end{proof}

We also have the oplax counterpart of \Cref{thm:coreflective_embedding_thm}:
\begin{itemize}
  \item
        If $\K$ admits and $T$ preserves $(c,l)$-codescent objects,
        $\calg{l}{T}$ is a coreflective sub-2-category of
        \emph{colax doctrinal adjunctions} $r\colon\alA\leftrightarrows X$,
        which is a full sub-2-category of $U/\Id_\K$.
  \item If $\K$ admits and $T$ preserves $(p,p)$-codescent objects,
        $\palg{p}{T}$ is a coreflective sub-2-category of
        \emph{pseudo doctrinal adjunctions} $X\eqv\alA$.
\end{itemize}

Notice that even if $\alX$ is a pseudo algebra,
the adjunction between its colax morphism classifier $\Q_c\alX$
does not need to be an equivalence.

In \cite{Lack_2002_cod_obj_coh}, Lack declared that for a 2-monad $T$,
``the coherence theorem for pseudo-$T$-algebras'' should have the following statement.
\begin{quote}
  \textbf{Theorem-Schema.}
  \textit{
    The inclusion $\alg{s}{T} \rightarrow \palg{p}{T}$ has a left adjoint,
    and the component of the unit are equivalences in $\palg{p}{T}$.
  }
\end{quote}
For lax algebras, we expect that the statement in \Cref{thm:w-barw_classifier_induce_adjunction,thm:coreflective_embedding_thm}
should be the coherence condition of a 2-monad $T$ for lax-$T$-algebras.
That is,

\begin{schema}
  The inclusion $\alg{s}{T}\rightarrow\lalg{c}{T}$ has a left adjoint $\Q$,
  and the component of the unit have right adjoints in $\K$.
  Thus, there is a 2-functor $\QQ\colon \lalg{c}{T} \rightarrow \lladj{T}$,
  and it defines a fully-faithful left 2-adjoint of $\I\colon\lladj{T}\rightarrow\lalg{c}{T}$.
\end{schema}

\begin{example}
  Let $\A$ be a small 2-category and $\K$ be a cocomplete 2-category.
  As we observed in \Cref{eg:what_is_ldadj_for_lax_functor},
  $\lladj{T^\mathbf{Fun}}$ is a 2-category whose objects are
  2-functors $F\colon\A\rightarrow \K$ with adjunctions $\widehat{F}A\inlineadj FA$ in $\K$.
  \Cref{thm:coreflective_embedding_thm} shows that $\TO{\A,\K}$ is a coreflective sub-2-category of $\lladj{T^\mathbf{Fun}}$.

  Dually, if $\A$ is a small 2-category and $\K$ is a complete 2-category,
  $\LL{\A,\K}$ is a reflective sub-2-category of
  2-functors $F\colon\A\rightarrow\K$ with right adjoints
  $FA\rightarrow \widehat{F}A$ for each $A\in\A$.
  The reflection is given by the lax morphism coclassifier $\R_l$.
\end{example}

%

\section{Gray tensor product and distributive law}\label{sec:Gray_dist}
Our goal in this section is to generalize distributive laws
of monads to lax functors between 2-categories
and prove the counterpart of Beck's theorem of distributive laws.

\subsection{Gray tensor product}\label{subsec:Gray_tensor_prod}

Let $\alpha\colon K\rightarrow H$ be a lax transformation between lax functors $K,H\colon\A\rightarrow\B$.
If $F\colon \X \rightarrow\A$ is another lax functor, a composition $\alpha F$ defines a lax transformation.
However, as we saw in \Cref{rem:ww_is_not_2-functorial_in_first_variable}, post-composition $G\alpha$ of a lax functor $G\colon \B \rightarrow\Y$
does not always define a lax transformation
since the 2-cells in the following diagram are not composable.
\[
  \begin{tikzpicture}
    \node(GKA) at (0,2) {$GKA$};
    \node(GKA') at (0,0) {$GKA'$};
    \node(GHA) at (4,2) {$GHA$};
    \node(GHA') at (4,0) {$GHA'$};
    \draw[1cell,swap] (GKA) to node {$GKf$} (GKA');
    \draw[1cell] (GHA) to node {$GHf$} (GHA');
    \draw[1cell] (GKA) to node {$G\alpha_A$} (GHA);
    \draw[1cell,swap] (GKA') to node {$G\alpha_{A'}$} (GHA');
    \draw[1cell,bend right=17,swap,sloped] (GKA) to node [auto=false,fill=white] {$G(\alpha_{A'}.K\!f)$} (GHA');
    \draw[1cell,bend left=17, sloped] (GKA)      to node [auto=false,fill=white] {$G(H\!f.\alpha_A)$} (GHA');
    \twocelldl[]{3.1,1.5}{$G^2$}
    \twocellur[swap]{0.9,0.5}{$G^2$}
    \twocelldl[]{1.85,1.05}{$G\alpha_f$}
  \end{tikzpicture}
\]
One of the sufficient conditions for $F\alpha$ to become a lax transformation $GK\rightarrow GH$ is
that $G$ is a 2-functor, which says that two 2-cells $G^2$ in the diagram above are both identities.

In \cite{Gray_1974_formal_caty_theor}, Gray first extended the category $\TwoCat_0$ to a multicategory
with objects as 2-categories and arrows $\A_1,\dots,\A_n\rightarrow \B$ as 2-functors
in $\TL{\A_n,\TL{\dots,\TL{\A_1,\B}\dots}}$.
(He named those 2-functors as \emph{quasi-functors}
and denoted the hom-sets by $\text{q-fun}(\prod_{i} \A_i, \B)$,
which would be better to be written in modern multicategorical notation such as $\TL{\A_1,\dots,\A_n; \B}$.)
He then showed that the composite $\alpha F$ and $G\alpha$ above extends to a 2-functor
$\TL{\B,\C}\rightarrow\TL{\TL{\A,\B},\TL{\A,\C}}$, and hence $\TL{-,-}$ makes $\TwoCat_0$ a closed category.
Moreover, he finally showed that there is a tensor product $\A \otimes_l \B$ of 2-categories,
which we call a \emph{lax Gray tensor product}, making $(\TwoCat_0,\otimes_l,\One,\TL{-,-})$ a non-symmetric monoidal biclosed category.


Now, let us review the construction of $\A\otimes_l\B$.
\begin{definition}
  Let $\A$ and $\B$ be 2-categories.
  We first define a computad $\cG$ by the following data.
  \begin{itemize}
    \item A node of the quiver ${\left|\cG\right|}$ is a pair of objects $A\in\A$ and $B\in\B$ denoted by $A\otimes B$.
    \item An edge of ${\left|\cG\right|}$ has the form of either $f\otimes B \colon A\otimes B \rightarrow A'\otimes B$ or
          $A\otimes g \colon A\otimes B \rightarrow A\otimes B'$, where $f\colon A\rightarrow A'$ is a 1-cell of $\A$, and $g\colon B\rightarrow B'$ is a 1-cell of $\B$.
    \item There are three types of 2-cells in $\cG$:
          \begin{itemize}
            \item for each $\alpha\colon f\Rightarrow f'\in\A$ and $B\in\B$, a 2-cell $\alpha\otimes B \colon f\otimes B \Rightarrow f'\otimes B$,
            \item for each $\beta\colon g\Rightarrow g'\in\B$ and $A\in\A$, a 2-cell $A\otimes\beta \colon A\otimes g \Rightarrow A\otimes g'$,
            \item for each $f\colon A\rightarrow A'$ in $\A$ and $g\colon B\rightarrow B'$ in $\B$, a \emph{swapping 2-cell} $\gamma_{f,g}$.
                  \[
                    \begin{tikzpicture}
                      \node(s) at (0,2) {$A\otimes B$};
                      \node(u) at (3,2) {$A\otimes B'$};
                      \node(d) at (0,0) {$A'\otimes B$};
                      \node(t) at (3,0) {$A'\otimes B'$};
                      \draw[1cell] (s) to node {A$\otimes g$} (u);
                      \draw[1cell, swap] (s) to node {$f\otimes B$} (d);
                      \draw[1cell] (u) to node {$f\otimes B'$} (t);
                      \draw[1cell, swap] (d) to node {$A'\otimes g$} (t);
                      \twocelld{1.5,1}{$\gamma_{f,g}$}
                    \end{tikzpicture}
                  \]
          \end{itemize}
  \end{itemize}
  Then the \emph{lax Gray tensor product} $\A\otimes\B$ is the quotient of the 2-category generated from $\cG$
  with the following equalities.
  \begin{description}
    \item[1-cell id] $\id_A\otimes B = A\otimes \id_B = 1_{A\otimes B}$.
    \item[1-cell comp] $(f'\otimes B)(f\otimes B) = f'f\otimes B$ and $(A\otimes g')(A\otimes g) = A\otimes g'g$.
    \item[2-cell horizontal] $(\alpha'\otimes B)(\alpha\otimes B) = \alpha'\alpha\otimes B$
      and $(A\otimes\beta')(A\otimes\beta) = A\otimes \beta'\beta$.
    \item[2-cell vertical id] $1_f\otimes B = 1_{f\otimes B}$ and $A\otimes 1_g = 1_{A\otimes g}$.
    \item[2-cell vertical comp]
      $(\alpha'\otimes B)*(\alpha\otimes B) = \alpha'*\alpha\otimes B$
      and $(A\otimes\beta')*(A\otimes\beta) = A\otimes \beta'\beta$.
    \item[2-cell swap id] $\gamma_{1_A,g} = 1_{A\otimes g}$ and $\gamma_{f, 1_B} = 1_{f\otimes B}$.
    \item[2-cell swap comp]
      \[
        \begin{tikzpicture}[baseline={(2,2)}]
          \node(ab) at (0,4) {$A\otimes B$};
          \node(ab') at (3,4) {$A\otimes B'$};
          \node(a'b) at (0,2) {$A'\otimes B$};
          \node(a'b') at (3,2) {$A'\otimes B'$};
          \node(a''b) at (0,0) {$A''\otimes B$};
          \node(a''b') at (3,0) {$A''\otimes B'$};
          \draw[1cell] (ab) to node {$A\otimes g$} (ab');
          \draw[1cell] (a'b) to node {$A'\otimes g$} (a'b');
          \draw[1cell] (a''b) to node {$A''\otimes g$} (a''b');
          \draw[1cell, swap] (ab) to node {$f\otimes B$} (a'b);
          \draw[1cell, swap] (a'b) to node {$f'\otimes B$} (a''b);
          \draw[1cell] (ab') to node {$f\otimes B'$} (a'b');
          \draw[1cell] (a'b') to node {$f'\otimes B'$} (a''b');
          \twocelld{1.5,3}{$\gamma_{f,g}$};
          \twocelld{1.5,1}{$\gamma_{f',g}$};
        \end{tikzpicture}
        =
        \begin{tikzpicture}[baseline={(2,2)}]
          \node(ab) at (0,4) {$A\otimes B$};
          \node(ab') at (3,4) {$A\otimes B'$};
          \node(a''b) at (0,0) {$A''\otimes B$};
          \node(a''b') at (3,0) {$A''\otimes B'$};
          \draw[1cell] (ab) to node {$A\otimes g$} (ab');
          \draw[1cell] (a''b) to node {$A''\otimes g$} (a''b');
          \draw[1cell, swap] (ab) to node {$f'.f\otimes B$} (a''b);
          \draw[1cell] (ab') to node {$f'.f\otimes B'$} (a''b');
          \twocelld{1.5,2}{$\gamma_{f'.f,g}$};
        \end{tikzpicture}
      \]
      and
      \[
        \begin{tikzpicture}[baseline={(0,1)}]
          \node(ab) at (0,2) {$A\otimes B$};
          \node(ab') at (2.5,2) {$A\otimes B'$};
          \node(ab'') at (5,2) {$A\otimes B''$};
          \node(a'b) at (0,0) {$A'\otimes B$};
          \node(a'b') at (2.5,0) {$A'\otimes B'$};
          \node(a'b'') at (5,0) {$A'\otimes B''$};
          \draw[1cell] (ab) to node {$A\otimes g$} (ab');
          \draw[1cell] (ab') to node {$A\otimes g'$} (ab'');
          \draw[1cell] (a'b) to node {$A'\otimes g$} (a'b');
          \draw[1cell] (a'b') to node {$A'\otimes g'$} (a'b'');
          \draw[1cell,swap] (ab) to node {$f\otimes B$} (a'b);
          \draw[1cell] (ab') to node [above,fill=white,pos=0.7] {$f\otimes B'$} (a'b');,
          \draw[1cell] (ab'') to node {$f\otimes B''$} (a'b'');
          \twocelld[swap]{1.25,1}{$\gamma_{f,g}$};
          \twocelld{3.75,1}{$\gamma_{f,g'}$};
        \end{tikzpicture}
        =
        \begin{tikzpicture}[baseline={(0,1)}]
          \node(ab) at (0,2) {$A\otimes B$};
          \node(ab'') at (5,2) {$A\otimes B''$};
          \node(a'b) at (0,0) {$A'\otimes B$};
          \node(a'b'') at (5,0) {$A'\otimes B''$};
          \draw[1cell] (ab) to node {$A\otimes g'.g$} (ab'');
          \draw[1cell] (a'b) to node {$A'\otimes g'.g$} (a'b'');
          \draw[1cell,swap] (ab) to node {$f\otimes B$} (a'b);
          \draw[1cell] (ab'') to node {$f\otimes B''$} (a'b'');
          \twocelld{2.5,1}{$\gamma_{f,g'.g}$};
        \end{tikzpicture}
      \]
    \item[2-cell swap commute]
      \[
        \begin{tikzpicture}[baseline={(2,1.5)}]
          \node(ab) at (0,3) {$A\otimes B$};
          \node(ab') at (4,3) {$A\otimes B'$};
          \node(a'b) at (0,0) {$A'\otimes B$};
          \node(a'b') at (4,0) {$A'\otimes B'$};
          \draw[1cell, bend left=15] (ab) to node {$A\otimes g$} (ab');
          \draw[1cell, bend right=15, swap] (ab) to node {$A\otimes g'$} (ab');
          \draw[1cell, bend left=30] (ab') to node {$f\otimes B'$} (a'b');
          \draw[1cell, bend right=30, swap] (ab') to node {$f'\otimes B'$} (a'b');
          \draw[1cell, bend right=15, swap] (a'b) to node {$A'\otimes g'$} (a'b');
          \draw[1cell, bend right=30, swap] (ab) to node {$f'\otimes B$} (a'b);
          \twocelldl[swap]{1.6,1.1}{$\gamma_{f',g'}$}
          \twocelld{2,3}{$A\otimes \beta$};
          \twocelll{4,1.5}{$\alpha \otimes B'$};
        \end{tikzpicture}
        =
        \begin{tikzpicture}[baseline={(2,1.5)}]
          \node(ab) at (0,3) {$A\otimes B$};
          \node(ab') at (4,3) {$A\otimes B'$};
          \node(a'b) at (0,0) {$A'\otimes B$};
          \node(a'b') at (4,0) {$A'\otimes B'$};
          \draw[1cell, bend left=15] (ab) to node {$A\otimes g$} (ab');
          \draw[1cell, bend left=30] (ab') to node {$f\otimes B'$} (a'b');
          \draw[1cell, bend right=15, swap] (a'b) to node {$A'\otimes g$} (a'b');
          \draw[1cell, bend right=30, swap] (ab) to node {$f'\otimes B$} (a'b);
          \draw[1cell, bend left=30] (ab) to node {$f\otimes B$} (a'b);
          \draw[1cell, bend left=15] (a'b) to node {$A'\otimes g$} (a'b');
          \twocelldl{2.4,1.8}{$\gamma_{f,g}$}
          \twocelld{2,0}{$A'\otimes \beta$};
          \twocelll{0,1.5}{$\alpha\otimes B$};
        \end{tikzpicture}
      \]
  \end{description}

  This tensor product $\A \otimes_l \B$ lifts to a functor $\otimes_l\colon \TwoCat_0\times\TwoCat_0\rightarrow\TwoCat_0$ and
  defines two adjunctions,
  \begin{align*}
    \A\otimes_l(-)   & \quad\dashv\quad \TL{\A,-}  \\
    (-) \otimes_l \B & \quad\dashv\quad \TO{\B,-},
  \end{align*}
  making $\TwoCat_0$ a monoidal biclosed category.
\end{definition}

The adjunction $\A\otimes_l(-)~\dashv~\TL{\A,-}$ provides
a natural bijection $\TwoCat_0(\A\otimes_l\B, \K) \iso \TwoCat_0(\B,\TL{\A,\K})$.
This bijection extends to an isomorphism of 2-categories
$\TL{\A\otimes_l\B,\K}\iso\TL{\B,\TL{\A,\K}}$,
which is a direct consequence of the monoidal left closed structure.
Here, we list some useful isomorphisms of 2-categories.

\begin{proposition}\label{prop:Gray_tensor_property}
  Let $\A,\B,\K$ be 2-categories.
  We have the following isomorphisms natural in all variables
  as objects in $\TwoCat_0$.
  \begin{enumerate}
    \setlength\itemsep{3pt}
    \item $\TL{\A\otimes_l\B,\K} \iso \TL{\B,\TL{\A,\K}}$
    \item $\TO{\A\otimes_l\B,\K} \iso \TO{\A,\TO{\B,\K}}$
    \item $\TO{\B,\TL{\A,\K}} \iso \TL{\A,\TO{\B,\K}}$
    \item ${\left(\A\otimes_l\B\right)}^\op \iso \B^\op\otimes_l\A^\op$
    \item ${\left(\A\otimes_l\B\right)}^\co \iso \B^\co\otimes_l\A^\co$
  \end{enumerate}
\end{proposition}
\begin{proof}
  (i)(ii)(iii) are trivial from the monoidal biclosed structure on $\TwoCat_0$.
  For (iv), it can be calculated as
  $\TL{{\left(\A\otimes_l\B\right)}^\op, \K} \iso {\TO{\A\otimes_l\B, \K^\op}}^\op \iso {\TO{\A,\TO{\B,\K^\op}}}^\op \iso \TL{\A^\op,\TL{\B^\op,\K}} \iso \TL{\B^\op\otimes_l\A^\op, \K}$.
  So it follows $\TwoCat_0\left({\left(\A\otimes_l\B\right)}^\op, \K\right)\iso\TwoCat_0\left(\B^\op\otimes_l\A^\op, \K\right)$.
  Since this is natural in $\K\in\TwoCat_0$, the Yoneda lemma proves (iv).
  (v) can also be shown similarly.
\end{proof}

More strongly, Gray also showed that the isomorphisms in \Cref{prop:Gray_tensor_property} are 2-natural in $\K$ as an object in $\TwoCat$.
Therefore, from (i), we can say that
$\TL{-,-}$ is a strong monoidal functor from $(\TwoCat^\op_0,\otimes_l)$ to $(\TT{\TwoCat, \TwoCat}_0, \circ)$,
where the latter is the category of 2-functors and 2-natural transformations with a monoidal structure defined by the composition.

Each 1-cell $h = (A_n\otimes g_n)(f_n\otimes B_{n-1})\dots(A_1\otimes g_1)(f_1\otimes B_0)$
in $\A\otimes_l\B$ has two composable lists of 1-cells $f_1,\dots,f_n \in\A$ and
$g_1,\dots,g_n\in\B$.
Since, as the name indicates, the swapping 2-cell $\gamma$ swaps a pair of 1-cells from $\A$ and $\B$,
there are canonical 2-cells $\hat{\gamma}_h\colon h\Rightarrow (A_n\otimes g_n\dots g_1)(f_n\dots f_1\otimes B_0)$
and $\tilde{\gamma}_h\colon (f_n\dots f_1\otimes B_n)(A_n\otimes g_n\dots g_1)\Rightarrow h$,
both defined as composites of the swapping 2-cells $\gamma$.
\[
  \begin{tikzpicture}
    \node(s) at (0,8.2) {$A_0\otimes B_0$};
    \node(t) at (5.7,3) {$A_n\otimes B_n$};
    \node(l) at (0,3) {$A_n\otimes B_0$};
    \node(r) at (5.7,8.2) {$A_0\otimes B_n$};
    \node(10) at (0,7) {$\cdot$};
    \node(11) at (1,7) {$\cdot$};
    \node(21) at (1,6) {$\cdot$};
    \node(22) at (2,6) {$\cdot$};
    \node[labelsize] (h) at (2.4,6.5) {$h$};
    \node(dot) at (2.5,5.6) {$\ddots$};
    \node(33) at (3,5) {$\cdot$};
    \node(43) at (3,4) {$\cdot$};
    \node(44) at (4,4) {$\cdot$};
    \node(54) at (4,3) {$\cdot$};
    \draw[1cell] (s) to node{} (10);
    \draw[1cell] (10) to node{} (11);
    \draw[1cell] (11) to node{} (21);
    \draw[1cell] (21) to node{} (22);
    \draw[1cell] (33) to node{} (43);
    \draw[1cell] (43) to node{} (44);
    \draw[1cell] (44) to node{} (54);
    \draw[1cell] (54) to node{} (t);
    \draw[1cell,swap] (s) to node{$f_n\dots f_1\otimes B_0$} (l);
    \draw[1cell] (s) to node{$A_0\otimes g_n\dots g_1$} (r);
    \draw[1cell] (r) to node{$f_n\dots f_1\otimes B_n$} (t);
    \draw[1cell,swap] (l) to node{$A_n\otimes g_n\dots g_1$} (t);
    \twocelldl{4.5,6.5}{$\tilde{\gamma}_h$};
    \twocelldl{1.5,4.5}{$\hat{\gamma}_h$};
  \end{tikzpicture}
\]

The rules \textbf{2-cell swap id} and \textbf{2-cell swap comp} imply that
$\hat{\gamma}_h$ and $\tilde{\gamma}_h$
are unique and well-defined for each 1-cell $h$.

A lax Gray tensor product $\A\otimes_l\B$ has
canonical comparison maps between the Cartesian product $\A\times\B$.
One is a 2-functor $\rho = (\rho_1,\rho_2)\colon\A\otimes_l\B\rightarrow\A\times\B$
sending $A\otimes B$ to $(A,B)$ that identifies the swapping 2-cells $\gamma$.
This 2-functor $\rho$ can also be seen as a representation of the inclusion 2-functor
$\left[\B,\left[\A,\K\right]\right]\rightarrow\TL{\B,\TL{\A,\K}}$.
The others are normal lax and normal oplax functors of the opposite direction
$\sigma_l, \sigma_{op}\colon\A\times\B \rightsquigarrow \A\otimes_l\B$.
The normal lax functor $\sigma_l$ sends $(f,g)\colon (A,B)\rightarrow (A',B')$ to $(A'\otimes g)(f\otimes B)$,
where the comparison map $\sigma_l^2$ for the composition of $(f,g)$ and $(f',g')$
is defined by the following $\hat{\gamma}$.
\[
  \begin{tikzpicture}
    \node(a) at (0,4) {$A\otimes B$};
    \node(b) at (6,0) {$A''\otimes B''$};
    \node(s) at (0,2) {$A'\otimes B$};
    \node(u) at (3,2) {$A'\otimes B'$};
    \node(d) at (0,0) {$A''\otimes B$};
    \node(t) at (3,0) {$A''\otimes B'$};
    \draw[1cell] (a) to node {$f \otimes B$} (s);
    \draw[1cell] (t) to node {$A'' \otimes g'$} (b);
    \draw[1cell] (s) to node {$A'\otimes g$} (u);
    \draw[1cell, swap] (s) to node {$f'\otimes B$} (d);
    \draw[1cell] (u) to node {$f'\otimes B'$} (t);
    \draw[1cell, swap] (d) to node {$A''\otimes g$} (t);
    \twocelldl{1.5,1}{$\gamma_{f',g}$}
  \end{tikzpicture}
\]
Dually, the normal oplax functor $\sigma_{op}$ sends $(f,g)\colon (A,B)\rightarrow (A',B')$
to $(f\otimes B')(A\otimes g)$,
and the comparison map $\sigma_{op}^2$ for the composition of $(f,g)$
and $(f',g')$ is defined by the following $\tilde{\gamma}$.
\[
  \begin{tikzpicture}
    \node(a) at (0,4) {$A\otimes B$};
    \node(s) at (3,4) {$A\otimes B'$};
    \node(u) at (6,4) {$A\otimes B''$};
    \node(d) at (3,2) {$A'\otimes B'$};
    \node(t) at (6,2) {$A'\otimes B''$};
    \node(b) at (6,0) {$A''\otimes B''$};
    \draw[1cell] (a) to node {$A \otimes g$} (s);
    \draw[1cell] (t) to node {$f' \otimes B''$} (b);
    \draw[1cell] (s) to node {$A \otimes g'$} (u);
    \draw[1cell, swap] (s) to node {$f\otimes B'$} (d);
    \draw[1cell] (u) to node {$f\otimes B''$} (t);
    \draw[1cell, swap] (d) to node {$A'\otimes g'$} (t);
    \twocelldl{4.5,3}{$\gamma_{f,g'}$}
  \end{tikzpicture}
\]
When $\rho$ is post-composed, $\sigma_l$ and $\sigma_{op}$ become
the identity 2-functor $\rho\sigma_l = \rho\sigma_{op} = \Id_{\A\times\B}$.
Conversely, though $\sigma_l\rho$ is not an identity,
there is an icon $\hat{\gamma}\colon\Id\rightarrow \sigma_l\rho$.
Since $\hat{\gamma}$ satisfies $\rho\hat{\gamma} = \id$ and $\hat{\gamma}\rho = \id$,
$\rho$ and $\sigma_l$ are adjoint with identity counit in
$\TwoCat^{\mathrm{Lax}}_{\mathrm{Icon}}$.

In particular, the presentation of Gray tensor products of the form $\barA\otimes \barB$
is well studied in \cite{Nikolic_2019_strict_tensor_prod}.

\begin{remark}
  The lax Gray tensor product does not extend to a structure of a monoidal 2-category
  on $\TwoCat$ or $\TwoCatTI$.
  First, $\TwoCat\left(\A\otimes_l\B,\K\right)$ and $\TwoCat\left(\B,\TL{\A,\K}\right)$ are not isomorphic in general,
  since if we take $\One$ for $\B$, it says every lax natural transformation between 2-functors $\A\rightarrow\K$ is 2-natural.
  Second, comparing the data of 1-cells in $\TwoCatTI(\A\otimes_l\B,\K)$ and $\TwoCatTI(\B,\TL{\A,\K})$,
  an icon $F\Rightarrow G\colon\A\otimes_l\B\rightarrow\K$ has the following two kinds of data of 2-cells,
  but an icon $F'\Rightarrow G'\colon \B\rightarrow\TL{\A,\K}$ requires the latter 2-cells to be identities.
  \[
    \begin{tikzpicture}
      \node(Q) at (-5,1.5) {$FAB$};
      \node(A) at (-5,0) {$FAB'$};
      \node(W) at (-2,1.5) {$GAB$};
      \node(S) at (-2,0) {$GAB'$};
      \node(E) at (2,1.5) {$FAB$};
      \node(D) at (2,0) {$FA'B$};
      \node(R) at (5,1.5) {$GAB$};
      \node(F) at (5,0) {$GA'B$};
      \draw[1cell equality] (Q) to (W);
      \draw[1cell equality] (A) to (S);
      \draw[1cell equality] (E) to (R);
      \draw[1cell equality] (D) to (F);
      \draw[1cell,swap] (Q) to node {$FAg$} (A);
      \draw[1cell] (W) to node {$GAg$} (S);
      \draw[1cell,swap] (E) to node {$FfB$} (D);
      \draw[1cell] (R) to node {$GfB$} (F);
      \twocellr{-3.5,0.75}{}
      \twocellr{3.5,0.75}{}
    \end{tikzpicture}
  \]

  And also, the monoidal structure does not have a natural extension to $\mathbf{2\text{-}Cat}^\mathrm{Lax}_0$.
  That is because,
  for each lax functor $F\colon \A\rightsquigarrow\A'$,
  we want $F\otimes_l \Id_\B \colon \A\otimes_l\B\rightsquigarrow \A'\otimes_l\B$ to
  map $A\otimes g$ to $FA\otimes g$ and $f\otimes B$ to $Ff\otimes B$,
  but
  $1_{FA\otimes B} = FA\otimes 1_B$ is not identical to $F1_A \otimes B$
  unless $F$ is a normal lax functor.
  Therefore, $F\otimes_l\Id_\K$ is not well-defined.
\end{remark}

\subsection{Generalized distributive law}

A distributive law was an object in
$\mnd_l(\mnd_l(\K)) \iso \LL{\One,\LL{\One, \K}} \iso \TL{\Dist, \K}$.
The previous section and \Cref{prop:lax_fun_classifier} show that this is also isomorphic to
$\TL{\overline{\One}\otimes_l\overline{\One}, \K}$,
which proves $\Dist \iso \overline{\One}\otimes_l\overline{\One}$.
As a generalization, we define a distributive law of lax functors
in the following way as in \cite{Nikolic_2019_strict_tensor_prod}.
A similar construction can also be found in~\cite{2dimensionalbifunctortheoremsdistributive}.

\begin{definition}
  A \emph{generalized distributive law} in $\K$ is a 2-functor $\barC\otimes\barD\rightarrow \K$.
\end{definition}

The next is immediate from \Cref{prop:Gray_tensor_property}.

\begin{proposition}\label{prop:generalized_dist_law}
  All the followings are equivalent notions.
  \begin{enumerate}
    \item A generalized distributive law $\Gamma\colon\barC\otimes\barD\rightarrow \K$.
    \item A lax functor $\D\rightsquigarrow \LL{\C,\K}$.
    \item A lax functor $\C\rightsquigarrow \LO{\D,\K}$.
  \end{enumerate}
\end{proposition}

Such a distributive law $\Gamma$
gives rise to lax functors $\Gamma(C,-) \colon \D \rightsquigarrow \K$ for each $C\in\C$
and $\Gamma(-,D) \colon \C \rightsquigarrow \K$ for each $D\in\D$,
where $\Gamma(C,-)$ is defined as the following.
\[
  \begin{tikzpicture}
    \node(a) at (0,0) {$\D$};
    \node(s) at (3,0) {$\barD \iso \One\otimes\barD$};
    \node(d) at (6,0) {$\barC\otimes \barD$};
    \node(f) at (8.3,0) {$\K$};
    \draw[1cell squig] (a) to node{$p$} (s);
    \draw[1cell] (s) to node{$C\otimes \Id$} (d);
    \draw[1cell] (d) to node{$\Gamma$} (f);
  \end{tikzpicture}
\]

Conversely, suppose that we are given two families of lax functors $F_C\colon \D \rightsquigarrow \K$
and $G_D\colon \C \rightsquigarrow \K$ such that $\Gamma_{CD} \coloneqq F_C(D) = G_D(C)$.
The remaining data needed to form a distributive law $\Gamma$ are 2-cells
${\left\{\gamma_{f,g} \colon G_{D'}(f).F_{C}(g) \Rightarrow F_{C'}(g).G_D(f)\right\}}_{f\colon C\rightarrow C', g\colon D\rightarrow D'}$
satisfying all the following.
\[
  \begin{tikzpicture}[baseline={(0,0)}]
    \node(ul) at (-1.2, 1) {$\Gamma_{CD}$};
    \node(ur) at ( 1.2, 1) {$\Gamma_{CD}$};
    \node(dl) at (-1.2,-1) {$\Gamma_{C'D}$};
    \node(dr) at ( 1.2,-1) {$\Gamma_{C'D}$};
    \draw[1cell equality,bend left=25] (ul) to (ur);
    \draw[1cell,swap,bend right=25] (ul) to node{$F_C1_D$} (ur);
    \draw[1cell,swap,bend right=25] (dl) to node{$F_{C'}1_D$} (dr);
    \draw[1cell,sloped,swap] (ul) to node{$G_Df$} (dl);
    \draw[1cell,sloped] (ur) to node{$G_Df$} (dr);
    \twocelldl{-0.3,-0.3}{$\gamma_{f,1_D}$}
    \twocelld{-0.1,1}{$F^0_C$}
  \end{tikzpicture}
  \,=\,
  \begin{tikzpicture}[baseline={(0,0)}]
    \node(ul) at (-1.2, 1) {$\Gamma_{CD}$};
    \node(ur) at ( 1.2, 1) {$\Gamma_{CD}$};
    \node(dl) at (-1.2,-1) {$\Gamma_{C'D}$};
    \node(dr) at ( 1.2,-1) {$\Gamma_{C'D}$};
    \draw[1cell equality,bend left] (ul) to (ur);
    \draw[1cell equality,bend left] (dl) to (dr);
    \draw[1cell,swap,bend right=25] (dl) to node{$F_{C'}1_D$} (dr);
    \draw[1cell,sloped,swap] (ul) to node{$G_Df$} (dl);
    \draw[1cell,sloped] (ur) to node{$G_Df$} (dr);
    \twocelld{-0.1,-1}{$F^0_{C'}$}
    \draw[double,double equal sign distance, shorten >=2pt, shorten <=3pt]
    (0.175,0.55) to (-0.15,0.225);
  \end{tikzpicture}
  \quad
  \begin{tikzpicture}[baseline={(0,0)}]
    \node(ul) at (-1, 1.2) {$\Gamma_{CD}$};
    \node(ur) at ( 1, 1.2) {$\Gamma_{CD}$};
    \node(dl) at (-1,-1.2) {$\Gamma_{CD'}$};
    \node(dr) at ( 1,-1.2) {$\Gamma_{CD'}$};
    \draw[1cell] (ul) to node{$F_Cg$} (ur);
    \draw[1cell,swap] (dl) to node{$F_Cg$} (dr);
    \draw[1cell,swap,bend right=25,sloped] (ul) to node{$G_D1_C$} (dl);
    \draw[1cell equality,bend left=25] (ur) to (dr);
    \draw[1cell,swap,bend right=25,sloped] (ur) to node{$G_D1_C$} (dr);
    \twocelldl[swap]{-0.1,-0.35}{$\gamma_{1_C,g}$}
    \twocelll[swap]{1,-0.1}{$G^0_C$}
  \end{tikzpicture}
  \,=\,
  \begin{tikzpicture}[baseline={(0,0)}]
    \node(ul) at (-1, 1.2) {$\Gamma_{CD}$};
    \node(ur) at ( 1, 1.2) {$\Gamma_{CD}$};
    \node(dl) at (-1,-1.2) {$\Gamma_{CD'}$};
    \node(dr) at ( 1,-1.2) {$\Gamma_{CD'}$};
    \draw[1cell] (ul) to node{$F_Cg$} (ur);
    \draw[1cell,swap] (dl) to node{$F_Cg$} (dr);
    \draw[1cell,bend right=25] (ul) to (dl);
    \draw[1cell,swap,bend right=25,sloped] (ul) to node{$G_D1_C$} (dl);
    \draw[1cell equality, bend left=25] (ur) to (dr);
    \draw[1cell equality, bend left=25] (ul) to (dl);
    \twocelll[swap]{-1,-0.1}{$G^0_{C'}$}
    \draw[double,double equal sign distance, shorten >=2pt, shorten <=3pt]
    (0.475,0.175) to (0.125,-0.175);
  \end{tikzpicture}
\]
\[
  \begin{tikzpicture}[baseline={(0,0)},x=9mm]
    \node(ul) at (-2, 1) {$\Gamma_{CD}$};
    \node(ur) at ( 2, 1) {$\Gamma_{CD''}$};
    \node(um) at ( 0, 1.7) {$\Gamma_{CD'}$};
    \node(dl) at (-2,-1) {$\Gamma_{C'D}$};
    \node(dr) at ( 2,-1) {$\Gamma_{C'D''}$};
    \draw[1cell,bend left=12] (ul) to node{$F_Cg$} (um);
    \draw[1cell,bend left=12] (um) to node{$F_Cg'$} (ur);
    \draw[1cell,swap,bend right=25] (ul) to node{$F_C(g'g)$} (ur);
    \draw[1cell,swap,bend right=25] (dl) to node{$F_{C'}(g'g)$} (dr);
    \draw[1cell,swap] (ul) to node{$G_Df$} (dl);
    \draw[1cell] (ur) to node{$G_{D''}f$} (dr);
    \twocelld[swap]{0.2,1}{$F^2_C$}
    \twocelldl[swap]{0,-0.8}{$\gamma_{f,g'g}$}
  \end{tikzpicture}
  \ =\
  \begin{tikzpicture}[baseline={(0,0)}]
    \node(ul) at (-2, 1) {$\Gamma_{CD}$};
    \node(ur) at ( 2, 1) {$\Gamma_{CD''}$};
    \node(um) at ( 0, 1.7) {$\Gamma_{CD'}$};
    \node(dl) at (-2,-1) {$\Gamma_{C'D}$};
    \node(dr) at ( 2,-1) {$\Gamma_{C'D''}$};
    \node(dm) at ( 0,-0.3) {$\Gamma_{C'D'}$};
    \draw[1cell,bend left=12] (ul) to node{$F_Cg$} (um);
    \draw[1cell,bend left=12] (um) to node{$F_Cg'$} (ur);
    \draw[1cell,bend left=12,sloped] (dl) to node{$F_{C'}g$} (dm);
    \draw[1cell,bend left=12,sloped] (dm) to node{$F_{C'}g$} (dr);
    \draw[1cell,swap,bend right=25] (dl) to node{$F_{C'}(g'g)$} (dr);
    \draw[1cell,swap] (ul) to node{$G_Df$} (dl);
    \draw[1cell] (ur) to node{$G_{D''}f$} (dr);
    \draw[1cell] (um) to node[auto=false,fill=white]{$G_{D'}f$} (dm);
    \twocelld[swap]{0.2,-1}{$F^2_{C'}$}
    \twocelldl[swap]{-0.9,0.4}{$\gamma$}
    \twocelldl[swap]{ 1.2,0.4}{$\gamma$}
  \end{tikzpicture}
\]
\[
  \begin{tikzpicture}[baseline={(0,0)},x=9mm]
    \node(ul) at (-1.5, 1.25) {$\Gamma_{CD}$};
    \node(ur) at ( 1.5, 1.25) {$\Gamma_{CD'}$};
    \node(mr) at ( 2.2,  0) [transform canvas={xshift=7pt}]{$\Gamma_{C'D'}$};
    \node(dl) at (-1.5,-1.25) {$\Gamma_{C''D}$};
    \node(dr) at ( 1.5,-1.25) {$\Gamma_{C''D'}$};
    \draw[1cell] (ul) to node{$F_Cg$} (ur);
    \draw[1cell,swap] (dl) to node{$F_{C''}g$} (dr);
    \draw[1cell,swap,bend right=35] (ul) to node{$G_{D}(f'f)$} (dl);
    \draw[1cell,bend left=12] (ur) to node{$G_{D'}f$} (mr);
    \draw[1cell,bend left=12] (mr) to node{$G_{D'}f'$} (dr);
    \draw[1cell,swap,bend right=35,pos=0.7] (ur) to node{$G_{D'}(f'f)$} (dr);
    \twocelldl[swap]{-0.6,0}{$\gamma_{f'f,g}$}
    \twocelll[swap]{1.5,-0.2}{$G^2_{C'}$}
  \end{tikzpicture}
  \qquad=\quad
  \begin{tikzpicture}[baseline={(0,0)},x=9mm]
    \node(ul) at (-1.5, 1.25) {$\Gamma_{CD}$};
    \node(ur) at ( 1.5, 1.25) {$\Gamma_{CD'}$};
    \node(mr) at ( 2.2,  0) {$\Gamma_{C'D'}$};
    \node(dl) at (-1.5,-1.25) {$\Gamma_{C''D}$};
    \node(dr) at ( 1.5,-1.25) {$\Gamma_{C''D'}$};
    \node(ml) at (-0.8, 0)  [transform canvas={xshift=2pt}] {$\Gamma_{C'D}$};
    \draw[1cell] (ul) to node{$F_Cg$} (ur);
    \draw[1cell,swap] (dl) to node{$F_{C''}g$} (dr);
    \draw[1cell,swap,bend right=35] (ul) to node{$G_{D}(f'f)$} (dl);
    \draw[1cell,bend left=12] (ur) to node{$G_{D'}f$} (mr);
    \draw[1cell,bend left=12] (mr) to node{$G_{D'}f'$} (dr);
    \draw[1cell,bend left=12,pos=0.6] (ul) to node{$G_{D}f$} (ml);
    \draw[1cell,bend left=12,pos=0.4] (ml) to node{$G_{D}f'$} (dl);
    \draw[1cell] (ml) to node[auto=false,fill=white]{$F_{C'}g$} (mr);
    \twocelll[swap]{-1.6,-0.2}{$G^2_C$}
    \twocelldl[swap]{0.8,0.65} {$\gamma$}
    \twocelldl[swap]{0.8,-0.65}{$\gamma$}
  \end{tikzpicture}
  \quad
\]
\[
  \begin{tikzpicture}[baseline={(2,1.5)}]
    \node(ab) at (0,3)  {$\Gamma_{C D}$};
    \node(ab') at (4,3) {$\Gamma_{C D'}$};
    \node(a'b) at (0,0) {$\Gamma_{C'D }$};
    \node(a'b') at (4,0){$\Gamma_{C'D'}$};
    \draw[1cell, bend left=15] (ab) to node {$F_Cg$} (ab');
    \draw[1cell, bend left=30] (ab') to node {$G_{D'}f$} (a'b');
    \draw[1cell, bend right=15, swap] (a'b) to node {$F_{C'}g'$} (a'b');
    \draw[1cell, bend right=30, swap] (ab) to node {$G_Df'$} (a'b);
    \draw[1cell, bend right=15, swap] (ab) to node {$F_Cg'$} (ab');
    \draw[1cell, bend right=30, swap] (ab') to node {$G_{D'}f'$} (a'b');
    \twocelldl[swap]{1.6,1.1}{$\gamma_{f',g'}$}
    \twocelld{2,3}{$F_C\beta$};
    \twocelll{4,1.5}{$G_{D'}\alpha$};
  \end{tikzpicture}
  =
  \begin{tikzpicture}[baseline={(2,1.5)}]
    \node(ab) at (0,3) {$A\otimes B$};
    \node(ab') at (4,3) {$A\otimes B'$};
    \node(a'b) at (0,0) {$A'\otimes B$};
    \node(a'b') at (4,0) {$A'\otimes B'$};
    \draw[1cell, bend left=15] (ab) to node {$F_Cg$} (ab');
    \draw[1cell, bend left=30] (ab') to node {$G_{D'}f$} (a'b');
    \draw[1cell, bend right=15, swap] (a'b) to node {$F_{C'}g'$} (a'b');
    \draw[1cell, bend right=30, swap] (ab) to node {$G_Df'$} (a'b);
    \draw[1cell, bend left=30] (ab) to node {$G_Df$} (a'b);
    \draw[1cell, bend left=15] (a'b) to node {$F_{C'}g$} (a'b');
    \twocelldl{2.4,1.8}{$\gamma_{f,g}$}
    \twocelld{2,0}{$F_{C'}\beta$};
    \twocelll{0,1.5}{$G_{D}\alpha$};
  \end{tikzpicture}
\]

One can observe that the first four axioms look similar to those of the original distributive law for monads.
The last axiom says the compatibility with 2-cells in $\C$ and $\D$, which were trivial in the original case $\C=\D=\One$.

\begin{example}
  Let $\D$ be the distribution monad on $\Set$
  and $\M$ be the finite multiset monad on $\Set$.
  We denote the set of $n$-element multisets over a set $X$ as $\M[n](X)$.
  Assuming the monoid $(\N,\times,1)$ is a discrete monoidal category,
  there is a lax monoidal functor $\M[-]\colon\N \rightarrow \mathbf{End}(\Set)$
  sending $n$ to $\M[n]$.
  The multiplication ${\M[-]}^2_{nm}\colon \M[n]\circ\M[m]\Rightarrow\M[mn]$ is defined by
  the union of the multisets.

  In \cite{Jacobs_2021_multiset_distribution},
  it was shown that there exists a distributive law $\mathrm{pml}\colon \M\D\Rightarrow\D\M$
  called \emph{parallel multinomial law} in the paper.
  To illustrate this $\mathrm{pml}$ intuitively,
  let us assume an element in $\D(X)$ as a jar containing a finite number of elements $x_1,\dots,x_n$ of $X$,
  each of which can be extracted with the probability $p_i$.
  Then, $\mathrm{pml}_X$ is the function that sends a multiset of those jars to
  the probability distribution of multisets resulting from picking an element out from each jar.

  Since the number of elements extracted is the same as the number of jars,
  $\mathrm{pml}$ can be restricted to $\mathrm{pml}^n\colon \M[n]\D\Rightarrow\D\M[n]$.
  This $\mathrm{pml}^n$ makes $\M[n]$ a colax monad morphism $(\M[n],\mathrm{pml}^n)\colon \D\rightarrow \D$,
  and one can show that the lax functor $\M[-]$ lifts
  to $\M[-]\colon\Sigma\N \rightsquigarrow \mathbf{Monad}_{c}(\Set); *\mapsto \D$,
  where $\mathbf{Monad}_c(\Set)$ is the 2-category of monads on $\Set$ and colax morphisms.
  Therefore, $\mathrm{pml}^{(-)}$ defines a generalized distributive law $\overline{\Sigma\N}\otimes_l\bar{\One} \rightarrow \Sigma\mathbf{End}(\Set) \subseteq \Cat$.
\end{example}

\begin{definition}
  A 2-functor of the form $\overline{\C^\co}^\co \otimes_l \barD\rightarrow \K$ will also be called a generalized distributive law.
\end{definition}

When $\C = \D = \One$, this 2-functor is a distributive law $wt \Rightarrow tw$ of a comonad $w$ over a monad $t$.
As in \Cref{prop:generalized_dist_law},
a generalized distributive law $\overline{\C^\co}^\co \otimes_l \barD\rightarrow \K$ is equivalent to
a lax functor $\D\rightsquigarrow \OO{\C,\K}$
or an oplax functor $\C\rightsquigarrow \LL{\D,\K}$.
And when $\C$ and $\D$ are monoidal categories, our distributive law
coincides with the $(\pi,\pi')$-distributive law in \cite{GKOBU_2016_combining_effect_coeffects_via_grading}.

Note that all the combination of distributive laws is presented by two patterns,
$\barC\otimes_l\barD$ and $\overline{\C^\co}^\co \otimes_l \barD$,
because of \Cref{prop:Gray_tensor_property} (v).

Let $\A$ be a 2-category and $T^\mathbf{Fun}$ be the 2-monad on $\TT{\ob(\A), \K}$.
Suppose a lax functor $F\colon\C\rightsquigarrow \lladj{T^\mathbf{Fun}}$ is given.
Such a lax functor consists of the following data.
\begin{itemize}
  \item A lax functor $\widehat{F}\colon\C \rightsquigarrow \TT{\A,\K}$.
  \item For each $A\in\A$, a lax functor $\dot{F}_A\colon\C\rightsquigarrow\K$.
  \item For each $A\in\A$ and $C\in\C$, an adjunction $\dot{F}_{A}C\inlineadj\widehat{F}(C,A)$ in $\K$
        such that the left adjoints defines a strict natural transformation between $\widehat{F}(-,A)$ and $\dot{F}_A$.
        \[
          \begin{tikzpicture}
            \node(dl) at (0,0) {$\widehat{F}(C,A)$};
            \node(dr) at (3,0) {$\widehat{F}(C',A)$};
            \node(ul) at (0,1.5) {$\dot{F}_{A}C$};
            \node(ur) at (3,1.5) {$\dot{F}_{A}C'$};
            \draw[1cell,swap] (dl) to node {$\widehat{F}(f,A)$} (dr);
            \draw[1cell] (ul) to node {$\dot{F}_{A}f$} (ur);
            \draw[1cell] (ul) to (dl);
            \draw[1cell] (ur) to (dr);
          \end{tikzpicture}
        \]
\end{itemize}
Then by post-composing the 2-functor $\I\colon\lladj{T^\mathbf{Fun}} \rightarrow \LO{\A, \K}$ to $F$,
we obtain a generalized distributive law $\barC\otimes_l\barA\rightarrow\K$.
Conversely, \Cref{thm:coreflective_embedding_thm} shows that a generalized distributive law
corresponds to a lax functor $F\colon\C\rightsquigarrow \lladj{T^\mathbf{Fun}}$
such that the adjunction $\dot{F}_{A}C\inlineadj\widehat{F}(C,A)$
is of the form $\dot{F}_{A}C\inlineadj \Q^\mathbf{Fun}_c\dot{F}_{A}C$
induced by the colax morphism classifier,
which is a generalization of the dual of \Cref{cor:fact_dist_is_lift_of_monad}.

\subsection{Composition via distributive law}
For use in this section, let us recall the following two adjunctions with identity counits in $\BiCatLI$,
which we proved just after \Cref{prop:lax_fun_classifier} and
at the end of \Cref{subsec:Gray_tensor_prod} for each.
\[
  \begin{tikzpicture}
    \node(ll) at (0,0) {$\C$};
    \node(lr) at (2,0) {$\barC$};
    \node(rl) at (5,0) {$\A\times \B$};
    \node(rr) at (7.5,0) {$\A\otimes \B$};
    \node[labelsize] at (1,0) {$\bot$};
    \node[labelsize] at (6.25,0) {$\bot$};
    \draw[1cell squig,swap,bend right] (ll) to node {$p$} (lr);
    \draw[1cell squig,swap,bend right] (lr) to node {$q$} (ll);
    \draw[1cell squig,swap,bend right] (rl) to node {$\sigma_l$} (rr);
    \draw[1cell,swap,bend right] (rr) to node {$\rho$} (rl);
  \end{tikzpicture}
\]
For each adjunction, we write the unit icons as $\eta^\A\colon 1\Rightarrow pq$ and
$\hat{\gamma}\colon 1 \Rightarrow \sigma_l\rho$.

\begin{lemma}\label{lem:pre-compose_laxfn_bo_retract_gene}
  Let $\A$ be a 2-category and $\C$ be a bicategory.
  Suppose that a bijective-on-objects lax functor $J\colon\C\rightsquigarrow \A$
  has a retraction $R\colon\A\rightsquigarrow\C$.
  Also, suppose that all the 1-cells in $\A$ are generated from the image of $J$.
  Then, the 2-functor $J^*\colon\Lw{w}{\A,\K} \rightarrow \Lw{w}{\C,\K}$ defined by the pre-composition of $J$ is 2-fully-faithful.
\end{lemma}
\begin{proof}
  The 2-functor $J^*$ is surjective on 1-cells since $\alpha RJ = \alpha$
  for each $w$-natural transformation $\alpha\colon FJ\rightarrow GJ\colon\C\rightsquigarrow\K$.

  Let $\alpha \colon F\Rightarrow G\colon \A\rightsquigarrow \K$ be a $w$-natural transformation.
  Since $J$ is bijective-on-objects, the family of 1-cells $\{\alpha_A\}$ is determined by $\{\alpha_{JC}\}$.
  Let $g$ be a 1-cell in $\A$.
  Then there exists a sequence $f_1,\dots,f_n$ of 1-cells of $\C$ such that $g = Jf_n\dots Jf_1$.
  Such a sequence is composable in $\C$ because $J$ is bijective-on-object.
  Since $\alpha_g$ is a composition of $\alpha_{Jf_i}$, the family of 2-cells ${\{\alpha_g\}}_{g\in\A}$
  is also determined by ${\{\alpha_{Jf}\}}_{f\in\C}$.
  Therefore, $J^*$ is injective on 1-cells.

  It is clear that $J^*$ is injective on objects.
  If $\alpha,\beta \colon F\Rightarrow G\colon \A\rightsquigarrow \K$ are $w$-natural transformations,
  a family $\{\Gamma_A\colon \alpha_A\Rightarrow \beta_A \}$ of 2-cells defines a modification
  if and only if it is compatible with $\alpha_{Jf}$ and $\beta_{Jf}$ for all 1-cells $f$ in $\C$.
  That is because $\alpha_g$ and $\beta_g$ are compositions of some $\alpha_{Jf}$ and $\beta_{Jf}$.
  Therefore, $J^*$ is also surjective on 2-cells.
\end{proof}

In \Cref{eg:T_2-Cat_is_2-monad}, we saw that the 2-category $\BiCatLI$ of bicategories and lax functors is
2-equivalent to $\palg{l}{T_\mathbf{\TwoCat}}$.
Also, in \cite{Lack_2010_2-caty_companion}, it was shown that there is another 2-monad $S_\mathbf{\TwoCat}$
whose 2-category $\alg{l}{S_\mathbf{\TwoCat}}$ is isomorphic to $\palg{l}{T_\mathbf{\TwoCat}}$.
In either point of view, from \cite{Lack_2005_limit_for_lax_morphism,LackShulman_2012_enhanced_limit_for_morphism},
we can conclude that the forgetful functor from $\BiCatLI$ to the 2-categories of $\Cat$-graphs lifts any oplax limits.
In particular, $\BiCatLI$ admits products.

Here, we would like to show that $\sigma_l\colon \A\times\B\rightarrow \A\otimes_l\B$
defines a lax monoidal structure on the inclusion $\iota\colon\TwoCat_0\rightarrow\BiCatLZ$,
and thus the left adjoint $\Q$ in \Cref{eg:T_2-cat_admits_classifier} extends to an oplax monoidal functor.
\[
  \begin{tikzpicture}
    \node(l) at (-2.3,0) {$(\TwoCat_0,\otimes_l)$};
    \node(r) at (2.3,0) {$(\BiCatLZ,\times)$};
    \node[labelsize] at (0,0.3) {$\bot$};
    \draw[1cell,swap] (l) to node {$\iota$} (r);
    \draw[1cell, bend right,swap] (r) to node {$\Q$} (l);
  \end{tikzpicture}
\]
Note that the unit and counit of this adjunction were $p$ and $q$.

\begin{proposition}
  $(\iota, \sigma_l)$ is a normal lax monoidal functor.
\end{proposition}
\begin{proof}
  Let $\A,\B,\E$ be 2-categories.
  Given a pair of 2-functors $F\colon\A\rightarrow \A'$ and $G\colon\B\rightarrow\B'$,
  one can easily calculate that the diagram below commutes,
  which shows that $\sigma_l\colon \A\otimes\B\rightsquigarrow\A\times\B$ is a natural transformation.
  \[
    \begin{tikzpicture}
      \node(ul) at (-1.2, 1) {$\A\times\B$};
      \node(dl) at (-1.2,0) {$\A'\times\B'$};
      \node(ur) at ( 1.2, 1) {$\A\otimes_l\B$};
      \node(dr) at ( 1.2,0) {$\A'\otimes_l\B'$};
      \draw[1cell,swap] (ul) to node{$F\times G$} (dl);
      \draw[1cell] (ur) to node{$F\otimes_l G$} (dr);
      \draw[1cell squig] (ul) to node{$\sigma_l$} (ur);
      \draw[1cell squig] (dl) to node{$\sigma_l$} (dr);
    \end{tikzpicture}
  \]
  The remaining is to show the axiom of lax functors.
  \[
    \begin{tikzpicture}
      \node(ul) at (-2,3) {$(\A\times\B)\times\E$};
      \node(ml) at (-2,1.5) {$(\A\otimes_l\B)\times\E$};
      \node(dl) at (-2,0) {$(\A\otimes_l\B)\otimes_l\E$};
      \node(ur) at (2,3) {$\A\times(\B\times\E)$};
      \node(mr) at (2,1.5) {$\A\times(\B\otimes_l\E)$};
      \node(dr) at (2,0) {$\A\otimes_l(\B\otimes_l\E)$};
      \draw[1cell squig,swap] (ul) to node{$\sigma_l\times 1$} (ml);
      \draw[1cell squig,swap] (ml) to node{$\sigma_l$} (dl);
      \draw[1cell squig] (ur) to node{$1\times \sigma_l$} (mr);
      \draw[1cell squig] (mr) to node{$\sigma_l$} (dr);
      \draw[1cell] (ul) to node {$\sim$} (ur);
      \draw[1cell] (dl) to node {$\sim$} (dr);
    \end{tikzpicture}
    \qquad
    \begin{tikzpicture}
      \node(ul) at (-1.2, 3) {$\A\times\One$};
      \node(d) at (0,2) {$\A\otimes_l\One$};
      \node(ur) at ( 1.2, 3) {$\A$};
      \draw[1cell,swap] (ul) to node{$\sigma_l$} (d);
      \draw[1cell,swap] (d) to node{$\sim$} (ur);
      \draw[1cell] (ul) to node{$\sim$} (ur);

      \node(Ul) at (-1.2, 1) {$\One\times\A$};
      \node(D) at (0,0) {$\One\otimes_l\A$};
      \node(Ur) at ( 1.2, 1) {$\A$};
      \draw[1cell,swap] (Ul) to node{$\sigma_l$} (D);
      \draw[1cell,swap] (D) to node{$\sim$} (Ur);
      \draw[1cell] (Ul) to node{$\sim$} (Ur);
    \end{tikzpicture}
  \]
  The two on the right are immediate.
  The most non-trivial part on the left is that the following comparison 2-cells
  of the lax functors coincide.
  \[
    \begin{tikzpicture}
      \coordinate(V) at (0,-1.5);
      \coordinate(H) at (1.5,0);
      \coordinate(T) at (1,0.6);
      \node(q) at (0,0) {$\cdot$}; 
      \node(a) at (V) {$\cdot$}; 
      \node(z) at ($(a)+(V)$) {$\cdot$}; 
      \node(s) at ($(a)+(H)$) {$\cdot$}; 
      \node(x) at ($(z)+(H)$) {$\cdot$}; 
      \node(c) at ($(x)+(H)$) {$\cdot$}; 
      \node(e) at ($(s)+(T)$) {$\cdot$}; 
      \node(d) at ($(e)+(V)$) {$\cdot$}; 
      \node(f) at ($(d)+(H)$) {$\cdot$}; 
      \node(t) at ($(f)+(T)$) {$\cdot$}; 
      \draw[1cell,auto=false, swap] (q) to node[fill=white]{$f$} (a);
      \draw[1cell,auto=false, swap] (a) to node[fill=white]{$f'$} (z);
      \draw[1cell,auto=false,swap]  (z) to node[fill=white]{$g$} (x);
      \draw[1cell,auto=false,swap]  (x) to node[fill=white]{$g'$} (c);
      \draw[1cell,auto=false]       (a) to node[fill=white]{$g$} (s);
      \draw[1cell,auto=false]       (e) to node[fill=white]{$f'$} (d);
      \draw[1cell,auto=false]       (d) to node[fill=white]{$g'$} (f);
      \draw[1cell,swap]             (c) to node[]{$h$} (f);
      \draw[1cell,swap]             (f) to node[]{$h'$} (t);
      \draw[1cell]                  (s) to node[]{$h$} (e);
      \draw[1cell] (s) to (x);
      \draw[dashed] (x) to (d);
      \twocelld{$(a)!0.5!(x)$}{}
      \twocelld{$(d)!0.25!(x)$}{}
    \end{tikzpicture}
    \qquad
    \begin{tikzpicture}
      \coordinate(V) at (0,-1.5);
      \coordinate(H) at (1.5,0);
      \coordinate(T) at (1,0.6);
      \node(q) at (0,0) {$\cdot$}; 
      \node(a) at (V) {$\cdot$}; 
      \node(z) at ($(a)+(V)$) {$\cdot$}; 
      \node(s) at ($(a)+(H)$) {$\cdot$}; 
      \node(x) at ($(z)+(H)$) {$\cdot$}; 
      \node(c) at ($(x)+(H)$) {$\cdot$}; 
      \node(e) at ($(s)+(T)$) {$\cdot$}; 
      \node(d) at ($(e)+(V)$) {$\cdot$}; 
      \node(f) at ($(d)+(H)$) {$\cdot$}; 
      \node(t) at ($(f)+(T)$) {$\cdot$}; 
      \draw[1cell,auto=false, swap] (q) to node[fill=white]{$f$} (a);
      \draw[1cell,auto=false, swap] (a) to node[fill=white]{$f'$} (z);
      \draw[1cell,auto=false,swap]  (z) to node[fill=white]{$g$} (x);
      \draw[1cell,auto=false,swap]  (x) to node[fill=white]{$g'$} (c);
      \draw[1cell,auto=false]       (a) to node[fill=white]{$g$} (s);
      \draw[1cell,auto=false]       (e) to node[fill=white]{$f'$} (d);
      \draw[1cell,auto=false]       (d) to node[fill=white]{$g'$} (f);
      \draw[1cell,swap]             (c) to node[]{$h$} (f);
      \draw[1cell,swap]             (f) to node[]{$h'$} (t);
      \draw[1cell]                  (s) to node[]{$h$} (e);
      \draw[1cell] (x) to (d);
      \draw[dashed] (s) to (x);
      \twocelldr{$(d)!0.5!(c)$}{}
      \twocelldl{$(s)!0.5!(x)$}{}
    \end{tikzpicture}
  \]
  But since both are the composition of $\gamma_{f,g}$, $\gamma_{f,h}$, and $\gamma_{g,h}$,
  these two are the same 2-cells.
\end{proof}

Since the right adjoint $\iota$ has a structure of a lax monoidal functor,
there is a structure of an oplax monoidal functor on the left adjoint $\Q$,
where the comparison map is defined by
\[
  \begin{tikzpicture}
    \node(a) at (0,0) {$\overline{\C\times\D}$};
    \node(s) at (3,0) {$\overline{\barC\times\barD}$};
    \node(d) at ($(s)+(s)$) {$\overline{\barC\otimes\barD}$};
    \node(f) at ($(d)+(s)$) {$\barC\otimes\barD.$};
    \draw[1cell] (a) to node {$\Q(p\times p)$}(s);
    \draw[1cell] (s) to node {$\Q\sigma_l$}(d);
    \draw[1cell] (d) to node {$q$}(f);
  \end{tikzpicture}
\]
Since $qp = 1$, this is the unique 2-functor at the bottom right.
\[
  \begin{tikzpicture}
    \node(a) at (0,0) {${\C\times\D}$};
    \node(s) at (2.5,0) {${\barC\times\barD}$};
    \node(d) at ($(s)+(s)$) {${\barC\otimes\barD}$};
    \node(X) at ($(s)+(0,-1)$) {$\overline{\C\times\D}$};
    \draw[1cell squig] (a) to node {$p\times p$}(s);
    \draw[1cell squig] (s) to node {$\sigma_l$}(d);
    \draw[1cell squig,swap] (a) to node {$p$}(X);
    \draw[1cell,swap] (X) to node {$ $}(d);
  \end{tikzpicture}
\]
\begin{theorem}
  For each bicategory $\C$,
  the lax functor classifier $\barC$ is a comonoid for lax Gray tensor product $\otimes_l$ where the comultiplication
  $\delta\colon\barC\rightarrow\barC\otimes_l\barC$ is the unique 2-functor such that $\delta p$ is equal to $\sigma_l\Delta p$.
  \[
    \begin{tikzpicture}
      \node(f) at (0,0) {$\C$};
      \node(a) at (1.5,0) {$\barC$};
      \node(s) at (3.5,0) {$\barC\times\barC$};
      \node(d) at (6,0) {$\barC\otimes\barC$};
      \draw[1cell squig] (f) to node {$p$}(a);
      \draw[1cell] (a) to node {$\Delta$}(s);
      \draw[1cell squig] (s) to node {$\sigma_l$}(d);
    \end{tikzpicture}
  \]
\end{theorem}
\begin{proof}
  Since an oplax monoidal functor preserves any comonoid,
  and every object $\C$ in the Cartesian monoidal category $\BiCatLZ$ has a unique comonoid structure,
  $\barC$ has a comonoid structure where the comultiplication is defined by
  \[
    \begin{tikzpicture}
      \node(q) at (-2.5,0) {$\barC$};
      \node(a) at (0,0) {$\overline{\C\times\C}$};
      \node(s) at (3,0) {$\overline{\barC\times\barC}$};
      \node(d) at ($(s)+(s)$) {$\overline{\barC\otimes\barC}$};
      \node(f) at ($(d)+(s)$) {$\barC\otimes\barC.$};
      \draw[1cell] (q) to node {$\Q \Delta$} (a);
      \draw[1cell] (a) to node {$\Q(p\times p)$}(s);
      \draw[1cell] (s) to node {$\Q\sigma_l$}(d);
      \draw[1cell] (d) to node {$q$}(f);
    \end{tikzpicture}
  \]
  When the unit $p$ is precomposed, this is equal to $\sigma_l(p\times p)\Delta = \sigma_l\Delta p$.
\end{proof}

Since $\TL{-,-}$ defines a strong monoidal functor
$(\TwoCat^\op_0,\otimes_l)\rightarrow(\TT{\TwoCat, \TwoCat}_0, \circ)$,
we conclude the following.

\begin{corollary}
  For each small bicategory $\C$, $\TL{\barC, -}\colon \TwoCat\rightarrow\TwoCat$ defines a 2-monad.
  In particular, if $\C = \One$, this is the 2-monad $\mnd_l(-)$ on $\TwoCat$.
\end{corollary}

The multiplication
$\TL{\Dist,\K} \iso \mnd_l(\mnd_l(\K)) \rightarrow \mnd_l(\K)$
of the monad $\mnd_l(-)$ was given by the composition of monads.
Therefore, we define a generalization of the composition of monads to lax functors as follows.

\begin{definition}
  Given a distributive law $\Gamma\colon \barC\otimes\barC\rightarrow \K$,
  we call the following lax functor the \emph{composition via $\Gamma$}.
  \[
    \begin{tikzpicture}
      \node(f) at (0,0) {$\C$};
      \node(a) at (1.5,0) {$\barC$};
      \node(s) at (3.5,0) {$\barC\times\barC$};
      \node(d) at (6,0) {$\barC\otimes\barC$};
      \node(p) at (8.5,0) {$\barC\otimes\barC$};
      \draw[1cell squig] (f) to node {$p$}(a);
      \draw[1cell] (a) to node {$\Delta$}(s);
      \draw[1cell squig] (s) to node {$\sigma_l$}(d);
      \draw[1cell] (d) to node {$\Gamma$}(p);
    \end{tikzpicture}
  \]
\end{definition}

As in the remark after \Cref{prop:generalized_dist_law},
a distributive law $\barC\otimes_l\barC\rightarrow\K$ can be seen as families of lax functors $\Gamma(C,-),\Gamma(-,C)\colon\C\rightsquigarrow\K$
together with 2-cells
\[
  \begin{tikzpicture}
    \node(dl) at (0,0) {$\Gamma_{C_1'C_2}$};
    \node(dr) at (3,0) {$\Gamma_{C_1'C_2'}$};
    \node(ul) at (0,2) {$\Gamma_{C_1C_2}$};
    \node(ur) at (3,2) {$\Gamma_{C_1C_2'}$};
    \draw[1cell]      (ul) to node {$\Gamma(C_1,g)$} (ur);
    \draw[1cell,swap] (dl) to node {$\Gamma(C_1',g)$} (dr);
    \draw[1cell,swap] (ul) to node {$\Gamma(f,C_2)$} (dl);
    \draw[1cell] (ur) to node {$\Gamma(f,C_2')$} (dr);
    \twocelldl{1.5,1}{$\gamma_{f,g}$}
  \end{tikzpicture}
\]
satisfying some coherence conditions.
Let $H\colon\C\rightsquigarrow\K$ be the composed lax functor via the distributive law.
Unpacking the definition of $H$,
we observe that $H$ maps $C$ to $\Gamma(C,C)$ and  $f\colon C\rightarrow C'$ to $\Gamma(C',g)\Gamma(f,C)$,
and the multiplication $H^2$ is defined by the following 2-cell.
\[
  \begin{tikzpicture}
    \node(a) at (0,4) {$\Gamma_{C_1C_1}$};
    \node(b) at (6,0) {$\Gamma_{C_3C_3}$};
    \node(s) at (0,2) {$\Gamma_{C_2C_1}$};
    \node(u) at (3,2) {$\Gamma_{C_2C_2}$};
    \node(d) at (0,0) {$\Gamma_{C_3C_1}$};
    \node(t) at (3,0) {$\Gamma_{C_3C_2}$};
    \draw[1cell] (a) to node {$\Gamma(f_1,C_1)$} (s);
    \draw[1cell] (t) to node {$\Gamma(C_3,f_2)$} (b);
    \draw[1cell] (s) to node {$\Gamma(C_2,f_1)$} (u);
    \draw[1cell,pos=0.6,sloped] (s) to node {$\Gamma(f_2,C_1)$} (d);
    \draw[1cell] (u) to node {$\Gamma(f_2,C_2)$} (t);
    \draw[1cell] (d) to node {$\Gamma(C_3,f_1)$} (t);
    \draw[1cell,bend right=90,swap] (a) to node {$\Gamma(f_2f_1,C_1)$} (d);
    \draw[1cell,bend right,swap] (d) to node {$\Gamma(C_3,f_2f_1)$} (b);
    \twocelldl{1.8,1.2}{$\gamma_{f_1,f_2}$}
    \twocelll[swap]{-0.9,2.3}{${\Gamma(-,C_1)}^2$}
    \twocelld{2.7,-0.6}{${\Gamma(C_3,-)}^2$}
  \end{tikzpicture}
\]

Let $\A$ be a small 2-category.
Recall that the lax morphism coclassifier $\R F \colon \A\rightarrow\K$ of a lax functor $F\colon\A\rightsquigarrow\K$
was a generalization of Eilenberg-Moore objects,
and $\R$ defines the right adjoint of the inclusion $\TT{\A,\K} \rightarrow \LL{\A,\K}$.

\begin{proposition}\label{prop:AB_EM_on_EM}
  Let $\Gamma\colon \barA\otimes_l\barB\rightarrow \K$ be a generalized distributive law
  and $\Gamma'\colon \B\rightsquigarrow \LL{\A,\K}$ be the corresponding lax functor.
  The following 2-functors from $\A\times \B$ to $\K$ are all isomorphic.
  \begin{enumerate}
    \item The transpose of $\B\xrightarrow{\R \Gamma'} \LL{\A,\K}\xrightarrow{\R} \TT{\A,\K}$.
    \item The transpose of the coclassifier of
          $\begin{tikzpicture}[baseline={(0,-0.1)}]
              \node(l) at (-1.5,0) {$\B$};
              \node(m) at (0,0) {$\LL{\A,\K}$};
              \node(r) at (2,0) {$\TT{\A,\K}$};
              \draw[1cell squig] (l) to node{$\Gamma'$} (m);
              \draw[1cell] (m) to node{$\R$} (r);
            \end{tikzpicture}$.
    \item The coclassifier of
          $\begin{tikzpicture}[baseline={(0,-0.1)}]
              \node(l) at (-2,0) {$\A\times\B$};
              \node(m) at (0,0) {$\barA\otimes_l\barB$};
              \node(r) at (1.5,0) {$\K$};
              \draw[1cell squig] (l) to node{$J$} (m);
              \draw[1cell] (m) to node{$\Gamma$} (r);
          \end{tikzpicture}$, where $J$ be the composite $\sigma_l(p \times p)$.
  \end{enumerate}
\end{proposition}
\begin{proof}
  The right adjoint of the inclusion $\TT{\B,\TT{\A,\K}}\rightarrow\LL{\B,\LL{\A,\K}}$ sends $\Gamma'$ to (i) and (ii).
  On the other hand, (iii) is where $\Gamma J$ is mapped to by
  the right adjoint of the inclusion $\TT{\B,\TT{\A,\K}}\rightarrow\LL{\B,\LL{\A,\K}} \rightarrow \LL{\A\times\B, \K}$.
  To show that (iii) coincides with (i) and (ii), it suffices to show that $\LL{\B,\LL{\A,\K}} \rightarrow \LL{\A\times\B, \K}$ is fully-faithful.
  This inclusion is isomorphic to the pre-composition 2-functor $J^*\colon \TL{\barA\otimes \barB,\K} \rightarrow \LL{\A\times \B,\K}$.
  Since $J$ is the composition $\sigma_l(p\times p)$,
  $J$ is bijective-on-objects, has a retract $(q\times q)\rho$, and its image of 1-cells generates all 1-cells in $\barA\otimes \barB$.
  Therefore, from \Cref{lem:pre-compose_laxfn_bo_retract_gene}, $J^*$ is fully-faithful.
\end{proof}

\begin{remark}
  Let $\Gamma\colon\barA\otimes_l\barA\rightarrow \K$ be a generalized distributive law
  and $F\colon \A\rightsquigarrow\K$ be the composition via $\Gamma$.
  Suppose that $\widehat{\Gamma}\colon\A\times\A\rightarrow \K$ is the 2-functor constructed in \Cref{prop:AB_EM_on_EM}.
  Then, we have two 2-functors $\A\Rightarrow\K$: one is the coclassifier of $F$, and the other is $\widehat{\Gamma}\Delta$.
  The former is $\R(\Gamma J\Delta)$, and the latter is $\R(\Gamma J)\Delta$, so these do not need to coincide.
  When $\A = \One$, since $\Delta\colon\One\rightarrow\One\times\One$ was an isomorphism of 2-categories, these were the same.
\end{remark}



\printbibliography

\end{document}